\documentclass[12pt,a4paper,reqno]{amsart}
\usepackage{amsmath,amssymb,graphics,epsfig,color,enumerate,mathrsfs}
 \usepackage{tikz}
\usepackage{hyperref}

\usepackage[T1]{fontenc}
\usepackage[utf8]{inputenc}
\usepackage{dsfont}  \usepackage{verbatim}
\definecolor{refkey}{gray}{.75}
\definecolor{labelkey}{gray}{.5}

\newtheorem{Theorem}{Theorem}[section]
\newtheorem{Fact}{Fact}

\newtheorem{Lemma}[Theorem]{Lemma}
\newtheorem{Proposition}[Theorem]{Proposition}
\newtheorem{Corollary}[Theorem]{Corollary}
\newtheorem{Remark}[Theorem]{Remark}

\newtheorem{Definition}[Theorem]{Definition}

 \definecolor{darkgreen}{rgb}{0,0.4,0.5}

\newcommand{\cA}{\ensuremath{\mathcal A}}
\newcommand{\cB}{\ensuremath{\mathcal B}}

\newcommand{\cE}{\ensuremath{\mathcal E}}
\newcommand{\cF}{\ensuremath{\mathcal F}}
\newcommand{\cG}{\ensuremath{\mathcal G}}

\newcommand{\cK}{\ensuremath{\mathcal K}}
\newcommand{\cL}{\ensuremath{\mathcal L}}

\newcommand{\cN}{\ensuremath{\mathcal N}}

\newcommand{\cP}{\ensuremath{\mathcal P}}

\newcommand{\cR}{\ensuremath{\mathcal R}}

\newcommand{\cT}{\ensuremath{\mathcal T}}

\newcommand{\cW}{\ensuremath{\mathcal W}}
\newcommand{\cX}{\ensuremath{\mathcal X}}

\newcommand{\cZ}{\ensuremath{\mathcal Z}}

\newcommand{\leb}{\cL}

\newcommand{\bbC}{{\ensuremath{\mathbb C}} }

\newcommand{\bbE}{{\ensuremath{\mathbb E}} }

\newcommand{\bbG}{{\ensuremath{\mathbb G}} }

\newcommand{\bbL}{{\ensuremath{\mathbb L}} }

\newcommand{\bbN}{{\ensuremath{\mathbb N}} }

\newcommand{\bbP}{{\ensuremath{\mathbb P}} }
\newcommand{\bbQ}{{\ensuremath{\mathbb Q}} }
\newcommand{\bbR}{{\ensuremath{\mathbb R}} }

\newcommand{\bbZ}{{\ensuremath{\mathbb Z}} }
 \let\so=\o

\let\a=\alpha \let\b=\beta   \let\d=\delta  \let\e=\varepsilon
 \let\g=\gamma \let\h=\eta    \let\k=\kappa  \let\l=\lambda
      \let\o=\omega      
  \let\s=\sigma \let\t=\tau   
 \let\x=\xi \let\z=\zeta
\let\D=\Delta   \let\G=\Gamma  \let\L=\Lambda 
\let\O=\Omega      

\newcommand{\da}{\downarrow}

\newcommand{\be}{\begin{equation}}
\newcommand{\en}{\end{equation}}
\newcommand{\bee}{\begin{multline}}
\newcommand{\ene}{\end{multline}}

\newcommand{\rmdt}{{\rm DT}}
\newcommand{\rosso}{\textcolor{black}} %red

\newcommand{\rmE}{{\rm E}}
 \newcommand{\rmVor}{{\rm Vor}}

\DeclareMathOperator*{\argmin}{arg\,min}
\author[A.~Faggionato]{Alessandra Faggionato}
\address{Alessandra Faggionato.
  Dipartimento di Matematica, Universit\`a di Roma `La Sapienza'
  P.le Aldo Moro 2, 00185 Roma, Italy}
\email{faggiona@mat.uniroma1.it}

\author[C.~Tagliaferri]{Cristina Tagliaferri}
\address{Cristina Tagliaferri.
  Dipartimento di Matematica, Universit\`a di Roma `La Sapienza'
  P.le Aldo Moro 2, 00185 Roma, Italy}
\email{cristina.tagliaferri@gmail.com}

\title[]{Moment bounds and exclusion processes  on random  Delaunay   triangulations with  conductances}
\begin{document}

\bigskip

\begin{abstract} 
We consider the Voronoi tessellation associated to a stationary  simple point process  on $\mathbb{R}^d$ with finite and positive intensity. We introduce the Delaunay triangulation  as its dual graph, i.e. the graph with vertex set given by   the point process and with edges between vertices whose Voronoi cells share a $(d-1)$--dimensional face. We also attach to each edge a random weight, called conductance.
We provide sufficient conditions ensuring the integrability w.r.t. the Palm distribution of several quantities as  weighted degrees and associated moments. These integrability properties are crucial in applications, as they allow  to apply existing results  on  random walks, resistor networks  and the symmetric simple exclusion processes  with random conductances (cf.~\cite{AFST,F_sep,F_hom,F_muratore,F_resistor,F_in_prep}).
For the latter, while the moment bounds ensure  its  well definiteness and several properties, the same does not hold 
when the jump rates are not symmetric, i.e. for a generic simple exclusion process. In this case, by using a criterion from \cite{F_muratore}, we recover construction and properties  of the simple exclusion process
 whenever the simple point process has finite range of dependence  and the conductances are uniformly upper bounded. This last result relies on a suitable analysis of Bernoulli bond percolation on the Delaunay triangulation inspired by \cite{BB,BBQ}.
All the above  results remain valid if the simple point process is stationary with respect to integer translations.  

\smallskip


\smallskip


\end{abstract}

\maketitle

\hspace{5.9cm} \emph{This paper is dedicated to Claudio Landim}

\hspace{5.9cm} \emph{on the occasion of his 60th birthday.}

%\today

%\centerline{\bf Preliminary version}
 
\section{Introduction}

Random structures generated by point processes in $\mathbb{R}^d$ provide a natural and versatile framework for modeling disordered media and spatially inhomogeneous systems. Among these, the Delaunay triangulation associated with a simple point process (i.e.~a locally finite configuration of points)  plays a particularly important role, as it encodes the geometric adjacency relations induced by the corresponding Voronoi tessellation. When random weights (or \emph{conductances}) are assigned to the edges of the Delaunay triangulation, one obtains a weighted random graph that serves as a fundamental model for studying stochastic processes in random environments.

The aim of this paper is to investigate quantitative properties of such weighted Delaunay graphs, with a particular focus on moment bounds and integrability conditions for key functionals of the system. These functionals include, for instance, weighted degrees and sums of conductances around a typical point, as well as higher-order quantities involving spatial displacements. Establishing suitable integrability properties for some of these objects w.r.t. the Palm distribution (as $\sum_{x:x\sim 0} |x|^\z$,  $\l_0$, $\l_2$, $\big( {\rm deg}_{{\rm DT}(\xi)} (0) \big)^p$, $\mu^\o(0)^p$, $\nu^\o(0)^p$ introduced in Section \ref{sec_results1})  is a crucial step in the analysis of several probabilistic models on the Delaunay triangulation, including random walks, electrical resistor networks, and interacting particle systems such as the symmetric simple exclusion process, all  among random conductances  (cf.~\cite{AFST,F_sep,F_hom,F_muratore,F_resistor,F_in_prep}).

A central difficulty in this setting stems from the interplay between geometric randomness and probabilistic dependence. The structure of the Delaunay triangulation is highly sensitive to the configuration of points, and even local modifications of the point process may induce long-range changes in the graph. This makes it challenging to control quantities such as the degree of a vertex or the spatial range of its neighbors. To overcome this difficulties, inspired by \cite{Ro,Zu} we develop an estimate  procedure based on the so-called \emph{fundamental region}. For previous moment bounds we refer to \cite[Section~11]{Ro} (where the assumptions are much stronger than ours) and \cite{Zu} for the Poisson point process.

Our analysis is carried out in the general setting of simple point processes that are stationary and ergodic, with finite and positive intensity. We consider a broad class of models, ranging from processes with finite range of dependence to more general situations where correlations may decay slowly or are not explicitly controlled. In addition, we treat point processes exhibiting positive association, which allows us to exploit monotonicity properties in the derivation of moment bounds. To give a flavor of applications, in Section \ref{sec_esempi} we discuss  Poisson point processes, determinantal point processes, and Gibbsian point processes arising in statistical mechanics.

In addition to moment estimates, we also investigate Bernoulli bond percolation on the 
 Delaunay triangulation. In particular, we provide conditions under which the resulting random graph has only finite connected components, a property that is relevant for the graphical construction of exclusion processes with possibly non-symmetric jump rates and for the derivation of some fundamental properties (cf. \cite{F_muratore}). This analysis is restricted to the case  of bounded conductances and simple point processes with  finite-range dependence.
 To perform this analysis we use the ``$\bbZ^d$--process" approach developed in \cite{BB,BBQ} for Poisson point processes and we  extend it to the case of finite range of dependence.
 \smallskip
 
 \noindent
  {\bf Outline of the paper}. In  Section \ref{sec_model} we introduce the Voronoi tessellation, the  Delaunay triangulation and the conductance field and we introduce some fundamental assumptions. In Section \ref{sec_palm} we recall the definition of Palm distribution  and some of its properties. In Section \ref{sec_results1} we discuss the fundamental region and present our main results concerning the integrability w.r.t. the Palm distribution of 
  $\sum_{x:x\sim 0} |x|^\z$,  $\l_0$, $\l_2$, $\big( {\rm deg}_{{\rm DT}(\xi)} (0) \big)^p$, $\mu^\o(0)^p$, $\nu^\o(0)^p$. This integrability  is required when    applying some  existing results on random walks, resistor networks and SSEPs. In Section~\ref{sec_intermezzo} we illustrate this application in the case of the SSEP. In Section~\ref{sec_results2} we give a criterion assuring the construction and several properties of the SEP on the Delaunay triangulation with conductances.  Specific examples are discussed in Section~\ref{sec_esempi}.  The remaining sections and the appendix are devoted to proofs.

\section{Model}\label{sec_model}

\subsection{Basic notation} 
 We start by fixing some basic notation. We denote by  $\bbN$  the set of non-negative integers,
  %We denote by  $\bbN$ and $\bbN_{>0}$ the set of non-negative and positive integers,  respectively, 
  while we write $\bbR_+$ for $[0,+\infty)$. The vectors $e_1, \dots, e_d$ form  the canonical basis of $\bbR^d$,  $\leb(A)$   is the Lebesgue measure of the Borel  set $A\subset \bbR^d$, $a\cdot b$ is  the  scalar product of $a,b\in \bbR^d$, $|a|$ and $|a|_\infty$ are respectively   the Euclidean and the uniform norm of $a\in \bbR^d$. Given $r>0$ and $x\in \bbR^d$,  the ball $B_r(x)
$ and the boxes $\L_r(x)$ and $\L_r$ are defined as 
\be \label{basilisco}
B_r(x)=\{y\in \bbR^d\,:\, |y-x|\leq r\}\,,\qquad \L_r:=[-r, r]^d\,,\qquad  \L_r(x):=x+\L_r\,.
\en
 
  Given a topological space $\cX$,  without further mention, $\cX$ will be  thought of as  a measurable space endowed with the $\s$--algebra  $\cB(\cX) $   of its Borel subsets. 
  
We  denote by $\mathds{1}_A$ or $\mathds{1}(A)$ the indicator function  of the event $A$.
If ${\rm E}[\,\cdot\,]$ is the expectation w.r.t. some probability measure, we use the notation ${\rm E}[X,A]$ for  ${\rm E}[X \mathds{1}_A]$, where $X$ is a random variable and $A$ is an event.

  \subsection{Space $\cN$  of locally finite subsets of $\bbR^d$}
We denote by $\cN$ the set of locally finite subsets  of $\bbR^d$ and write $\xi$ for a generic element of $\cN$.  

An element of $\xi\in\cN$ is naturally identified with the simple counting measure $\sum_{x\in \xi } \d_x$,  which we  again denote by $\xi$ with a slight abuse of notation.
In particular,  $\xi (A)$ and $\int _{\bbR^d} d \xi (x) f(x)$ correspond respectively  to $\sharp (\xi\cap A)$ and $\sum_{x\in \xi }f(x)$ for all $A\subset \bbR^d$ and $f:\bbR^d \to \bbR$. 
This identification also allows one to endow 
  $\cN$ with the standard metric used for locally finite measures  on $\bbR^d$ (see \cite[Eq.~(A2.6.1),\;App.~A2.6]{DV}). Its precise definition will not be used  in what follows. For completeness, we recall that the convergence  $\xi_n\to \xi$  in $\cN$ as $n\to +\infty$  is equivalent to each of the following two properties (see  \cite[Proposition~A2.6.II]{DV}):
\begin{itemize}
 \item $\xi_n(A)\to \xi (A)$ as $n\to+\infty$ for all $A\in \cB(\bbR^d)$ with $\xi(\partial A)=0$;
 \item $\xi_n\to \xi$ vaguely as locally finite measures, i.e.  $\int _{\bbR^d} d \xi _n(x) f(x)\to \int _{\bbR^d} d \xi (x) f(x)$  as $n\to+\infty$ for all $f\in C_c(\bbR^d)$.
\end{itemize}

It can be shown that   the Borel $\s$--algebra $\cB(\cN)$ is  generated by the sets $  \{ \xi (A) =n\} $ with $A\in \cB(\bbR^d)$ and $n \in \bbN$ (see \cite[Proposition~7.1.III and  Corollary~7.1.VI]{DV}).

The group $\bbR^d$ acts on the space $\cN$ by  translations  $(\t_x)_{x\in \bbR^d}$, where 
\be\label{caldo_stufa}\t_x \xi := \xi -x \qquad \xi \in \cN\,,\;x\in \bbR^d\,.
\en
Here $\xi$ is regarded as a  subset of $\bbR^d$. The choice of the sign is purely conventional. Its advantage in the present context is that, given $x\in \xi$,   $\t_x \xi$ represents the new configuration of points observed by someone located  at   $x$, with Cartesian axes obtained by translating the original ones so that the observer’s position becomes the new origin\footnote{By the above identification of  $\xi$ with an atomic  measure,  definition \eqref{caldo_stufa} guarantees that $\t_x\xi(A)=\xi(A+x)$ for all $x\in \bbR^d$ and $A\in \cB(\bbR^d)$.}.

Finally, recall that a \emph{simple point process} is a measurable map  from a probability space to  the measure space $\bigl(\cN, \cB(\cN)\bigr)$. 

\subsection{Voronoi tessellations and Delaunay triangulations}
 Given $\xi\in \cN$ and $x\in \xi$, the Voronoi cell with nucleus $x$ is given by
\[\rmVor (x|\xi) =\{ y \in \bbR^d\,:\,|y-x| \leq |y-z| \; \forall z \in \xi\}\,.
\]
The Voronoi tessellation associated with $\xi$ is the collection of the Voronoi cells $\rmVor(x|\xi)$, $x\in \xi$ (see Figure~\ref{fig1}).
Since it is  a countable intersection of closed half-spaces, each Voronoi cell is closed and convex.
Moreover, if $x\not =y$ in $\xi$,  then the cells  ${\rm Vor}(x|\xi)$ and ${\rm Vor}(y|\xi)$ either share a $(d-1)$-dimensional face or they do not share any $(d-1)$--dimensional region.

\begin{Definition}[Delaunay triangulation  ${\rm DT}(\xi)$] \label{def_DT} Given   $\xi \in \cN$ we define the \emph{Delaunay triangulation} associated to $\xi$ as 
the graph ${\rm DT}(\xi)$ with vertices the points of   $\xi$ and edges  given by the  pairs  $\{x,y\}$ with $x\not = y$ in $\xi$ such that $\rmVor(x|\xi)$ and $\rmVor(y|\xi)$ share a $(d-1)$--dimensional face (see Figure~\ref{fig1}). 
We write $\cE_\rmdt (\xi)$ for the edges of $\rmdt(\xi)$. 
Given $x,y\in \xi$, we write $x\sim y$ (understanding the dependence on $\xi$)  whenever  $\{x,y\} \in \cE_\rmdt (\xi)$.\end{Definition}
%%%%%%%%%%%%%%%%%%%
\begin{figure}
\includegraphics[scale=0.3]{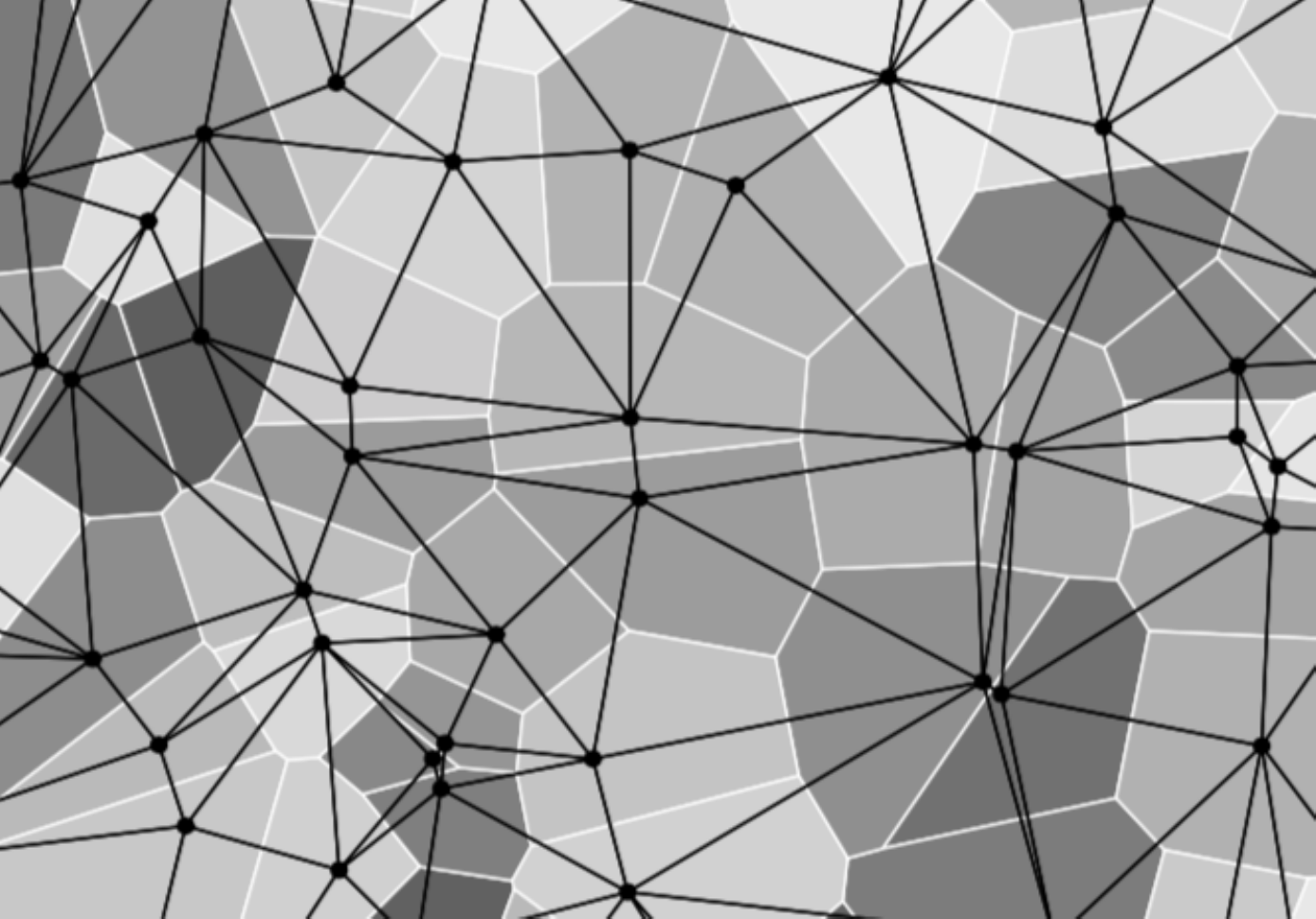}
\caption{ Voronoi tessellation and Delaunay triangulation associated to the set $\xi$, whose points are given by  the black dots ($d=2$). }\label{fig1}
\end{figure}

We stress that in computational geometry  the notion of Delaunay triangulation is slightly different and can refer to a complex,    whose 0-skeleton and 1-skeleton correspond to  the above ${\rm DT}(\xi)$. See e.g. \cite[Section~10.3]{G}. For us the Delaunay triangulation is just the graph dual to the Voronoi tessellation.  We also point out that, although
in Figure \ref{fig1} the edges of the Delaunay triangulation in $\bbR^2$ bound triangles, this is not always true for $d=2$ (take for example $\xi=\bbZ^2$). In general, 
%  \club \verde{In the appendix we will come back to this issue.} 
the 
Delaunay triangulation in $\bbR^d$ is the $1$--skeleton of a tessellation made  by $d$-simplices when 
points of $\xi$ are in \emph{general  quadratic position} (cf.~\cite[Chapter~10]{G} and \cite{M} for definitions and results). 
 We refer to \cite{Ze} for sufficient conditions on simple point processes  ensuring that  a.s. points of $\xi$ are  in  general quadratic position (an example is given by the homogeneous Poisson point process).

We recall  (cf.~\cite[Definition~1.5]{HW}) that a 
\emph{polyhedron} in $\bbR^d$  is defined as the  intersection of finitely many   closed half-spaces, while 
  a \emph{convex polytope} in $\bbR^d$ is defined as   the convex hull of finitely many points in $\bbR^d$. It is known that  bounded polyhedra coincide with  convex polytopes (cf.~\cite[Theorems~1.20 and 1.22]{HW}).  As shown  below, for our simple point processes a.s. the Voronoi cells will be  convex polytopes. Therefore it is natural to introduce the following subset $\cN_{\rm pol}\subset \cN$:

%%%%%%%%%%%%%%%%%%%
\begin{Definition}[Set $\cN_{\rm pol}$] \label{def_Npol}
We denote by $\cN_{\rm pol}$ the family of  $\xi \in \cN$ such that  the Voronoi cell $\rmVor(x|\xi)$ is a convex polytope (equivalently, a bounded polyhedron) for all $x\in \xi$. 
\end{Definition}

 One can  prove that  $\cN_{\rm pol}$ belongs to $\cB(\cN)$. Note that, according to the above definition, $\emptyset\in \cN_{\rm pol}$.
In Appendix \ref{app_criterio} we prove the following criterion ensuring that $\xi \in \cN_{\rm pol}$, where  the orthant $Q_\s$ is defined as follows: 
\[
Q_\s:=\{x\in \bbR^d\,:\, \s_i x_i >0 \text{ for } 1\leq i \leq d \}\,, \qquad \s\in \{-1,+1\}^d\,.\]
%%%%%%%%%%%%%%%%%%%
\begin{Lemma}\label{criterio_poliedro}
 If $\xi \in \cN$  satisfies $\xi \cap \left( x + Q_\s\right) \not = \emptyset$    for all $x \in \bbZ^d$ and  all $\s\in \{-1,+1\}^d$, then $\xi \in \cN_{\rm pol}$.
\end{Lemma}
%%%%%%%%%%%%%%%%%%%%%%%%%%%%%%%%%%%%%%%%%%%%%%%%%%%%%%%%%%%%%%%%%%%%%%%%%%%%%%%%%%%%%%%%%%%%%%%%%%%%%%%%%%%%%%%%%%%%%%%%%%%%%%%%%%%%%%%%%%%%%%%%%%%%%%%%%%%%%%%%%%%%%%%%%%%%%%%%%%%%%%%%%%%%%%%%%%%%%%%%%%%%%%%%%%%%%%%%%%%%%%%
\subsection{Delaunay triangulation in a random environment with a conductance field}
We consider a probability space $(\O, \cF, \cP)$ and denote by ${\rm E}[\cdot]$ the expectation w.r.t. $\cP$.  A generic configuration $\o\in  \O$  has to be thought of as describing  the \emph{environment} of the system under investigation. 

 We assume that the additive group $\bbR^d$ (endowed with the Euclidean metric) acts  on 
the probability space. The action 
is given by 
 a family of  maps  $(\theta_x)_{x\in \bbR^d}$ with $\theta_x: \O\to \O$ such that 
\begin{itemize}
\item[(P1)] $ \theta_0=\mathds{1}$, 
\item[(P2)] $ \theta _x \circ \theta _{y}= \theta_{x+y}$ for all $x,y\in \bbR^d$,
\item[(P3)] the map $\bbR^d\times \O \ni (x,\o) \mapsto \theta _x\o \in \O$ is measurable.
\end{itemize}
%Below we will denote by $\rmE[\cdot]$ the expectation w.r.t. $\cP$.

\begin{Definition}[Simple point process $\o\mapsto\hat\o$]
We fix a  simple point process $ \O\ni \o \to \hat \o \in \cN $    defined on the probability 
space $(\O, \cF, \cP)$.   In what follows we write $\bbP$ for its law  on $\cN$ and we write  $\bbE[\cdot]$ for the associated expectation. 
\end{Definition}

For later use we recall some basic terminology. 
 $\cP$ is called  \emph{stationary} if   $\cP(\theta_x A)=\cP(A)$ for all $x\in \bbR^d$ and $A\in \cF$. A set $A\subset \O$ is called \emph{translation invariant} if $\theta_x A= A$ for all $x\in \bbR^d$.  $\cP$ is  called \emph{ergodic} if 
  $\cP(A)\in \{0,1\}$ for any  translation invariant set $A\in \cF$. Strictly speaking, stationarity, translation invariance and ergodicity are w.r.t. the action $(\theta_x)_{x\in\bbR^d}$ on $\O$ (but the action will be understood in what follows).
  The same definitions hold when replacing $\O, \cF, \cP, \theta_x$ by $\cN, \cB(\cN), \bbP, \t_x$ respectively (recall \eqref{caldo_stufa}). 
 If $\bbP$ is stationary,  then the \emph{intensity} of the SPP (or of  $\bbP$)  is denoted by $m$:
\be\label{eq_intensa}
m := {\rm E}\bigl[ \hat \o  ( [0,1]^d)\bigr]= \bbE\bigl[ \xi  ( [0,1]^d)\bigr]\,.
\en
We recall the so-called  \emph{zero/infinity property} for stationary SPP (cf.~\cite[Proposition~10.1.IV]{DV}):
\be\label{zero_infinity}\bbP \text{ stationary } \;\; \Longrightarrow \;\;\bbP( \xi=\emptyset \text{ or } \sharp \xi=\infty)=1\,.
\en

As proved in Appendix~\ref{app_forza}, an important consequence of the stationarity of $\bbP$ is the following:
%%%%%%%%%%%%%%%%%%%%
%\begin{Lemma}\label{lemma_npot} If $\bbP$ is stationary and $\bbP(\xi=\emptyset)=0$, then $\bbP(\cN_{\rm pot})=1$. \end{Lemma}

\begin{Lemma}\label{lemma_npot} If $\bbP$ is stationary, then $\bbP(\cN_{\rm pol})=1$.
\end{Lemma}

%   \smallskip
  
%   As proved in Appendix \ref{app_basico} we have 
%   \begin{Lemma}\label{pimpa1}
%   The following holds:
%   \begin{itemize}
%   \item[(i)]
%  \verde{If $\bbP$ is stationary, then necessarily $m<+\infty$ and $\bbE[\xi(A)]=m \ell(A) $ for any $A\in \cB(\bbR^d)$.}
%  \item[(ii)] If  $\bbP$ is stationary and  ergodic, then  
% $1=\bbP(|\xi|=+\infty)=\cP( | \hat \o |=+\infty) $ and $1=\bbP( \cN_{\rm pol})=\cP(\hat \o \in \cN_{\rm pol})$.
% % (it is enough to apply  ergodicity and   the zero-infinity dichotomy in  \cite[Proposition 10.1.IV]{DV}).
%  \end{itemize}
% \end{Lemma}
% In all our applications $\bbP$ will be stationary and ergodic, hence for $\cP$--a.s. the Delaunay triangulation ${\rm DT}(\o)$ is well defined.

  \smallskip
  
 We now introduce the last ingredient of our model:
\begin{Definition}[Conductance field]
 We assume to have a measurable map  (called \emph{conductance field})
\[\O \times \bbR^d \times \bbR^d \ni (\o, x, y) \mapsto c_{x,y}(\o) \in [0,+\infty)
\]
such that $c_{x,y}(\o)=c_{y,x}(\o)$. % and $c_{x,y}(\o)>0$  whenever $x\sim y$ in ${\rm DT}(\hat\o)$.
\end{Definition}
As will be made clear from our applications,  $c_{x,y}(\o) $ will be relevant only for  $x\sim y $ in ${\rm DT}(\hat\o)$. Moreover, without loss of generality and to simplify some formulas below,   we convey that 
\[ c_{x,x}(\o)=0 \text{ and } c_{x,y}(\o) =0 \text{ if $\{x,y\}$ is not an edge of DT($\hat\o$)} \,.\]

\smallskip
 %%%%%%%%%%%%%%%%%%%%%%%%%%%%%%%%%%%%%%%%%%%%%%
 % In (A5) below $\cP_0$ denotes  the Palm distribution associated to $\cP$ and the point process $\O\ni \o \mapsto \hat \o\in \cN$ 
% As mentioned above $\cP$ is assumed to be  stationary. \rosso{We will also assume that, for $\cP$--a.a.~$\o$, it holds  $ \widehat{ \theta _x \o}= \t_x \hat{\o} $  for all $x\in \bbR^d$ (see \eqref{rel_cov_1} below)}. Hence the map $\cB(\bbR^d) \ni A\mapsto \int _{\O} d\cP (\o) \, \hat \o (A)\in [0,+\infty]$ is proportional to the Lebesgue measure. We will assume that the factor $m$ of proportionality (usually called \emph{intensity}) is 
% positive and finite, i.e.
%  \be\label{def_m}
%m := {\rm E}\bigl[ \hat \o  ( [0,1]^d)\bigr]\in (0,+\infty) \,.
%\en
%We recall that 
%the stationarity of $\cP$ and \eqref{def_m}    allow to introduce  the \emph{Palm distribution} $\cP_0$ (see \cite{F_hom} for further details):
% %%%%%%%%%%%%%%%%%%%%%%%%%%%%%%%%%%%%%%%%%%%%%%
%\begin{Definition}[Palm distribution $\cP_0$] \label{def_P0}The Palm  distribution $\cP_0$, associated to $\cP$ and the simple point process $\o \mapsto \hat\o$,
% is the  probability measure on $\O$ concentrated on $\O_0:= \{\o \in \O\,:\, 0 \in \hat{\o}\}$ such that 
% \be\label{astro}
% \cP_0(A)=\frac{1}{m} \int _\O d\cP(\o) \int_{[0,1]^d} d\hat\o (x)  \mathds{1}_A( \theta _x \o )  \qquad \forall A\in \cF\,. 
% \en
% The expectation associated to  $\cP_0$ is denoted by $\rmE_0$.
%\end{Definition}
%

For several results we will take  the following assumptions (they will be stated explicitly):

\smallskip

\noindent
{\bf Main Assumptions}:
\emph{The following holds:
\begin{itemize}
\item[(A1)] The probability measure $\cP$ is stationary and $\cP(\{\o\in\O:\hat\o = \emptyset\})=0$.
\item[(A2)]   The simple point process $ \O\ni \o  \mapsto \hat\o \in  \cN$ 
 has finite and positive intensity $m$ \emph{(}cf.~Eq.~\emph{\eqref{eq_intensa})}.
 \item[(A3)]  For some  translation invariant measurable set $\O_*\subset \O$ with $\cP(\O_*)=1$
 the following covariant relations are satisfied for all $\o\in \O_*$:
\begin{align}
& \widehat{ \theta _x \o}= \t_x \hat{\o}  \qquad \forall  x\in \bbR^d\,,\label{rel_cov_1}\\
& c_{y-x,z-x} (\theta_x \o)= c_{y, z } (\o ) \qquad \forall x \in \bbR^d\,, \; \forall y\sim z\in {\rm DT}(\hat \o)\,.\label{rel_cov_2}
\end{align}
%\item[(A4)] For all $\o\in \O_*$ it holds $c_{x,y}(\o)= c_{y,x}(\o)$ for all $x,y\in \hat \o$.
\end{itemize}
}

%\verde{\club Move this part: We point out that by (A1), $\bbP$ is stationary, and if $\cP$ is ergodic, then $\bbP$ is automatically  ergodic as well  (see Lemma \ref{pimpa2} in Appendix \ref{app_basico}).  In particular, (A1) and (A2) are implied by the stronger assumptions (A1') and (A2'):
%\begin{itemize}
%\item[(A1')] The probability measure $\cP$ is stationary and ergodic; 
%\item[(A2')]   The simple point process $ \O\ni \o  \mapsto \hat\o \in  \cN$ 
% has finite and positive intensity $m$ (cf.~Eq.~\emph{\eqref{eq_intensa}}).
% \end{itemize}
%The above (A1') and (A2')  appear e.g. in \cite{F_sep,F_hom,F_resistor}.}

Suppose for the moment just that $\cP$ is stationary and \eqref{rel_cov_1} holds for all $\o\in\O_*$ as in (A3). Then, by Lemma~\ref{pimpa2}--(i),  also   $\bbP$ is stationary. Moreover,      the event $\cA:=\{\o\in \O_*: \hat \o =\emptyset\}$ is translation invariant. Hence, excluding the trivial case $\cP(\cA)=1$,   at cost to decompose $\cP$  and condition on
 $\cA^c$,  we can always reduce to the case $\cP(\{\o\in \O: \hat \o=\emptyset\})=0$ as in (A1).  When  $\hat \o=\emptyset$ the Delaunay triangulation and all our statements become trivial, hence the request $\cP(\{\o\in\O:\hat\o = \emptyset\})=0$ in (A1) combined with (A3) is just to avoid trivialities, it is without loss of generality and  could be removed as well. We also point out that the above request implies that the intensity $m$ of the SPP  is not zero, but we have stated this explicitly in (A2) since it is relevant in the rest.

Under the above Main Assumptions, 
the expected number of points of the SPP in a generic set $B\in \cB(\bbR^d)$ equals $m \leb(B)$, where $\leb(B)$ denotes the Lebesgue measure of $B$. Moreover,  see \eqref{zero_infinity} and  Lemma~\ref{lemma_npot}, it holds 
% We also point out that, under our  main assumptions (A1), (A2), (A3),  it holds
% \be\label{tutto_npot} 
 \[ \cP\big(\{\o\in\O\,: \, \sharp \hat\o =+\infty\})=1 \; \text{ and } \;
\cP\big(\{\o\in\O\,:\,\hat\o\in\cN_{\rm pol}\}\big)=\bbP(\cN_{\rm pol})=1\,.
\] 
%\en
%Indeed, \eqref{tutto_npot} is a consequence of   Lemma~\ref{lemma_npot} if one shows that $\bbP(\xi=\emptyset)=0$. By ergodicity of $\bbP$ and translation invariance of the event 
%$\{\xi=\emptyset\}$, $\bbP(\xi=\emptyset)$ is zero or one but the latter has to be excluded since it would imply that $m=0$ against (A2).
%

 To give a more graphical interpretation of (A3),  let us assign to  $\o$ 
 the following weighted graph $\cG(\o)$:
 \begin{Definition}[Graph $\cG(\o)$] Given $\o\in \O$,  the weighted undirected graph $\cG(\o)$ is  defined as the Delaunay triangulation on $\hat\o$, where each edge  $\{x,y\}$ 
 is assigned the weight equal to its conductance
 $c_{x,y}(\o)$.
 \end{Definition}
 
Then the covariant relations \eqref{rel_cov_1} and \eqref{rel_cov_2} in Assumption (A3), together, are equivalent to   the identity of weighted graphs 
\[ \cG(\theta_x \o)= \cG(\o) -x \qquad \forall x\in \bbR^d\,,
\] where $\cG(\o)-x$  is the weighted graph obtained by translating $\cG(\o)$  in $\bbR^d$ along the vector $-x$ (in the translation, each edge keeps its own weight).

%
%The above assumptions will allow  to apply the results in \cite{F_hom,F_sep,F_resistor} as they appear also there (the symmetry assumption (A5) appears only in \cite{F_sep,F_resistor}).
%Moreover, $\hat \o $ is an infinite set  $\cP$--a.s. due to Lemma \ref{lemma_erg_stat} below, the positivity of the intensity $m$ in (A2) and  the zero-infinity dichotomy theorem (see [9,Proposition 12.1.VI]).
%As pointed out in \cite{F_hom}, Assumption (A3) is usually costless as one can add some randomness enlarging $\O$  to achieve (A3)
%(similarly to \cite[Remark 4.2-(i)]{demasi}).  We refer to \cite[Section~2.4]{F_hom} for an example. We recall that,    by  \cite[Theorem~4.13]{Br}, (A5) is fulfilled if $(\O_0,\cF_0,\cP_0)$ is a separable  measure space where $\cF_0:=\{A\cap \O_0\,:\, A\in \cF\}$ (i.e. there is a countable family $\cG\subset  \cF_0$ such that  the $\s$--algebra  $\cF_0$ is generated by $\cG$). For example, if $\O_0$ is a separable metric space and 
%$\cF_0= \cB(\O_0)$ (which is valid if $\O$ is a separable metric space and 
%$\cF= \cB(\O)$) then (cf. \cite[p.~98]{Br}) $(\O_0,\cF_0,\cP_0)$ is a separable  measure space  and (A5) is valid.
%

\begin{Remark}
We have considered the actions of the group $\bbG:=\bbR^d$ on the probability space $\O$ by $(\theta_g)_{g\in \bbG}$ and 
on the space $\cN$ by  $(\t_g)_{ g\in \bbG}$. As in  \cite{F_sep,F_hom,F_resistor} one can extend the analysis by considering also the action of the group $\bbG:=\bbZ^d$. In both cases $\bbG=\bbR^d$ and $\bbG=\bbZ^d$ one can   also deal with $\t_g$  of the form $\t_g \xi :=\xi- Vg$, where $V$ is a fixed $d\times d$ invertible matrix. These extensions are natural for example when dealing with simple point processes on crystal lattices (cf.~\cite[Section~5.2]{F_sep}, \cite[Section~5.6]{F_hom}).  All the results presented below can be restated in this  extended setting  with slight modifications in the proofs.
\end{Remark}

\section{Palm distributions $\cP_0$ and $\bbP_0$}\label{sec_palm}
In this section, we assume (A1),  (A2) and (A3). 
\begin{Definition}[Palm distribution $\cP_0$] \label{def_P0}The Palm  distribution $\cP_0$, associated to $\cP$ and the simple point process $\o \mapsto \hat\o$,
 is the  probability measure on $\O$ concentrated on $\O_0:= \{\o \in \O\,:\, 0 \in \hat{\o}\}$ such that 
 \be\label{astro}
 \cP_0(A)=\frac{1}{m} \int _\O d\cP(\o) \int_{[0,1]^d} d\hat\o (x)  \mathds{1}_A( \theta _x \o )  \qquad \forall A\in \cF\,. 
 \en
 The expectation associated to  $\cP_0$ is denoted by $\rmE_0$.
\end{Definition}

We recall that $\hat \o$ is naturally identified with the atomic measure $\sum _{x\in \hat\o}\d_x$, hence 
$\int_{[0,1]^d} d\hat\o (x) f(x)$ has to be thought of  as $\sum_{x\in \hat\o\cap [0,1]^d} f(x) $.
We refer the interested reader to \cite[Chapter~12]{DV}, \cite[Section~2 and Appendix~B]{F_hom} and references therein  for more details  on the Palm distribution $\cP_0$. Roughly, $\cP_0:=\cP(\cdot|\O_0)$. Since by stationarity of $\cP$ it can be proved that $\cP(\O_0)=0$,  the above identity is meaningless but it  can be  rigorously formalized  by means of regular conditional  probabilities as in \cite[Chapter~12]{DV}.

 Since by Lemma \ref{pimpa2}
the law $\bbP$ on $\cN$  of our simple  point process   is stationary  w.r.t.~the action $(\t_x)_{x\in \bbR^d}$ on $\cN$  and has finite positive  intensity $m=\bbE[ \xi([0,1]^d)]$,   also the Palm  distribution $\bbP_0$ associated to $\bbP$   is well defined (cf.~\cite[Chapter~12]{DV}). We recall that $\bbP_0$  is the  probability measure on $\cN$ concentrated on $\cN_0:= \{\xi \in \cN\,:\, 0 \in \xi\}$ such that 
 \be\label{campbell0}
 \bbP_0(A)=\frac{1}{m} \int _\cN d\bbP(\xi) \int_{[0,1]^d} d\xi (x)  \mathds{1}_A( \t_x \xi )  \qquad \forall A\in \cB( \cN)\,. 
 \en

 Given $A\in \cB(\cN)$ and setting $B:=\{\o \in \O\,:\, \hat \o \in A\}$,   for all $\o\in\O_*$ we have $\mathds{1}_A( \t_x \hat \o)= \mathds{1}_B (\theta_x \o)$  since $\widehat{\theta_x\o}= \t_x \hat \o$ by \eqref{rel_cov_1}. Therefore, by comparing \eqref{astro} with \eqref{campbell0}, we get $\bbP_0(A)=\cP_0(B)$, i.e.
 \[
 \bbP_0(A)= \cP_0(\hat \o \in A) \qquad \forall A\in \cB(\cN)\,.
 \]

 In what follows, we will use a generalization of \eqref{campbell0}  called Campbell's formula (cf. \cite[Eq.~(12.2.4), Theorem~12.2.II]{DV}): for each measurable function  $f:\bbR^d \times \cN_0 \to [0,+\infty)$   it holds
\be \label{campbell1}
\int _{\cN_0} d\bbP_0 (\xi)\int_{\bbR^d} d x f(x, \xi) = \frac{1}{m  }\int _{\cN} d\bbP(\xi) \int_{\bbR^d} d\xi(x) f (x, \t_x \xi) \,.
\en 
For completeness, although not used below, we recall that a similar formula holds for $\cP$ (see \cite[Appendix~B]{F_hom}).
 
 %%%%%%%%%%%%%%%%%%%%%%%%%%%%%%%%%%%%%%%%%%%%%%%%%%%%%%%%%%%%%%%%%%%%%%%%%%%%%%%%%%%%%%%%%%%%%%%%%%%%%%%%%%%%%%%%%%%%%%%%%%%%%%%%%%%%%%%%%%%%%%%%%%%%%%%%%%%%%%%%%%%%%%%%%%%%%%%%%%%%%%%%%%%%%%%%%%%%%%%%%%%%%%%%%%%%%%%%%%%%%%%%%%%%%%%%%%%%%%%%%%%%%%%%%%%%%%%%%%%%%%%%%%%%%%%%%%%%%%%%%%%%%%%%%%%%%%%%%%%%%%%%%%%%%%%%%%%%%%%%%%%%%%%%%%%%%%%%%%%%%%%%%%%%%%%%%%%%%%%%
 \section{Main results: %integrability of $\l_0,\l_1$ \rosso{w.r.t. $\bbP_0$}
 moment bounds w.r.t. $\cP_0$}\label{sec_results1}
 
In this section, we make   Assumptions  (A1), (A2) and (A3).
We recall that the Palm distribution $\cP_0$ has support in  $\O_0:= \{\o \in \O\,:\, 0 \in \hat{\o}\}$  and that ${\rm E}_0[\cdot]$ denotes the associated expectation (cf.~Definition~\ref{def_P0}). Moreover, ${\rm E}[\cdot]$ denotes the expectation w.r.t. $\cP$, while $\bbE[\cdot]$ denotes the expectation w.r.t. $\bbP$ (the latter is the image of $\cP$ by the map $\o\mapsto\hat\o$).

For $k=0,2$  we define the functions $\l_k :\O_0 \to [0,+\infty]$ as 
\be\label{def_lambda_k}
\l_k(\o) := \int_{\bbR^d} d \hat\o (x)  c_{0,x} (\o) |x|^k=\sum_{x:x\sim 0} c_{0,x}(\o) |x|^k\, .
\en
A fundamental condition  in order to apply the results in \cite{F_sep,F_hom,F_resistor} is  that $\l_0,\l_2\in L^1(\cP_0)$.  In this section we will provide sufficient conditions for the above integrability as well as for the integrability of $\sum_{x:x\sim 0} |x|^\z$ for $\z\geq 0$ (as intermediate step).

 As in \cite{AFST}, given $\o\in \O_0$, we define\footnote{Although $\mu^\o(0)=\l_0(\o)$, we prefer to introduce also the notation $\mu^\o(0)$ since closer to the one used in the moment bound conditions appearing in several papers about the random conductance model on $\bbZ^d$ or subgraphs.}
 \[
 \mu^\o(0):=\l_0(\o)= \sum_{x:x\sim 0} c_{0,x}(\o)\qquad \text{ and } \qquad \nu^\o(0):=\sum _{x:x\sim 0 } \frac{1}{c_{0,x}(\o)}\,.
 \]
Several results in  \cite{AFST} require  that   $\mu^\o(0)\in L^p(\cP_0)$ and  $\nu^\o(0) \in L^{p'}(\cP_0)$ for $p,p' \in [1,+\infty]$ satisfying a suitable inequality. In this section we will provide sufficient conditions assuring in general the moment bound $\mu^\o(0)\in L^p(\cP_0)$, or $\nu^\o(0) \in L^{p}(\cP_0)$, for $p \in [1,+\infty)$. As an intermediate step, we will also investigate the integrability of $\big({\rm deg}_{{\rm DT}(\hat \o)}(0)\big)^p$.
% We focus here on exponents in $[1,+\infty)$.

When providing sufficient conditions for  $\l_0,\l_2\in L^1(\cP_0)$, $\mu^{\o}(0)\in L^p(\cP_0)$ or $\nu^\o(0)\in L^p(\cP_0)$, we will treat two extreme situations concerning the decay of correlation of the SPP  $\o \mapsto \hat \o$:  one in which the SPP has a finite range of dependence, and a generic case in which no information on the decay of correlations is available.  We point out that the methods used in the proofs   could be further optimized for situations exhibiting an intermediate form of correlation decay. In addition to the above extreme two cases, we also treat the case in which the SPP has positive association.
To state our results, given $\g>0$, we define
\be\label{def_rho_k}
\rho_\g := {\rm E} \left[ \hat \o ([0,1]^d) ^\g\right] = \bbE\left[ \xi ([0,1]^d) ^\g\right] \,.
\en
We note that, by assumption (A2),  $\rho_1=m\in(0,+\infty)$.
\subsection{Fundamental region}\label{fund_reg}
Before presenting our results we discuss the so--called  fundamental region  ${\rm D}(x|\xi)$ associated to a vertex $x\in \xi $ of a Delaunay triangulation. This concept is crucial for the derivation of the moment bounds presented below.  

%To give a flavor of that, consider the case of unit conductances between nearest-neighbor vertices.  Then, given $\o\in \O_0$ we have  $\l_0(\o)= \deg

%and we give sufficient conditions to localize it within some ball. The fundamental region will then be used to estimate the degree  $\deg_{\rmdt(\xi)}(x)$ of the vertex $x$  in the Delaunay triangulation $\rmdt(\xi)$, as done in \cite{Zu} and then in \cite[Section 11]{Ro}, but with a different construction here. The results in this section are purely geometric (i.e. no probability appears) and  for a fixed configuration $\xi\in \cN_{\rm pol}$.
%%%%%%%%%%%%%%%%%%%%%%%%%%%%%%%%%%
\begin{Definition}[Fundamental region]  \label{def_fund_reg} Given  $\xi \in \cN_{\rm pol}$  and  $x\in \xi$, the \emph{fundamental region}  ${\rm D}(x|\xi)$ of $x$ is given by the union of the closed balls centered at $v$ and of radius $|v-x|$, where $v$ varies among the vertices of the cell  $\rmVor(x|\xi)$ (see Figure~\ref{fig_fund_reg}--(left)).
\end{Definition}
Since under our main assumptions (A1), (A2), and (A3) we have $\bbP(\cN_{\rm pol}) = 1$, we may, without loss of generality, restrict our attention to $\xi \in \cN_{\rm pol}$. We therefore do not attempt to formulate the results below under weaker geometric assumptions on $\xi$. For example, if $\xi \not\in \cN_{\rm pol}$ consists of two points, then the Voronoi cells of $\xi$ are two half-spaces and therefore have no vertices. On the other hand, there is no need for our purposes  to consider such degenerate situations.

%%%%%%%%%%%%%%%%%%%%%%%%%%%%%%%%%%
\begin{figure}[ht]
\includegraphics[scale=0.20]{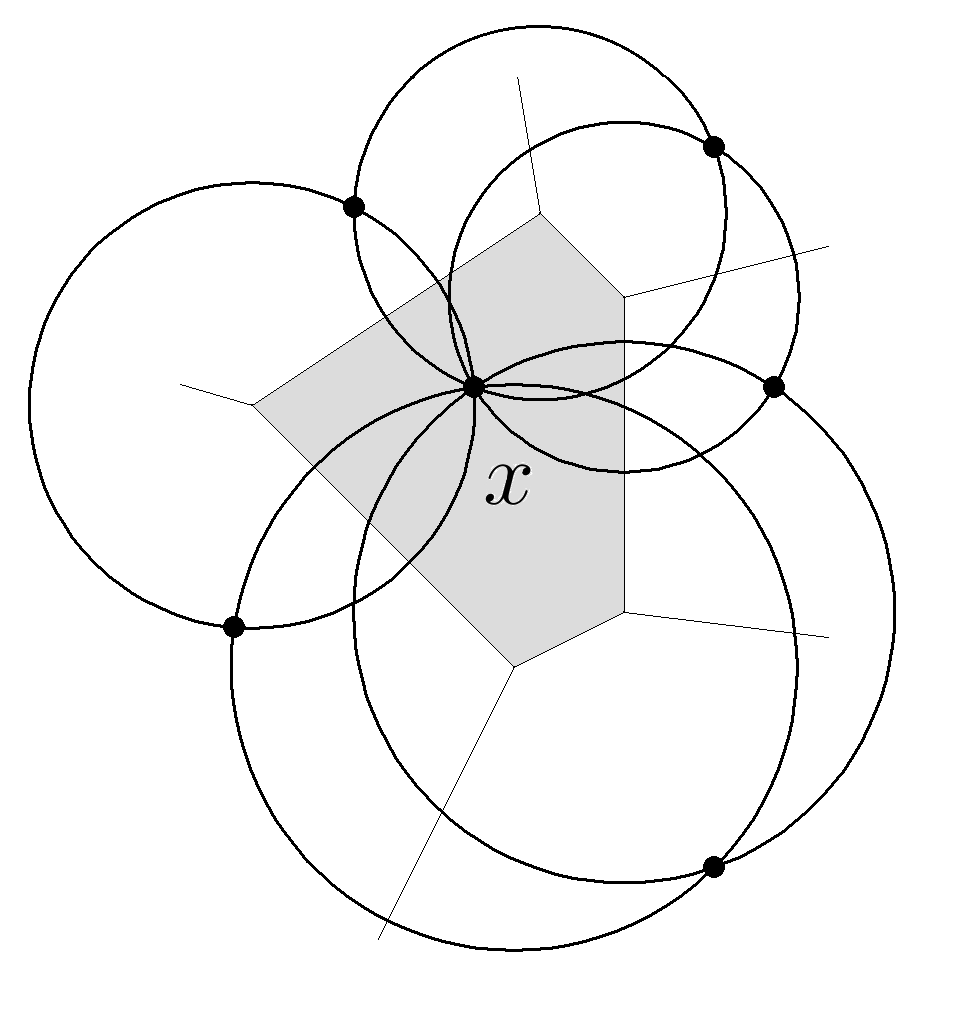} \qquad  \includegraphics[scale=0.30]{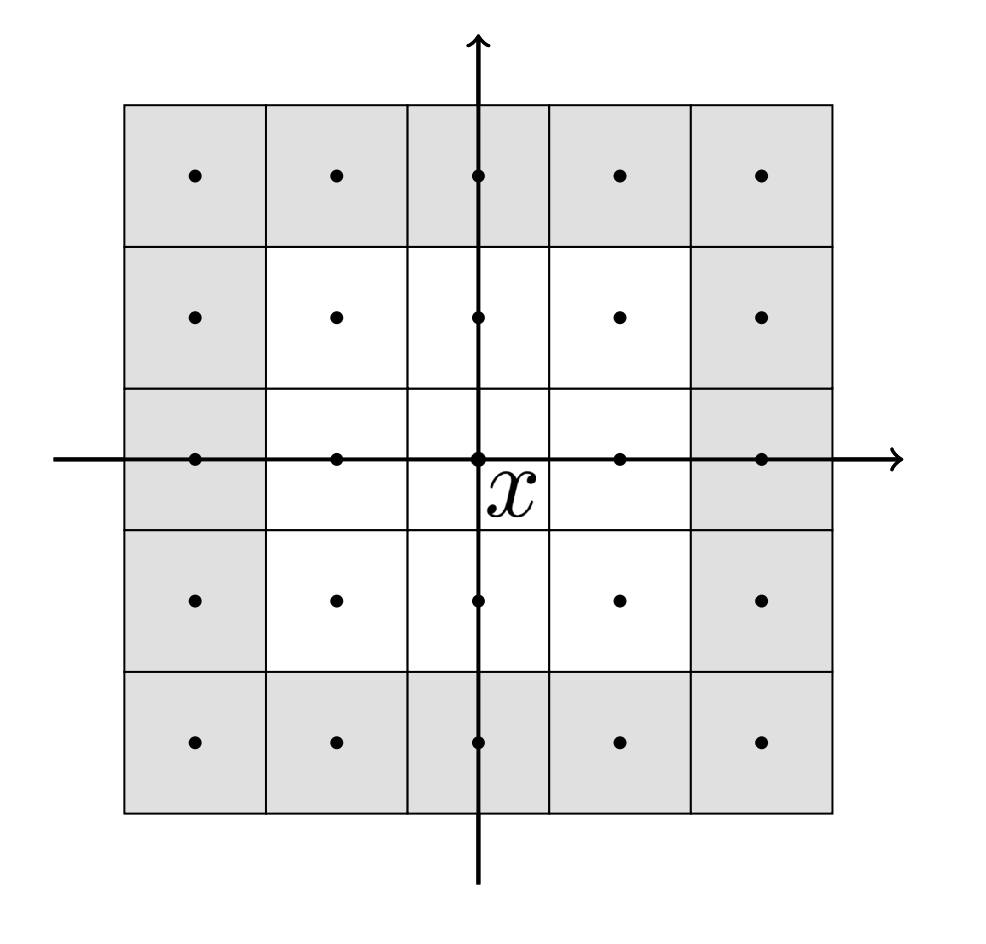}
\caption{(Left) Black dots are  points of $\xi$. ${\rm Vor}(x|\xi)$ is the grey region. $D(x|\xi)$  is the union of the balls. (Right)  In grey,  the shield of boxes $K_\ell(z)$ with $z\in I$.}\label{fig_fund_reg}
\end{figure}
%\verde{Nella figura sopra a dx fare bianca la grey box in 0}\\
  %%%%%%%%%%%%%%%%%%%%%%
  As stated in Lemma~\ref{lemma_FR_1} in Section~\ref{sec_fund_region}, 
  given  $\xi\in \cN_{\rm pol}$ and $ x\in \xi$,  all vertices adjacent to $x$ in ${\rm DT}(\xi)$ belong to $ D(x|\xi) $.
  We set  (see Figure \ref{fig_fund_reg}--(right))
\be\label{setI} I:= \{ z \in \bbZ^d\,:\, |z|_\infty =d\}
\en
and, given $z \in \bbR^d$ and $\ell>0$,\begin{equation}
    \label{cubetti}
    K_\ell(z):=z\ell+\big[-  \ell/2, \ell/2\bigr]^d=\L_{\ell/2}(z \ell)\,.
\end{equation}
Then, in Section \ref{sec_fund_region}  (cf.~Lemma~\ref{lemma_FR_1} and Lemma~\ref{star}) we will prove the following chain of implications, where  ${\rm deg}_{\rmdt (\xi)} (x)$ denotes the degree of $x$ in the graph $\rmdt (\xi)$:
\begin{multline}\label{armand_zero}
%\xi   \in \cN_{\rm pol}\cap \cN_0\cap \big\{ \xi (K_\ell(x))>0 \; \forall z\in I\big\}
\xi   \in \cN_{\rm pol}\cap \cN_0 \text{ and }  \xi \big(K_\ell(z)\big)>0 \text{ for all } z\in I
\\
 \Longrightarrow   \{y\,:\, y\sim 0\}\subset D(0|\xi) \subset B_{6\ell d^2}(0)\\
  \Longrightarrow \begin{cases}
  {\rm deg}_{\rmdt (\xi)} (0) \leq \xi\big ( B_{6\ell d^2}(0) \big)\\
  \max\{ |y|:y\sim 0\}\leq 6 \ell d^2
  \end{cases}
  .
\end{multline}
The above implications will be  the starting point for the derivation of the moment bounds presented below.

\subsection{Integrability of $ \sum_{x:x \sim 0}    |x|^\z$,  $\l_0$, $\l_2$  w.r.t.~$\cP_0$}\label{sec_sanremo2026}
To state our results, we  first recall two definitions concerning SPPs. 
%Below, fixed $\xi \in \cN$ and $A\subset \bbR^d$,  we denote by $\xi|_A$ the restriction of $\xi$ to the set $A$. 
\begin{Definition}[SPP with finite range of dependence] \label{def_fin_range}
A  probability measure ${\rm P}$ on $\cN$ has \emph{finite range of dependence} if there exists $L>0$ 
such that, for any  Borel sets   $A,B\subset \bbR^d$ having Euclidean distance  at least $L$, the random variables $\xi\cap A$ and $\xi \cap B$   defined on the probability space $(\cN, \cB(\cN), {\rm P})$ are independent. In this case, we say that the range of dependence is smaller than $L$.
\end{Definition}
%%%%%%%%%%%%%%%%%%%%%%%%%
\begin{Definition}[SPP with negative/positive association] A  probability measure ${\rm P}$ on $\cN$ has \emph{positive  association} if for any $n,k\geq 0$, any weakly increasing functions $f:\bbR_+^n \to \bbR$ and $g: \bbR_+^k\to \bbR$ and any family of pairwise disjoint bounded Borel sets  $A_1,A_2,\dots, A_n$, $B_1,B_2,\dots, B_k$ in $\bbR^d$,  it holds
\be\label{segno}{\rm Cov}\Big(\, f\big(\xi (A_1), \xi(A_2),\dots, \xi (A_n)\big)\,,\,g\big(\xi (B_1), \xi(B_2),\dots, \xi (B_k)\big)\,\Big)\geq 0\,,
\en
whenever the two random variables appearing in the above covariance has finite second moment w.r.t. ${\rm P}$.
The probability measure ${\rm P}$ is said to have  \emph{negative association} if the same holds with $\leq $ instead of $\geq $ in \eqref{segno}.  
\end{Definition}
%%%%%%%%%%%%%%%%%%%%%%%%%
\begin{Remark}\label{banff_arte}  
 ${\rm E}\Big[\prod_{i=1}^n  f_i(\xi(A_i )) \Big]
\leq \prod_{i=1}^n{\rm E}\Big[f_i(\xi(A_i )) \Big]$
whenever ${\rm P}$ has negative  association,   $A_1,\dots,A_n$ are pairwise disjoint bounded Borel sets in $\bbR^d$ and $f_1,\dots, f_n:\bbR_+\to \bbR_+$ are bounded   functions all weakly increasing or all weakly decreasing. Take for example $n=3$ and $f_i$ decreasing. Then $\bbE[f_1 f_2 f_3]=\bbE[ (-f_1 f_2)(- f_3)]\leq \bbE[ -f_1 f_2]\bbE[- f_3]= \bbE[ f_1 f_2]\bbE[ f_3]$ by \eqref{segno}, since $-f_1 f_2$ and $- f_3$ are weakly increasing (a further  iteration allows to conclude).
\end{Remark}

We now introduce a condition concerning the so--called \emph{void probabilities} relevant for our results. To this aim,
recall the definition of $B_r(x)$, $\L_r$ and $\L_r(x)$ given in \eqref{basilisco}. 
%As standard, we call  two exponents $p,q\in [1,+\infty]$ conjugate if $p^{-1}+q^{-1}=1$.
\begin{Definition}[Condition $C(\a)$] \label{def_Calpha}
Given  $\a > 0$
we say that Condition $C(\a)$ is satisfied if there exists $\k>0$ such that 
\be\label{cond_V}
 \bbP \left( \xi( \L_{\ell} ) =0\right) \leq \k \, \ell ^{-\a}  \qquad \forall \ell\geq 1\,,
\en
\end{Definition}

The following  two propositions can be readily  proved. For completeness we give their proofs in 
Appendixes~\ref{nerone1} and~\ref{nerone2}.
\begin{Proposition}\label{prop_FR} If  $\bbP$ has finite range of dependence or it has negative association, then for a suitable constant $c>0$ we have 
$ \bbP \left( \xi( \L_{\ell} ) =0\right) \leq e^{-c\ell^d } $ for all $\ell \geq 1$.
In particular, $\bbP$ satisfies Condition $C(\a)$ for any $\a>0$.
\end{Proposition}
%%%%%%%%%%%%%%%%%%%%%%%%%%%%%%  spostare la proof %%%%%%%%%%
We generalize the above result in the next proposition. Given $U\subset \bbR^d$ we denote by $\cF_U$ the $\s$--algebra of events determined  by the behavior of $\xi$ on $U$.
We set $\cT_a:= \cF_{\L_a^c}$.
%The above proof can be refined to lead to the following result:
\begin{Proposition}\label{prop_tail}
  Suppose that
  \be \label{mannaro}
  \lim_{k \to +\infty } \| \bbP( \xi (\L_{1/2})=0| \cT_k) - \bbP( \xi (\L_{1/2})=0)\|_\infty=0\,.
  \en
 Then for a suitable constant $c>0$ we have $ \bbP \left( \xi( \L_{\ell} ) =0\right) \leq e^{-c\ell^d } $ for all $\ell \geq 1$. In particular, $\bbP$ satisfies Condition $C(\a)$ for any $\a>0$.
\end{Proposition}
%%%%%%

We can finally state our main results.
\begin{Theorem}\label{primo_int} Assume (A1), (A2) and (A3).
% Let $\bbP$ be   stationary    with finite and positive intensity, such  that $\bbP(\xi\not=\emptyset)=1$.
 Then, given $\z\geq 0$, it holds  $\bbE_0\big[  \sum_{x \sim 0}    |x|^\z\big]<+\infty$ 
  if at least one of the following  hypotheses is satisfied:
 \begin{itemize}
 \item[(H1)]
   $\rho_\g<+\infty $ for some $\g>2$ and   Condition C$(\a)$ holds for some   $ \a$ with
   \be\label{siria99}
   \a> \frac{\g}{\g-2}  (d+\z) \,;
   \en
   % olds H1. $\rho_\g<+\infty $ for some $\g>2$ and   Condition C$(\a)$ holds for some   $ \a> (\z+d) q^2 $, where  $q$ is the exponent conjugate to $p:=\g-1$;
%%%%
\item[(H2)]  $\rho_2<+\infty$ and  $\bbP$ has finite range of dependence;% and Condition $C(\a)$ holds for some $\a> d+\z$.  
\item[(H3)] $\rho_2<+\infty$,  $\bbP$ has positive association and  Condition $C(\a)$ holds for some $\a> d+\z$.
  \end{itemize}
\end{Theorem}
\begin{proof}
The claim  is an immediate consequence of Proposition~\ref{cuore} and Corollary~\ref{cor_cuore}.
\end{proof}
%%%%%%%%%%%%%%%%%%%%%%%%%%%
%\verde{Sotto uso $\bbP$ stazionario  con intensita' positiva e finita}
\begin{Theorem}\label{teo1} In addition to  (A1), (A2), (A3), we suppose  that there exists a locally bounded measurable function $\phi:\bbR_+\to \bbR_+$ and  constants $C_0>0 $ and $\z\geq 0 $ such that  
\begin{itemize}
\item  $\phi (r) \leq C_0 r^{-2+\z}$  for $r\geq 1$;
\item 
for $\cP_0$--a.a. $\o$ it holds  
\be\label{condizione_zero}
{\rm E}_0 \bigl[ c_{0,x}  (\o)\,| \,\hat \o  \big]\leq \phi(|x|) \qquad \forall x: x \sim 0\,.
\en
\end{itemize}
Then $\l_0 ,\l_2 \in L^1(\cP_0)$ if at least one hypothesis between (H1), (H2) and (H3) in Theorem~\ref{primo_int} is satisfied.
% the following  hypotheses is satisfied:
% \begin{itemize}
% \item[(H1)]
%   $\rho_\g<+\infty $ for some $\g>2$ and   Condition C$(\a)$ holds for some   $ \a$ with
%   \be\label{siria99}
%   \a> \frac{\g}{\g-2}  (d+\z) \,;
%   \en
%\item[(H2)]  $\rho_2<+\infty$ and  $\bbP$ has finite range of dependence;\item[(H3)] $\rho_2<+\infty$,  $\bbP$ has positive association and  Condition $C(\a)$ holds for some $\a> d+\z$.
%  \end{itemize}
\end{Theorem}
Theorem \ref{teo1} will be  proved in Section \ref{sec_proof_teo1}.
By taking $\phi$ constant and $\z=2$ in the above theorem, we immediately get the following result which covers in particular the relevant case of unit conductances between nearest-neighbor vertices,  as well as the case of conductances that are i.i.d. conditionally on  $\cP (\cdot |\hat\o) $ with  distribution having finite mean (see Proposition~\ref{prop_suff} below):
 \begin{Corollary}\label{fornello}
 In addition to (A1), (A2), (A3), we  suppose that  for some $C_*>0$ and for $\cP_0$--a.a. $\o$ it holds  
\be\label{condizione_utto}
{\rm E}_0 \bigl[ c_{0,x}  (\o)\,| \,\hat \o  \big]\leq C_* \qquad \forall x: x\sim 0\,.
%\forall x \text{ such that  } \{0,x\}\in \cE_{\rm DT}(\hat\o) \,,
\en
Then $\l_0 ,\l_2 \in L^1(\cP_0)$
 if at least one of the following  hypotheses is satisfied:
 \begin{itemize}
\item[(H1')] $\rho_\g<+\infty $ for some $\g>2$ and   Condition C$(\a)$ holds
 for some 
 $ \a> \frac{\g}{\g-2}  (d+2)$;
%%%%    
\item[(H2')]  $\rho_2<+\infty$ and   $\bbP$ has finite range of dependence;
\item[(H3')] $\rho_2<+\infty$,  $\bbP$ has positive association and  Condition $C(\a)$ holds for some $\a> d+2$.
 \end{itemize}
\end{Corollary}
%%%%%%%%%%%%%%%%%%%%%%%%%%%

%\club\verde{negative association}

% In the above theorem, we have focused on the case that $\phi(r)$ is constant  or at most decay rather slowly.

\smallskip

If $\phi$ decays sufficiently fast, then  Theorem \ref{teo1} and its proof are  not optimal and one can give a much simpler proof of the integrability of $\l_0,\l_2$:
\begin{Proposition}\label{prop_veloce} In addition to  (A1), (A2), (A3), we suppose  that there  exists a locally bounded measurable  function $\phi:\bbR_+\to \bbR_+$ and  constants $C_0, \e>0 $ such that  
\begin{itemize}
\item  $\phi (r)  \leq C_0 r^{-d-2-\e}$  for $r\geq 1$;
\item 
for $\cP_0$--a.a. $\o$  it holds ${\rm E}_0 \bigl[ c_{0,x}  (\o)\,| \,\hat \o  \big]\leq \phi(|x|) $ for all $x: x \sim 0$.
%\be\label{condizione_zero}{\rm E}_0 \bigl[ c_{0,x}  (\o)\,| \,\hat \o  \big]\leq \rosso{\phi(|x|)} \qquad \forall x \text{ such that  } \{0,x\}\in \cE_{\rm DT}(\hat\o) \,m\en
\end{itemize}
Then, if $\rho_2<+\infty$, we have $\l_0 ,\l_2 \in L^1(\cP_0)$.
\end{Proposition}
Proposition \ref{prop_veloce} will be  proved  in   Section \ref{sec_proof_teo1}.

%We point out that condition \eqref{condizione_zero} is  satisfied whenever the conductances $c_{x,y}(\o)$ are 

\smallskip

In the next Proposition we  give a sufficient  condition  ensuring \eqref{condizione_utto} in a relevant application. 
For later use, we formulate the result in a more general form. For  \eqref{condizione_utto} it is enough to take $\varphi(t):=t$.

\begin{Proposition}\label{prop_suff} Let $\varphi:\bbR_+\to \bbR$ be a measurable function. Suppose that under $\cP(\cdot|\hat \o)$  the conductances $c_{x,y}(\o)$ with $\{x,y\}\in \cE _{\rm DT}(\hat \o)$    are independent  random variables with  
 distribution  $\nu_{|x-y|}$ parameterized by the distance $|x-y|$ and such that $\sup _{r>0}  \int _0^\infty d\nu_r(t)\varphi(t) \leq C_*$.  Then for $\cP_0$--a.a. $\o$ it holds  
\be\label{urlo}
{\rm E}_0 \bigl[ \varphi\big( c_{0,x}  (\o)\big)\,| \,\hat \o  \big]\leq C_* \qquad \forall x: x\sim 0\,.
\en
 \end{Proposition}
\begin{proof} 
%If the conductances $c_{x,y}(\o)$ are upper bounded by an absolute  constant $C_*$, then the conclusion is trivial.  Let us consider the other case (we deal with $\nu_r$, which corresponds to the  more general situation).
 Let $\bbP$ be the law of the SPP. Without loss, we can take as space $\O$  the set $\{(\xi, \Theta)\,:\, \xi \in \cN_{\rm pol}, \; \Theta \in (0,+\infty) ^{\cE_{\rm DT}(\xi)} \}$. We set  $\hat\o:=\xi$ and    $c_{x,y}(\o):= \Theta_{\{x,y\}}$ if $\o=(\xi,\Theta)$. Finally we 
 suppose that the probability measure  $\cP$ on $\O$   corresponds to  sampling $\xi$ with law $\bbP$, and afterwards assigning independently a conductance   to the edge $\{x,y\}\in \cE_{\rm DT}(\xi)$ with distribution $\nu_{|x-y|}$.  
 Then, by using Campbell's formula,  
one can easily derive  that   the Palm distribution $\cP_0$ corresponds to sampling $\xi$ with the Palm distribution  $\bbP_0$ associated to $\bbP$, and afterwards assigning independently a conductance   to the edge $\{x,y\}\in \cE_{\rm DT}(\xi)$ with distribution $\nu_{|x-y|}$. It then follows that $\max_{x\,:\, x\sim 0 } {\rm E}_0 \bigl[ \varphi\big( c_{0,x}  (\o)\big)\,| \,\hat \o  \big]\leq \sup _{r>0}  \int _0^\infty d\nu_r(t)\varphi(t) \leq C_*$.
\end{proof}

%%%%%%%%%%%%%%%%%%%%%%%%%%%%%%%%%%%%%%%%%%%%%%%%%%%%%%%%%%%%
%
%\newpage
%Recall:
%$\rho_\g := {\rm E} \left[ \hat \o ([0,1]^d) ^\g\right] = \bbE\left[ \xi ([0,1]^d) ^\g\right] $.
%Moreover Condition $C(\a)$ means:  there exists $\k>0$ such that 
%$ \bbP \left( \xi( \L_{\ell} ) =0\right) \leq \k \, \ell ^{-\a} $ $ \forall \ell >0$.
 \subsection{Integrability of
  $\big( {\rm deg}_{\rm DT(\hat\o)} (0)\big)^p$,  
  $\mu^\o(0)^p$ and $\nu^\o(0)^p$ w.r.t.~$\cP_0$ }\label{sec_results1.5}
Fixed $p\in [1,+\infty)$, below we will provide conditions ensuring that  ${\rm E}_0 \left[ {\rm deg}_{\rm DT(\hat\o)} (0)^p\right]<+\infty$. Let us now show how this moment bound is related to the integrability of  $\mu^\o(0)^p$ and $\nu^\o(0)^p$ w.r.t.~$\cP_0$.

Let $p_* $ be the exponent conjugate to $p$ (i.e. $p^{-1}+p_*^{-1}=1$).  By  H\"older's inequality, we have 
\[
\mu^\o(0) = \sum _{x: x\sim 0} c_{0,x}(\o)  \leq \Big(\sum _{x: x\sim 0} c_{0,x}(\o)^p\Big)^{\frac{1}{p}} \Big({\rm deg}_{\rm DT(\hat\o)} (0)\Big)^{\frac{1}{p_*}}\,.
\] 
Hence we have (also by conditioning w.r.t.~$\hat\o$)
\begin{equation}\label{pepsi2026}
\begin{split}
 {\rm E}_0 \big[ \mu^\o(0) ^p\big] & \leq {\rm E}_0 \Big[
 \Big(\sum _{x: x\sim 0} c_{0,x}(\o)^p\Big) {\rm deg}_{\rm DT(\hat\o)} (0)^{\frac{p}{p_*}}\Big]\\
&  \leq  {\rm E}_0 \Big[
 \Big(\sum _{x: x\sim 0} {\rm E}_0[c_{0,x}(\o)^p|\hat\o]\Big) {\rm deg}_{\rm DT(\hat\o)} (0)^{\frac{p}{p_*}}\Big]\,.
 \end{split}
\end{equation}
Suppose now  that for some constant $C<+\infty$ it holds that, for $\cP_0$--a.a.~$\o$,  
$ \sup_{x: x\sim 0} {\rm E}_0[c_{0,x}(\o)^p|\hat\o]\leq C $. Then, using that $1+p/p_*=p$, we get
 \[
 {\rm E}_0 \left[ \mu^\o(0)^p\right] \leq C {\rm E}_0 \left[ {\rm deg}_{\rm DT(\hat\o)} (0)^p\right]\,.
 \]
 Similarly we get
  that $  \sup_{x: x\sim 0} {\rm E}_0\left[
  1/c_{0,x}(\o)^p \,\big|\,\hat\o\right]\leq C $ implies that $ {\rm E}_0 \left[ \nu^\o(0)^p\right] \leq C {\rm E}_0 \left[ {\rm deg}_{\rm DT(\hat\o)} (0)^p\right]$. Note that in this case it must be $c_{0,x}(\o)>0$ for all $x\sim 0$.
  
 It remains to ensure that  ${\rm E}_0 \left[ {\rm deg}_{\rm DT(\hat\o)} (0)^p\right]=\bbE_0 \left[ {\rm deg}_{\rm DT(\xi)} (0)^p\right]<+\infty$. This task is carried out in Section~\ref{sec_aspettoV}.   Our final result is the following:
%\rosso{ \begin{Theorem}\label{cond_int}
% Let $\bbP$ be the law  of  a  stationary  simple point process  on $\bbR^d$  with finite and positive intensity, such  that $\bbP(\xi\not=\emptyset)=1$. Fix $p\in [1,\infty)$. Then   ${\rm E}_0 \left[ {\rm deg}_{\rm DT(\hat\o)} (0)^p\right]<+\infty$ is each of the following cases:
% \begin{itemize}
% \item[(C1)]
%   $\rho_\g<+\infty$  for some    $\g>p+1$ and Condition  $C(\a)$ holds for some 
% $\a >   \rosso{  \frac{d\g }{\g-1-p}   } $;
%    \item[(C2)] 
%      $\rho_{1+p}<+\infty $  and $\bbP$ has finite range of dependence;
%      \item[(C3)] $\rho_{1+p}<+\infty $, $\bbP$ has positive association and Condition   $C(\a)$ holds for some $\a> dp$.
%    \end{itemize}
%    \end{Theorem}}
% \begin{proof} The claim is an immediate consequence of Proposition~\ref{cuoreV} and Corollary~\ref{cor_cuoreV}.
% \end{proof}
 
  \begin{Theorem}\label{teo_mimmo} Fix $p\in [1,\infty)$. In addition to our main assumptions (A1), (A2), (A3), we assume
   that at least one  of the following cases holds:
 \begin{itemize}
 \item[(C1)]
   $\rho_\g<+\infty$  for some    $\g>p+1$ and Condition  $C(\a)$ holds for some 
 $\a >     \frac{dp \g }{\g-1-p}    $;
    \item[(C2)] 
      $\rho_{1+p}<+\infty $  and $\bbP$ has finite range of dependence;
      \item[(C3)] $\rho_{1+p}<+\infty $, $\bbP$ has positive association and Condition   $C(\a)$ holds for some $\a> dp$.
    \end{itemize}
Then   ${\rm E}_0 \left[ {\rm deg}_{\rm DT(\hat\o)} (0)^p\right]<+\infty$ and  the following implications hold:
\begin{itemize}
\item[(a)]  If   $\cP_0$--a.s. 
$ \sup_{x:x\sim 0} {\rm E}_0 \bigl [ c_{0,x}  (\o)^p\,| \,\hat \o  \big]\leq C $ 
for some   $ C \in [0,\infty)$, then ${\rm E}_0\left[\mu^\o(0)^p\right]<+\infty$.
\item[(b)]  If   $\cP_0$--a.s. 
$ \sup_{x: x\sim 0} {\rm E}_0\left[
  \frac{1}{c_{0,x}(\o)^p}\Large{|}\hat\o\right]\leq C$ 
for some   $ C \in [0,\infty)$, then ${\rm E}_0\left[\nu^\o(0)^p\right]<+\infty$.
\end{itemize}
\end{Theorem}
\begin{proof} Due to the discussion before the theorem, it is enough to prove that ${\rm E}_0 \left[ {\rm deg}_{\rm DT(\hat\o)} (0)^p\right]<+\infty$. This bound is an immediate consequence of  Proposition~\ref{cuoreV} and Corollary~\ref{cor_cuoreV}.
\end{proof}
Note that a sufficient condition implying the validity of the the hypotheses in the above implications (a) and (b) is provided by   Proposition~\ref{prop_suff} with $\varphi(t):= t^p$ and $\varphi(t):=t^{-p}$, respectively.

\subsubsection{Extensions} Also for ${\rm E}_0\left[\mu^\o(0)^p\right]$ and ${\rm E}_0\left[\nu^\o(0)^p\right]$ one could provide results similar to Theorem~\ref{teo1} and Proposition~\ref{prop_veloce} by using  the techniques developed in Sections~\ref{sec_aspetto} and \ref{sec_aspettoV}. We give some comments for the interested reader  assuming that,  for a measurable function $\phi$, for $\cP$--a.a.~$\o$ it holds $ {\rm E}_0 \bigl[ c_{0,x}  (\o)^p\,| \,\hat \o  \big]\leq \phi(|x|) $ for all $ x$ with $ x \sim 0$. 

 %As in Section \ref{sec_sanremo2026} one could consider the more general case that,  for a measurable function $\phi$, $ {\rm E}_0 \bigl[ c_{0,x}  (\o)^p\,| \,\hat \o  \big]\leq \phi(|x|) $ for all $ x$ with $ x \sim 0$. Then 
 From \eqref{pepsi2026} we get
\be\label{skipper}
 {\rm E}_0 \big[ \mu^\o(0) ^p\big] 
  \leq  {\rm E}_0 \Big[
 \Big(\sum _{x: x\sim 0} \phi(|x|)\Big) \Big({\rm deg}_{\rm DT(\hat\o)} (0)\Big)^{\frac{p}{p_*}}\Big]\,.
 \en
 Define $G(t):= \sup_{|x|\leq t} \phi(|x|)$. Then by the implication \eqref{armand} discussed in Section~\ref{sec_aspetto} and using the notation of Section~\ref{sec_aspetto},  from \eqref{skipper} we get 
 \be\label{ulla}
 {\rm E}_0 \big[ \mu^\o(0) ^p\big] 
  \leq  \sum _{n=0}^\infty G(6 d^2 \beta ^n) \bbE_0[   \xi(\G_n)^p,T_n\Big]\,.
  \en
  Above we have used that  $\xi(\G_n)^{1+\frac{p}{p_*}}= \xi(\G_n)^p$.
 Then one could proceed as in Section~\ref{sec_aspettoV} (compare \eqref{ulla} with \eqref{ottino})  to bound the series in the r.h.s. of \eqref{ulla},  obtaining a result in the  same spirit of 
  Theorem~\ref{teo1}.  This method works well when $\phi$ in not decreasing.
  
 If $\phi(r)$ is locally bounded and decreases fast, the above method is not the most efficient. In this case one could fix $t>1$ and by H\"older's inequality bound the r.h.s. of \eqref{skipper} by 
 \be\label{santissimo}
  {\rm E}_0 \Big[
 \Big(\sum _{x: x\sim 0} \phi(|x|)\Big)^t \Big]^\frac{1}{t} {\rm E}_0\Big[  \Big({\rm deg}_{\rm DT(\hat\o)} (0)\Big)^{\frac{p t_*}{p_*}}\Big]^\frac{1}{t_*}\,.
 \en
 Above $t_*$ is the exponent conjugate to $t$.
 To get that the second factor in the r.h.s. of \eqref{santissimo} is finite, one could use the first part of Theorem~\ref{teo_mimmo}, while to bound the first factor  in the r.h.s. of \eqref{santissimo} one could proceed  as in  the proof of Proposition~\ref{prop_veloce}.
 Suppose for example that, for some $\g>d$,   $\phi (r)  \leq C_0 r^{-\g }$  for $r\geq 1$ and $\phi$ is bounded on $(0,1]$.   By taking $t=2$ and by applying the Cauchy--Schwarz inequality,  we have 
\be
\begin{split} 
 {\rm E}_0 \Big[
 \Big(\sum _{x: x\sim 0} \phi(|x|)\Big)^t \Big]& \leq 
 \sum_{\substack{z\in \bbZ^d:\\ |z|_\infty \geq 1}}\sum_{\substack{z'\in \bbZ^d:\\ |z'|_\infty \geq 1}} |z|_\infty^{-\g}|z'|_\infty^{-\g}\bbE_0\big[\xi(\L_1(z))\xi(\L_1(z'))\big]
  \\
 & \leq \Big(
 \sup_{
    \substack{v\in \bbZ^d:\\ |v|_\infty \geq 1}
        } \bbE_0\big[\xi(\L_1(v))^2\big] \Big) \Big( \sum_{\substack{z\in \bbZ^d:\\ |z|_\infty \geq 1}} |z|_\infty ^{-\g}\Big)^2\,.
 \end{split}
 \en
 The above supremum is finite when $\rho_3<+\infty$ due to Lemma~\ref{prop1}, while the rightmost series is finite as  $\g>d$.
 
 \smallskip

 The same considerations presented above for ${\rm E}_0 \big[ \mu^\o(0) ^p\big] $ hold for   ${\rm E}_0 \big[ \nu^\o(0) ^p\big] $ when 
$ {\rm E}_0 \bigl[ 1/c_{0,x}  (\o)^p\,| \,\hat \o  \big]\leq \phi(|x|) $ for all $ x$ with $ x \sim 0$. 
  
% \verde{ \club Negative association}
 %%%%%%%%%%%%%%%%%%%%%%%%%%%%%%%%%%%%

%%%%%%%%%%%%%%%%%%%%%%%%%%%%%%%
\section{Intermezzo: An application of the integrability of $\l_0,\l_2\in L^1(\cP_0)$ to the SSEP}
\label{sec_intermezzo}
%\verde{I am not sure that it is valuable to keep this section since in \cite{F_in_prep} there is a full review of the results in \cite{F_sep,F_hom,F_muratore,F_resistor} with updated weaker conditions}

As already mentioned the integrability of $\l_0,\l_2$ w.r.t. $\cP_0$ is relevant  e.g.
in order to apply  to 
Delaunay triangulations the  results presented  in \cite{F_sep,F_hom,F_muratore,F_resistor,F_in_prep}  for random walks,  resistor networks and  symmetric simple exclusion processes (briefly, SSEPs), all with random conductances.
 Applications of the integrability of $\mu^\o(0)^p$ and $\nu^\o(0)^p$ w.r.t.  $\cP_0$ can be found in \cite{AFST}. In this section we just 
provide a brief overview of the results 
  in \cite{F_sep,F_in_prep} for the SSEP. In the next section we will focus on general (also non-symmetric) SEPs. 
%\begin{multline*}
%\centerline{
%\text{

\smallskip

\emph{In this section we assume  the following\footnote{The last assumption could indeed be weakened: it is sufficient that, $\cP$--a.s., the graph obtained from ${\rm DT}(\hat\o)$ by removing the edges $\{x,y\}$ with  zero conductance $c_{x,y}(\o)$ is connected.
}:
\begin{itemize}
\item  (A1') The probability measure $\cP$ is stationary and ergodic; % for the action $(\theta_x)_{x\in \bbR^d}$ on the probability  space $(\O, \cF, \cP)$;
\item (A2);  
 \item (A3);
\item $\l_0,\l_2 \in L^1(\cP_0)$;
\item $L^2(\cP_0)$ is separable;
\item   $c_{x,y}(\o)>0$ for all $\{x,y\}\in \cE_{\rm DT}(\hat \o)$ for all $\o \in \O_*$, with $\O_*$ given by (A3).
\end{itemize}}

It is easly to check that (A1'), (A2) and (A3) imply (A1), (A2) and (A3). Indeed, by ergodicity, the translation invariant set $\{\o\in \O_*\,:\, \hat\o=\emptyset\}$ has $\cP$--probability equal to either zero or one. In the latter case, we would have $m=0$, contradicting (A3). Using also that $\cP(\O_*)=1$, we conclude that  $\cP( \o\in \O\,:\, \hat\o=\emptyset)=0$.

%}\\
%} \centerline{
%\text{\emph{(H1');\; (H2');\; (H3);\; $\l_0,\l_2 \in L^1(\cP_0)$;\;$L^2(\cP_0)$ is separable.} }
%\end{multline*}

We point out that the above assumptions imply the ones in \cite{F_hom} with exception of \cite[(A3) and (A9)]{F_hom}. They also  imply all the assumptions  in \cite[Section 3]{F_in_prep}. In the latter, we  discuss how the results presented in \cite{F_sep,F_hom,F_resistor} can be
derived under weaker assumptions than the original ones.

As discussed in \cite{F_hom}, the separability of $L^2(\cP_0)$ is a very weak assumption. A very general condition assuring the separability  is provided  in \cite{F_hom} before Remark 2.1 there.

\begin{Definition}\label{fasma98}  We define the  effective homogenized   matrix $D$ as 
the $d\times d$ nonnegative symmetric matrix such that,  for all $a\in \bbR^d$, 
 \begin{equation}\label{def_D}
 a \cdot Da =\inf _{ f\in L^\infty(\cP_0) } \frac{1}{2}\int_{\O_0} d\cP_0(\o)\int_{\bbR^d} d\hat \o (x) c_{0,x}(\o) \left
 (a\cdot x - \nabla f (\o, x) 
\right)^2\,,
 \end{equation}
 where $\nabla f (\o, x) := f(\theta_x \o) - f(\o)$.
\end{Definition}
The above definition is well posed as discussed in  \cite{F_in_prep}.

We consider the SSEP on $\rmdt (\hat \o)$ with formal generator
\[
\cL_\o  f(\eta)= \sum _{ \{x,y\} \in \cE_{\rmdt(\hat \o)} }c_{x,y}(\o) \left( f(\eta^{x,y}) - f(\eta) \right)\,,
\]
where $\eta\in \{0,1\}^{\hat \o}$ is the particle configuration, $\eta(x)$  is the particle occupation number at $x$  and the configuration $\eta^{x,y}$ is obtained from $\eta$ by exchanging the occupation numbers $\eta(x),\eta(y)$.

As   discussed in  \cite{F_in_prep}, even for a larger class of weighted graphs $\cG(\o)$, thanks to the results obtained in \cite{F_muratore},  under our assumptions and  for $\cP$--a.a.~$\o$ the above SSEP admits a rigorous graphical construction.  It is a Feller process with state space $\{0,1\}^{\hat\o}$, and the  local functions form a core for the generator.  
Moreover  \cite[Theorem~1]{F_sep}  can be applied, even though  not all the assumptions stated there are required in the present setting  (\cite{F_in_prep} indeed   fills this gap). As a consequence,
 for $\cP$--a.a. $\o$, under diffusive space-time rescaling,  the above SSEP  on $\rmdt (\hat \o)$ satisfies  a hydrodynamic limit both in path space and at fixed times with hydrodynamic equation \[ \partial_t \rho= {\rm div} \bigl( D \nabla \rho\bigr)\,,\]  where $\rho(x,t)$ denotes the macroscopic density profile.
  
  For  results concerning the equilibrium fluctuations of the above SSEP we refer to \cite{CF}.
  
  %\verde{devo aggiungere la sezione di \cite{F_in_prep} quando lo cito}
%%%%%%%%%%%%%%%%%%%%%%%%%%%%%%%%%%%%%%%%%%%%%%%%%%%%%%%%%%%%%%%%%%%%%%

%%%%%%%%%%%%%%%%%%%%%%%%%%%%%%%%
\section{Main results: validity of Assumption SEP of \cite{F_muratore}}\label{sec_results2}

In this section we remove our main Assumptions (A1), (A2), (A3) and we start from scratch.  We assume, without further mention, the following:

\smallskip 

(A) \emph{On the probability space $(\O,\cF,\cP)$ we have a  simple point process  $\O\ni \o\mapsto \hat \o\in \cN$  on $\bbR^d$ whose law $\bbP$ is stationary (hence $\bbP(\cN_{\rm pol})=1$) and has finite intensity $m:=\bbE[\xi ([0,1]^d)]$. Moreover for each $x,y \in {\rm DT(\hat\o)}$ with $x\sim y $ we assume to have non-negative values $c^{\rm o}_{x,y}(\o)$ and $c^{\rm o}_{y,x}(\o)$.
}

\smallskip

We denote by $\cE_{\rm DT}^{\text{o}}(\hat\o)$ the family of the edges of ${\rm DT}(\hat \o)$ 
endowed with the orientation:
\[ 
\cE_{\rm DT}^{\text{o}}(\hat\o):=\left\{ (x,y), \;(y,x)\,: \,\{x,y\}\in \cE_{\rm DT}(\hat\o) \right\}\,.
\]
While the well-definedness of the SSEP of ${\rm DT}(\hat\o)$ is equivalent to the  well-definedness of the random walk $(X^\o_t)_{t\geq 0}$ as derived in \cite[Section~3]{F_muratore}, the construction of simple exclusion processes on ${\rm DT}(\hat\o)$  with non-symmetric rates  is more delicate. In particular, we are interested in the SEP with state space $\{0,1\}^{\hat\o}$ and formal  generator 
\be\label{lonzetta} \cL_\o f (\eta):= \sum _{x\in \hat \o}\sum _{y\in\hat\o: y\sim x} \eta(x) \big( 1-\eta(y) \big) c^{\rm o}_{x,y}(\o) \left( f(\eta^{x,y})-f(\eta) \right)\,,
\en
where $c^{\rm o}_{x,y}(\o) \geq 0$ is not necessarily symmetric.

Let us set $c_{x,y}(\o):= c^{\rm o}_{x,y}(\o)+ c^{\rm o}_{y,x}(\o)$ for $x\sim y$.
We recall Assumption SEP from \cite[Section~5]{F_muratore} adapted to our setting with a random environment:
\begin{Definition}[Assumption SEP]  \label{livio}
 For $\cP$--a.a. $\o$ there exists  $t_0=t_0(\o)>0$ such that 
the following holds. 
Consider the graph obtained by thinning the Delaunay triangulation
 $DT(\hat \o)$ by keeping each edge $\{x,y\}$ of $\cE_{\rmdt}(\hat\o) $ with probability $1- e^{-t_0 c_{x,y}(\o)}$ independently from the rest (otherwise erase it). Then a.s. the resulting undirected  graph   has only connected components with finite  cardinality.
\end{Definition}

Then, as derived in \cite[Section~5]{F_muratore}, under Assumption SEP  for $\cP$--a.a.~$\o$ one can provide a construction of the SEP with transition rates $c^{\rm o}_{x,y}(\o)$ based  on Harris' percolation argument. The resulting Markov process on $\{0,1\}^{\hat \o}$ is Feller and the infinitesimal Markov generator of the semigroup on the set $C(\{0,1\}^{\hat \o})$ of continuous functions on 
$\{0,1\}^{\hat \o}$ endowed  with uniform topology satisfies \eqref{lonzetta} when $f$ is a local function. %We recall that $f(\eta)$ is local if there exists a finite set $A\subset\hat\o$ such that  $f(\eta)$   is determined by the occupation numbers $\eta(x)$ as $x$ varies $A$.

Our new  result in this section gives sufficient conditions  implying Assumption SEP:\begin{Theorem}\label{teo2}  Suppose that the law $\bbP$ has a finite range of dependence. In addition suppose that, for some non-random constant $C_*>0$,  it holds 
\be \label{sissi}
c_{x,y}(\o) \leq C_* \qquad \text{for all } \{x,y\}\in \cE_{\rm DT}(\hat\o)\,,\; \text{ for }\cP\text{--a.a.~$\o$}\,.
\en 
 Then Assumption SEP is satisfied.
\end{Theorem}

The proof of Theorem \ref{teo2} is based on Theorem \ref{teo_perc} below and is given at the end of this  section. Trivially, the above  request about  $c_{x,y}(\o)  $
means simply that the proposed jump rates $c_{x,y}^o(\o)$ in \eqref{lonzetta} are uniformly bounded by some constant.
%\verde{\begin{Remark} One could think to weaken the hypotheses in Theorem \ref{teo2} by requiring instead of \eqref{sissi}  just  that   for $\cP$--a.a.~$\o$ there exists $C(\o)>0$ such that $c_{x,y}(\o) \leq C(\o)$ for all $\{x,y\}\in \cE_{\rm DT}(\o)$. On the other hand, by the ergodicity of $\cP$, one  concludes that \eqref{sissi} must hold  for some $C_*>0$.\end{Remark}}

\smallskip
Theorem \ref{teo_perc} below is a  result about the Bernoulli  bond percolation on the Delaunay triangulation, which has its own interest. For this part, the weights $c_{x,y}^{\rm o}(\o)$ are irrelevant and one could just deal  with $\bbP$.
To state Theorem \ref{teo_perc} we introduce the  Bernoulli bond percolation on the graph $\rmdt(\xi)$ with $\xi$ sampled according to $\bbP$. We fix $p\in [0,1]$ and we consider an enlarged probability space $(\cX, \bbQ)$,
where $\cX:=\left\{ 
(\xi, W)\,:\, \xi \in \cN_{\rm pol} \,,\, W \in \{0,1\} ^{\cE_{\rmdt(\xi)}}\right\}$ and $\bbQ$ is the  probability measure on $\cX$ such that 
\begin{itemize}
\item[(i)] $\bbQ \circ \pi_1^{-1}=\bbP_{|\cN_{\rm pol}}$, where
 $\pi_1 : \cX \to \cN_{\rm pol}$ is the projection $\pi_1 (\xi, W)= \xi$;
 \item[(ii)]
  $\bbQ( \cdot | \xi) $ is a Bernoulli probability measure on $\{0,1\} ^{\cE_{\rmdt(\xi)}}$ with  parameter $p$.
  \end{itemize} We have been somewhat informal  in the above definition for what concerns measurability  and in particular the $\s$--algebra of events. The complete construction of the probability space (including the $\s$-algebra) would be the same as for the connection model  built on a point process with \emph{connection function}  $g(x,y):= p \mathds{1}_{\{x,y\}\in \rmdt(\xi)}$  (cf. e.g.~\cite{FH} and references therein).

To simplify the notation in what follows we write $W_{x,y}$ instead of $W_{\{x,y\}}$. In particular, it holds $W_{x,y}=W_{y,x}$.
\begin{Definition}[Graph $\cG(\xi, W)$] \label{grafo_cG} Given $(\xi,W) \in \cX$ 
we  define the random graph $\cG= \cG(\xi, W)$ as follows: the vertices of $\cG$ are the points in  $\xi$, the edges of $\cG$ are  given by the  pairs  $\{x,y\}\in \cE_{\rmdt}(\xi)$ with $W_{x,y}=1$. 
\end{Definition}
We can now state our last  result concerning  the  subcritical regime for the random graph $\cG(\xi,W)$:
\begin{Theorem}\label{teo_perc}  Suppose that the law $\bbP$ has  finite range of dependence. Then there exists $p_*\in (0,1)$ such that, given $p\in [0,p_*]$, the graph  $\cG(\xi,W)$ has only bounded connected components $\bbQ$--a.s. 
\end{Theorem}
The proof of the above theorem is given at the end of Section  \ref{sec_bond_perc}.
We conclude with the proof of Theorem~\ref{teo2}.

\begin{proof}[Proof of Theorem~\ref{teo2}] Let $p_*$ be as in Theorem~\ref{teo_perc}.
By taking $t_0$ large enough we have   for $\cP$--a.a.~$\o$
\[ 1- e^{-t_0 c_{x,y}(\o)}\leq 1- e^{-t_0 C_*}\leq p_* \qquad \forall \{x,y\}\in \cE_{\rm DT}(\hat\o)\,.\]
In particular, by a coupling on an enlarged probability space,
     the thinned undirected graph appearing in Definition~\ref{livio} can be  embedded  into a graph with the same law of the graph $\cG(\xi, W)$ appearing in Definition~\ref{grafo_cG} when $(\xi, W)$ is sampled by the law $\bbQ$ associated to $p:=p_*$. On the other hand, by Theorem~\ref{teo_perc} this larger graph has a.s. only connected components of finite cardinality.  Hence the same holds for the thinned undirected graph appearing in Definition~\ref{livio}.
\end{proof}

%%%%%%%%%%%%%%%%%%%%%%%%%%%%%%%%%%%%%%%%%%%%%%%%%%%%%%%%%%%%%%%%%%%%%%%%%%%%%%%%%%%%%%%%%%%%%%%%%%%%%%%%%%%%%%%%%%%%%%%%%%%%%%%%%%%%%%%%%%%%%%%%%%%%%%%%%%%%%%%%%%%%%%%%%%%%%%%%%%%%%%%%%%%%%%%%%%%%%%%%%%%%%%%%%%%%%%%%%%%%%%%%%%%%%%%%%%%%%%%%%%%%%%%%%%%%%%%%%%%%%%%%%%%%%%%%%%%%%%%%%%%%%%%%%%%%%%%%%%%%%%%%%%%%%%%%%%%%%%%%%%%%%%%%%%%%%%%%%%%%%%%%%%%%%%%%%%%%%%%%%%%%%%%%%%%%%%%%%%%%%%%%%%%%%%%%%%%%%%%%%%%%%%%%%%%%%%%%

% L
%%%%%%%%%%%%%%%%%%%%%%%%%%%%%
\section{Some examples}\label{sec_esempi}
There are several popular classes of simple point processes. We just  give a flavor of applications, discussing  in what follows   void probabilities and  boundedness of $\rho_\g$  in particular for  determinantal and Gibbsian point processes.

Some examples of stationary ergodic point processes with finite range of dependence are the homogeneous  Poisson point processes,  the Mat\'ern cluster processes and the Mat\'ern hardcore processes discussed in 
\cite[Appendix~B]{Ro0}.   Due to Proposition~\ref{prop_FR} for all these SPPs the void probability decays exponentially in the volume and in particular  $\bbP( \{ \emptyset\})=0$. For the homogeneous  PPP it is known that $\rho_\g<+\infty$ for all $\g>0$, then same then holds for  Mat\'ern cluster processes and the Mat\'ern hardcore processes (cf.~\cite[Appendix~B]{Ro0}).

\subsection{Determinantal point processes} 
Let $\cK$  be a locally trace--class self--adjoint operator on $L^2(\bbR^d,dx)$ with $0\leq \cK\leq \mathds{1}$.  It is always possible
to associate with $\cK$ a kernel $K$ such that ${\rm Tr}(\cK \mathds{1}_B) =\int_B K(x,x)dx$
for any $B\in \cB(\bbR^d)$ bounded.  Then the law $\bbP$ of the determinantal SPP associated with $\cK$ satisfies 
\[\bbE\bigl[ \xi (B)\big] =\int_B K(x,x)dx\]
for any bounded Borel set $B$ and 
\[ \bbE\bigl[ \xi(B_1) \xi (B_2)\dots \xi(B_k)\big]=
 \int_{\prod_{i=1}^k B_i}  \rho_k (x_1,x_2, \dots, x_k) dx_1 dx_2\dots dx_k
\]
for any family $B_1,B_2,\dots, B_k$ of pairwise disjoint  bounded Borel sets, where
\[ \rho_k  (x_1,x_2, \dots, x_k) :={\rm det} \bigl( K(x_i,x_j) \bigr)_{1\leq i,j \leq k}\,.\]
In particular, the intensity is given by $m=\int_{\L_{1/2}} K(x,x)dx$.  The determinantal SPP becomes stationary if $K(x,y)=K(0,y-x)$. 
Since determinantal SPPs have negative association, we have that the void probability on a box $\L_\ell$ decays at least as $e^{- c \ell^d}$ for some $c>0$, as $\ell\to +\infty$ (cf. Proposition~\ref{prop_FR}). As a consequence $\bbP(\xi=\emptyset)=0$.
Finally we point out that  $\rho_\g<+\infty$ for all $\g>0$ due to \cite[Theorem~2]{S}.  In particular also $m$ is finite.
%As a consequence any  stationary determinantal SPP with $m:=\int _{\L_{1/2}} K(x,x)\not =0$ satisfies our main Assumptions (A1), (A2), (A3) as well as (H1) in Theorem~\ref{teo1}, (H1') in Corollary~\ref{fornello}, condition $\rho_2<+\infty$ in Proposition~\ref{prop_veloce} and (C3) in Theorem~\ref{teo_mimmo}.
\subsection{Gibbsian point processes} 
In this subsection we provide sufficient conditions for    the boundedness of $\rho_\g$ (cf.~Proposition~\ref{arpa}) and for  condition $C(\a)$ for a Gibbsian point process (cf.~Proposition~\ref{saif1} and \ref{saif2}). 

In what follows we set $\cN_f:=\{\xi \in \cN\,:\, \xi (\bbR^d)<+\infty\}$ and, given $\L\subset\bbR^d$, we set $\cN_\L:=\{\xi\in\cN\,:\, \xi\subset\L\}$. Moreover, given $\eta\in \cN$ and $\L\subset \bbR^d$, we denote by $\eta _\L$ the point configuration $\eta\cap \L$.

Let us consider a Gibbsian point process on $\bbR^d$ with Hamiltonian $H$ associated to a  potential $V$,  inverse temperature $\b$ and activity $z$ with respect to the Lebesgue measure $dx$ \cite{J}. 
This means that the law $\bbP$ of the above point process is 
a  Gibbs probability measure on $(\cN, \cB(\cN))$, i.e. it satisfies the DLR equation below.  Moreover, 
  the  Hamiltonian $H: \cN_f \to \bbR\cup \{+\infty\}$  is related to $V$ by the identity 
\[
H(\xi) =\sum_{A\subset \xi}  V(A) \qquad \forall \xi \in \cN_f\,.
\]

As in \cite[Definition~3.1]{J} we assume that  the  Hamiltonian $H: \cN_f \to \bbR\cup \{+\infty\}$ satisfies for all $\xi \in\cN_f$ the following properties where $C\geq 0 $ is a fixed constant:  
\begin{itemize}
\item[(i)] $H(\emptyset)=0$ (non-degeneracy);
\item[(ii)]  if $ H(\xi)<+\infty $, then $H(\xi \setminus \{x\})<+\infty$ for all $x\in \xi$ (hereditarity);
\item[(iii)] $H(\xi)\geq -C \xi (\bbR^d)$ (stability).
\end{itemize}
Non-degeneracy sometimes is  more generally stated as $H(\emptyset)<+\infty$ as in \cite[Definition~1]{De}. In what follows we keep (i).

\smallskip

We recall that $H$ is \emph{locally stable}  with constant $C\geq 0$ if for any $\xi\in \cN_f$ with $H(\xi)<+\infty$ it holds $H(\xi\cup\{x\})-H(\xi) \geq -C$ (see \cite[Definition~3.13]{J}).  Note that local stability implies stability.  As  discussed in \cite{J} (see there Examples 3.2 and 3.3 and the discussion after Definition 3.13)  examples of locally stable Hamiltonians satisfying the above properties (i), (ii), (iii), are 
the Widom-Rowlinson model, and     pair potentials (i.e. $V(\{x,y\})= v(|x-y|)$ with $v:\bbR_+\to \bbR\cup\{+\infty\}$ and $V(A)=0$ for $|A|\not =2$,) for which either $v(r)\geq 0$ or 
\begin{itemize}
\item $v$ has a hard core, i.e. for some $r_{\rm hc}>0$ it holds $v(r)=+\infty$ for $r<r_{\rm hc}$;
\item $v$ is lower regular, i.e. $v(r) \geq -\psi (r)$ for all $r\geq 0$, where $\psi:\bbR_+\to \bbR_+$ is weakly decreasing and $\int_0^\infty r^{d-1} \psi(r) dr<+\infty$.
\end{itemize}

\smallskip

We recall that given $\L\subset \bbR^d$ bounded,  $\xi \in \cN_\L$ and $\eta\in \cN$, the conditional energy $H_\L(\xi|\eta)$ 
is defined as (cf.~\cite[Section~3.3]{J}) 
\[  H_\L(\xi|\eta):= \sum _{\substack{A\subset \xi \cup \eta_{\L^c}\\ A\cap \xi\not= \emptyset}} V(A)\,.
\]
when the sum in the r.h.s. is absolute convergent, otherwise $H_\L(\xi|\eta):=+\infty$.

When $H$ is locally stable with constant $C$ then  (se the proof of \cite[Lemma~3.14]{J})\be\label{tartufo}
H_\L (\xi | \eta) \geq - C \xi(\L)
\en
for any $\xi \in \cN_\L$ and $\eta$ with  $H(\z)<\infty$ for all $\z\in \cN_f$ with 
$\z\subset \eta$.

In what follows,  given a finite measurable set $\L\subset \bbR^d$ and a configuration $\eta\in \cN$,  we denote  the grand-canonical partition function on $\L$ with Hamiltonian $H$, inverse temperature $\b$ and activity $z>0$ by $\Theta_\L^\eta$. We recall that (cf.~\cite[Section~3.3]{J})
\begin{equation}
\begin{split}
\Theta_\L^\eta:& =
\int _{\cN_\L} d{\rm P}_{\L} (\s) e^{\leb(\L)}z^{\s(\L)}  e^{-\b H_\L(\s|\eta)}\\
&  =1+ \sum_{n=1}^\infty \frac{1 }{n!} 
 \int _\L dx_1\int_\L dx_2 \dots \int_\L dx_n   z^n e^{-\b H_\L( \sum_{i=1}^n \d_{x_i} |\eta) } 
\,,
\end{split}
\end{equation}
where   ${\rm P}_\L$ is a  homogeneous Poisson 
point process with intensity $1$ and  $ H_\L(\s|\eta)$ is the Hamiltonian of $\s$ with boundary condition $\eta$. 
%%%%%%%
%\begin{equation}
%\begin{split} 
%& \int _{\cN_\L} d{\rm P}_{\L} (\s) e^{|\L|}z^{\s(\L)}  e^{-\b H_\L(\s|\eta)}\\
%& =e^{-|\L|} \sum_{n=0}^\infty \frac{|\L|^n}{n!} 
% \int _\L dx_1\int_\L dx_2 \dots \int_\L dx_n |\L|^{-n}  z^n e^{-\b H( \sum_{i=1}^n \d_{x_i} |\eta) }\\
% & = \sum_{n=0}^\infty \frac{1 }{n!} 
% \int _\L dx_1\int_\L dx_2 \dots \int_\L dx_n  e^{-|\L|}  z^n e^{-\b H( \sum_{i=1}^n \d_{x_i} |\eta) } 
%\,.
%\end{split}
%\end{equation}

We also set $\cR_\L:=\{\eta \in \cN\,:\, \Theta_\L^\eta<+\infty\}$.
Given $\eta\in \cR_\L$ we denote by $\mu^\h_{\L}$  the finite volume Gibbs measure on $\L$ with boundary condition $\h$, inverse temperature $\b$ and activity $z$. We recall that  $\mu^\h_{\L}$ is a probability measure on $(\cN_\L, \cB(\cN_\L))$ such that
\be\label{pera150}
\mu^\h_{\L}[f]
= \frac{1 }{\Theta_\L^\eta} \int _{\cN_\L} d{\rm P}_{\L} (\s) e^{\leb(\L)}z^{\s(\L)}  e^{-\b H_\L(\s|\eta)}f(\s)\,.
\en
Since $H_\L(\emptyset|\eta)=H(\emptyset)=0$,   for any $\eta\in \cR_\L$ it holds
\be\label{vuotone}
\mu^\h_{\L}\big( \s(\L)=0\big) =  \bigl(\Theta_\L^\eta \bigr)^{-1}\,.
%, \qquad \;  \Theta_\L^\eta= \int _{\cN_\L} d{\rm P}_{\L} (\s) e^{|\L|}z^{\s(\L)}  e^{-\b H_\L(\s|\eta)}\,.
\en

The law  $\bbP $ of our Gibbsian point process then satisfies the DLR equation (cf.~\cite[Section~5.2]{J}):
% that we now recall.  Given a set $\L\in \cB(\bbR^d)$ we define  $\cG_\L$ as the $\s$-algebra on $\cN$ generated by the occupation numbers $\cN \ni\xi \mapsto \xi(A)\in \bbN$  with $A$ varying in $\cB(\bbR^d \setminus \L)$. The DLR equation states that,
 for any bounded $\L\in \cB(\bbR^d)$, $\bbP$ has support in $\cR_\L$ and moreover
 \be\label{eq_DLR}
\bbE[f] = \int _\cN d\bbP(\eta) \int _{\cN_\L} d\mu_\L^{\eta}(\s) f( \s \cup  \eta_{\L^c})
\en
for $f$ non-negative.
%where $\s_\L $ denotes the restriction of $\s$ to $\L$, $\eta_{\L^c}$ denotes the restriction of $\eta$ to $\L^c$ and finally $( \s_\L, \eta_{\L^c})$ denotes the point  configuration given by $\s_\L\cup \eta_{\L^c}$.
% if we think in terms of sets of points (particles), equivalently given by $\s_\L+ \eta_{\L^c}$ if we think in terms of counting measures.

%%%%%%%%%%%%%%%%%%%%%%%%%%%%%%%%
\subsubsection{Finite moments} 

\begin{Proposition}\label{arpa}
If the Hamiltonian is locally   stable with constant $C\geq 0 $ (or in general if \eqref{tartufo}  holds), then for all $\a\geq 0$ and $\L\in \cB(\bbR^d)$  it holds 
\be\label{cervo1} \bbE[e^{\a \xi(\L)}] \leq  e^{\leb(\L) (e^t-1) }  \text{ where } t=\ln z + \b C+\a\,.
\en In particular, 
 $\rho_\g<+\infty$ for any $\g \geq 1$.
\end{Proposition}
\begin{proof} It is enough to prove the bound on the moment generating function.
By \eqref{pera150} and \eqref{eq_DLR} we have  \be%\label{eq_DLR}
\bbE[e^{\a \xi(\L)}]  = \int _\cN d\bbP(\eta) 
\frac{1 }{\Theta_\L^\eta} \int _{\cN_\L} d{\rm P}_{\L} (\s) e^{\leb(\L)}z^{\s(\L)}  e^{-\b H_\L(\s|\eta)+ \a \s(\L)}\,.
\en
By \eqref{tartufo} we have 
\be\label{boemondo}
\bbE[e^{\a \xi(\L)} ] \leq \int _\cN d\bbP(\eta) 
\frac{e^{\leb(\L)} }{\Theta_\L^\eta} \int _{\cN_\L} d{\rm P}_{\L} (\s)   e^{(\ln z + \b C+\a)  \s(\L)}\,.
\en
We now observe that  $\Theta_\L^\eta\geq 1$. Moreover,  by the form of the    moment generating function of the Poisson distribution, the last integral over $\cN_\L$ in \eqref{boemondo}  equals
$\exp\big\{  \leb(\L) (e^t-1) \big\}$. These observations allow  to derive 
 \eqref{cervo1} from \eqref{boemondo}.\end{proof}
%%%%%%%%%%%%%%%%%%%%%%%%%%%%%%%%%%%%%%
\subsubsection{Void probabilities}
To upper bound the void probability $\bbP(\xi(\L_\ell) =0)$ we can combine \eqref{vuotone} and \eqref{eq_DLR} getting
\be\label{forza}
\bbP (\xi (\L)=0)= \int _\cN d\bbP(\eta) \big( \Theta _{\L}^\eta)^{-1}\,.
\en

We recall that the potential $V$ is called of  \emph{finite range} if there exists $L\geq 0$ such that $V(\eta)=0$ for any finite point configuration $\eta$ whose diameter is greater than  $L$. When this property holds for some $L$, we say that the range is smaller  than $L$. We point out that in this case $\bbP$ does not necessarily have finite  range of dependence in the sense of Definition \ref{def_fin_range}. Indeed, suppose that $A,B \subset\bbR^d$ are bounded Borel subsets with distance at least $L$. Then, by the DLR equation, 
\be\label{bobo1}
\bbP [f(\xi_A) g(\xi_B)]= \int_\cN  d\bbP(\eta)\int _{\cN_{A\cup B} } d \mu_{A\cup B}^\eta (d\s) \, f(\s_A )g(\s_B)\,. 
\en
Since $\s(A\cup B)=\s(A)+\s(B)$ and $H_{A\cup B} (\s| \eta)= H_{A} (\s_A| \eta)+ H_B(\s_B|\eta)$ (in the last identity we used the assumption that $V$ has range smaller than $L$) we conclude that 
$ d \mu_{A\cup B}^\eta (\s) $ factorizes as
$ d\mu_{A\cup B}^\eta (\s)=d\mu_A^\eta (\s_A)d\mu_B^\eta ( \s_B)$ with $\s=\s_A\cup \s_B$. Hence for \eqref{bobo1} we can just say that 
\be\label{bobo2}
\bbP [f(\xi_A) g(\xi_B)]= \int_\cN  d\bbP(\eta)\int _{\cN_{A} } d \mu_{A }^\eta (\s) \int _{\cN_{B} } d \mu_{B }^\eta (\s')  f(\s)g(\s')\,.  
\en
For later use we point out that, reasoning as above we have that, given bounded sets $A_1,A_2,\dots, A_n \in \cB(\bbR^d)$ with reciprocal distance at least $L$, it holds
\be\label{palermo}
d \mu_{\L} ^\eta (\s)= \prod _{k=1}^n d\mu^\eta _{A_k} (\s_{A_k}) \;\;\text{ where } \L=\cup _{k=1}^n A_k\,, \;\; \s=\s_{A_1}\cup \s_{A_2}\cup\cdots\cup \s_{A_n}\,.
\en

\begin{Proposition}\label{saif1}
If the Hamiltonian is associated to a finite range potential and the one particle potential is lower bounded by a constant, then   $\bbP(\xi(\L_\ell) =0) \leq e^{-c \ell^d}$ for some $c>0$ and for $\ell $ large enough. In particular, the void probability satisfies condition $C(\a)$ for any $\a>0$.\end{Proposition}
%%%%%
\begin{Remark} If the potential is translation invariant, i.e. $V(A)=V(A+x)$ for all $A\in \cN_f$ and all $x\in \bbR^d$, then  trivially  the one particle potential $V(\{a\})$  is a constant independent from $a$.\end{Remark}

\begin{proof}
Suppose that $V(\eta)=0$ if ${\rm diam}(\eta)> L$.
We consider in $\L_\ell$ a maximal family of boxes $\{B_w\,:\, w\in \cW\}$ with side length $4L$ and with reciprocal distance at least $L$. Then,  for $\ell$ large enough,  $|\cW|\geq c\ell^d$ for a suitable constant $c$ depending on $d$ and $L$.

Using the DLR equation with $\L:=\cup _{w\in \cW} B_w$ and using the factorization formula
\eqref{palermo}, we have 
\be\label{cerrai}
\begin{split}
\bbP( \xi(\L_\ell)=0) & \leq \bbP( \xi(\L)=0) =\int _\cN d\bbP(\eta) \mu^\eta_\L( \s(\L)=0) 
\\
& = \int_\cN  d\bbP(\eta)  \prod_{w\in \cW} \mu^\eta _{B_w} (\s( B_w)=0) \leq  \k^{c\ell^d}
\end{split}
\en
where 
\[ 
\k:= \inf _\z \mu^\z_{\D} (\{ \emptyset\})=\left( \sup _\z \Theta^\z_{\D} \right)^{-1} \qquad \D:=\L_{ 2L}=[-2L,2L]\,.
\]
We note that $e^{\leb(\D)} {\rm P_\D}(\{\emptyset\})  \exp\{  -\b H_{\D}(\emptyset |\z)\} =1$ since $H_{\D}(\emptyset |\z)=0$.
Setting $\D':= \L_{L/2}$, let now $\cA\subset \cN$ be the event  that there are no points in $\D\setminus \D'$, while there is exactly one point in $\D'$.
We call $C$  the lower bound of  the one-point potential.  Then, given  $\s\in \cA$,  we have $H_\D(\s|\z)\geq C$. In particular,   we get
\[ 
\int _{\cA} d{\rm P}_{\D} (\s) e^{\leb(\D)}z^{\s(\D)}  e^{-\b H_\D(\s|\z)} \geq z e^{\leb(\D)-\b C}   {\rm P}_{\D}(\cA)=z e^{-\b C}\leb(\D')\,.\]
Hence $\Theta^\z_\D\geq1+ z e^{-\b C}\leb(\D')$. This proves that  $\k<1$, which - combined with  \eqref{cerrai} - implies our claim.
%\newpage
% Hence 
%%and distance from $\partial \L_\ell$ at least $L$.
%% Fixed $\cM \subset \cZ$ we define $\D_\cM:=\L \setminus \cup_{z\in \cM} \L_{1/2} (z)$.  
%For each $w\in \cZ$ we define $F_w$ as the event that in $B_w$ there is at most one particle, we define $G$ as the event that there are no particles in $\D:=\L_\ell \setminus\left( \cup_{w\in \cZ} \L_{1/2}(w) \right)$ and we set $F:=\cap _{w\in \cZ} F_w$. Then, by \eqref{vuotone} and setting $\L:=\L_\ell$, we have
%\be\label{eccolo} \Theta_\L^\eta\geq 
%\int _{\cN_{\L}} d{\rm P}_{\L} (\s) z^{\s(\L)}  e^{-\b H_\L(\s|\eta)}\mathds{1}_{F\cap G}(\s)\,.
%\en
%Then, by construction and since $V$ has range smaller than $L$, we have for $\s\in F\cap G$
%\be
%H_\L(\s |\eta)=\sum_{w\in \cZ} H_{B_w }(\s_{B_w}|\emptyset) \leq  c \s( \cup _w B_w) \,. 
%\en
%Calling $v(x)$ the one-particle potential for a particle at $x$, by assumption $v(x)\leq c$ for all $x\in \bbR^d$\footnote{This is immediate for translation invariant potentials}.
%By combining this bound with  \eqref{eccolo}  and using the ${\rm P}_\L$ is the PPP on $\cN_{\L}$ with intensity $dx$, we get
%\be
%\Theta ^\eta_\L\geq e^{- |\L_\ell \setminus B|} \prod_{w\in \cZ} (e^{-|B_w|} + e^{-|B_w|}|B_w| ze^{c}) = e^{-|\L_\ell|} (1+ z e^c)^{c\ell^d}
%\en
%
%$$**** $$
%
%we consider the boxes $\L_{1/2}(w)$, $w\in \cZ$, where the set $\cZ$ is  as in the proof of Proposition \ref{prop_FR}.
%
\end{proof}

If the potential is not of   finite range, one can anyway give reasonable sufficient conditions leading to  Condition $C(\a)$.  For simplicity we treat the case of pair potentials:
\begin{Proposition}\label{saif2} Suppose the Hamiltonian 
satisfies \eqref{tartufo} (e.g.~$H$ is locally stable)
  and is  due to pair interactions with  pair potential $v$ satisfying $v(x,y)\leq \phi (|x-y|)$ for some weakly decreasing function $\phi : (0,+\infty)\to \bbR$. Then Condition $C(\a)$ holds for any $\a>0$ in the following cases: 
    \begin{itemize}
  \item[(i)]   $d\geq 2$ and  $  \sum_{n=1}^\infty n^{d-1} \phi(n) <+\infty$,
  \item[(ii)]  $d=1$ and $  \sum_{n=1}^\infty (\ln n) \phi(n) <+\infty$.
  \end{itemize}
  Moreover in Case (i), for some $c',c''>0$, it holds   $\bbP(\xi(\L_\ell)=0)\leq c' e^{- c'' \ell^{d-1}}$ for  all $\ell\geq 1$.   
\end{Proposition}
%%%%%%%
\begin{proof} 
Given $\ell>0$ we want to estimate the void probability  $\bbP(\xi(\L_\ell)=0)$. Below constants as $c,C,C_i,..$ can change from line to line and do not depend on the relevant parameters as $\ell$ and on the point configurations.

We first take Item (i) where  $d\geq 2$. 

 Due to \eqref{cervo1} and since we have a stable potential, for any $\L\in \cB(\bbR^d)$ it holds $\bbE[e^{\xi(\L)} ] \leq  e^{c  \cL(\L)  } $ with $c:=e^{ \ln z+ \b C+1}$.
Consider the annulus $A_r:=\L_{r+1}\setminus\L_r$. Then, setting
 $E_r:=\{\xi \in \cN\,:\,\xi(A_r) \geq    (1+c) \leb(A_r)\}$, by the exponential Chebyshev inequality  we can bound
$\bbP(E_r) \leq e^{-(1+c) \leb(A_r)}\bbE[e^{\xi(A_r)}]\leq 
e^{-\leb(A_r)}$.
Since $d\geq 2$, this implies that 
\[
\bbP( \cup _{r=\ell}^{\infty} E_r) \leq \sum_{r=\ell}^{\infty}e^{-\leb(A_r)}\leq \sum_{r=\ell}^{\infty} e^{-C_0 r^{d-1}}
\leq  C_1 e^{-C_0 \ell^{d-1}}\,.
\]
Let $E=\cup _{r=\ell} ^{\infty} E_r$. Note that any $\xi\in E^c$ satisfies $\xi (A_r) \leq (1+c) \leb(A_r)$ for all integers $r\geq\ell$.
Take now a maximal  family of disjoint boxes $B_w$, $w\in \cW$, inside $\L_{\ell/2}$  with unit  side length and  reciprocal distance at least $1$. Then $\sharp \cW \geq C_2 \ell^{d}$ for $\ell$ large enough.  Let $B=\cup _{w\in \cW} B_w$ and 
let $\cA$ be the event in $\cN_{\L_\ell}$ given by the  configurations $\s$ which are  empty in $\L_\ell \setminus B$ and  have  at most one particle in each box $B_w$. Take $\s\in \cA$ and  $\eta\in E^c$. Then 
\be
\begin{split}
H_\L(\s|\eta) & = \sum _{x\in  \s} \sum _{\substack{y\in \s\cup \eta_{\L_\ell^c}\\x\not=y}} v(x,y)\\
& \leq 
\sum _{x\in  \s}  \sum _{\substack{y\in \s:\\ y\not =x}} \phi(|x-y|)+ \sum _{x\in  \s}  \sum _{y\in \eta_{\L_\ell^c}} \phi(|x-y|)=:A_1+A_2\,.
\end{split}
\en
By  definition of $\cA$, $A_1 \leq C_3 \s(B)  \sum_{z\in \bbZ^d} \phi(|z|)$. Moreover, using  also that $\eta\in E^c$, we get that 
\[A_2 \leq \sum_{x\in\s} \sum_{r=\ell}^\infty \sum_{y\in A_r\cap \eta_{\L_\ell^c}} \phi(|x-y|) \leq C_4  \s(B) \sum_{r=\ell}^\infty r^{d-1}\phi(r-\ell/2) \,.
\]
Hence $H_\L(\s|\eta) \leq (C_3+C_4') \s ( B) \sum _{n \geq 1} n^{d-1} \phi (n)=: C_5 \s(B)  $ for $\s\in \cA$ and $\eta \in E^c$ (by assumption $C_5$ is indeed finite). By \eqref{vuotone} we then have for all $\eta\in E^c$
\be
\begin{split}
 \Theta ^\eta_{\L_\ell} 
 &\geq \int _{\cN_{\L_\ell}} d{\rm P}_{\L_\ell} (\s) \mathds{1}_\cA(\s)e^{\leb(\L_\ell) }z^{\s(\L_\ell )}  e^{-\b H_{\L_\ell}(\s|\eta)}\\
 & \geq
 \int _{\cN_{\L_\ell}} d{\rm P}_{\L_\ell} (\s) \mathds{1}_\cA(\s)e^{\leb(\L_\ell) }
 z^{\s(B)}  e^{-\b C_5 \s(B) }  \\
 &= \int _{\cN_B} d{\rm P}_B e^{\cL(B)} \prod_{w\in \cW} \Big( \mathds{1}( \s(B_w)=0)+\mathds{1}(\s(B_w)=1) z e^{-\b C_5}\Big) 
\\
& =( 1+z e^{-\b C_5})^{\sharp \cW}\,.
\end{split}
\en
As a consequence for $\ell$ large we have  
\be\label{patatine}
\begin{split}
\bbP(\xi(\L_\ell)=0)& \leq \bbP(E)+ \int _{E^c}d\bbP(\eta) \mu_\L^\eta (\{\emptyset\})=\bbP(E)+\int _{E^c}d\bbP(\eta) (\Theta_\L^\eta)^{-1}\\
&\leq C_1 e^{-C_0 \ell^{d-1}} +( 1+z e^{-\b C_5})^{-C_2\ell^d}\,.
\end{split}
\en
%We move to $d=1$. 
%For Item (ii) we set 
%$E_r:=\{\s \in \cN\,:\,\s(A_r) \geq    c |A_r|+ r^\e\}$. By the exponential Chebyshev inequality and \eqref{cervo1} we have  
%$\bbP(E_r) \leq e^{- r^\e } $. Moreover in the bound of $A_2$ we have
%$A_2 \leq \sum_x \sum_{r=1}^\infty \sum_{y\in A_r} \phi(|x-y|) \leq C_4  \s(B) \sum_{r=1}^\infty (1+r^\e) \phi(r) $. 
%Hence we have    $H_\L(\s|\eta) \leq (C_3+C_4) \s ( B) \sum _{n \geq 1} (1+n^{\e}) \phi (n)=: \s(B) \g $. By reasoning as in Item (i) 
%Eq. \eqref{patatine} becomes
%\be\label{papatine_bis}
%\bbP(\xi(\L_\ell)=0) \leq C_1 \sum_{r=1}^\infty e^{-r^\e} +( 1+z e^{-\b \g})^{-C_2\ell^d}\,.
%\en
We move to  Item (ii).  We proceed as done for Item (i) but now we set 
$E_r:=\{\xi \in \cN\,:\,\xi(A_r) \geq    c \cL(A_r) +\g \ln r\}$ with $\g=\a+1$. By the exponential Chebyshev inequality and \eqref{cervo1} we have  
$\bbP(E_r) \leq e^{- c \cL(A_r) -\g \ln r}\bbE[e^{\xi(A_r)}]\leq   r^{-\g} $ (the constant $c$ is as in the proof of Item (i)). We take $\s\in \cA$ and $\eta\in E^c$. We bound $A_1$ as done above.  In the bound of $A_2$ we now have
\[
A_2 \leq\sum_{x\in\s} \sum_{r=\ell}^\infty \sum_{y\in A_r\cap \eta_{\L_\ell^c}}  \phi(|x-y|) \leq C_4  \s(B) \sum_{r=\ell}^\infty (1+\g \ln r ) \phi(r-\ell/2) \,.
\]
Hence we have    $H_\L(\s|\eta) \leq (C_3+C_4') \s ( B) \sum _{n \geq 1}(1+ \g \ln n)  \phi (n)=: C_5\s(B) $ (by assumption $C_5$ is indeed finite). By reasoning as in Item (i) 
Eq. \eqref{patatine} becomes for $\ell$ large
%\be\label{papatine_bis}
\[
\bbP(\xi(\L_\ell)=0) \leq \sum_{r=\ell}^\infty r^{-\g} +( 1+z e^{-\b C_5})^{-C_2\ell^d} \leq C \ell^{-\g+1}=C\ell^{-\a}\,.
\]
\end{proof}
%%%%%%%%%%%%%%%%%%%%%%%%%%%%%%%
%%%%%%%%%%%%%%%%%%%%%%%%%%%%%%%%%%%%%%%%%%%%%%%%%%%%%%%%%%%%%%%%%%%%%%%%%%%%%%%%%%%%%%%%%%%%%%%%%%%%%%%%%%%%%%%%%%%%%%%%%%%%%%%%%%%%%%%%%%%%%%%%%%%%%%%%%%%%%%%%%%%%%%%%%%%%%%%%%%%%%%%%%%%%%%%%%%%%%%%%%%%%%%
%In what follows we denote  $\deg_{\rmdt(\xi)}(x)$ the degree of the vertex $x\in \xi$ in the Delaunay triangulation $\rmdt(\xi)$.
% Then, for any $k,p \geq 0$,  we can bound
%\be  \bbE_0\Big[\bigl(\sum_{x \sim 0} |x|^k\bigr)^p\Big]\le  \bbE_0\bigl[\bigl(\deg_{\rmdt(\xi)}(0) \cdot \max_{x \sim 0}|x|^k\bigr)^p\bigr]
%\en
%%%%%%%%%%%%%%%%%%%%%%%%%%%%%%%%%%%%%%%%%%%%%%%%%%%%%%%%%%%%%%%%%%%%%%%%%%%%%%%%%%%%%%%%%%%%%%%%%%%%%%%%%%%%%%%%%%%%%%%%%%%%%%%%%%%%%%%%%%%%%%%%%%%%%%%%%%%%%%%%%%%%%%%%%%%%%%%%%%%%%%%%%%%%%%%%%%%%%%%%%%%%%%%%%%%%%%%%%%%%%%%%%%%%%%

\section{Moments of the    number of points in a box}\label{sec_momenti}
%%%%%%%%%%%%%%%%%%%%%%%%%%%%%%%%%%%%%%%%%%%%%%%%%%%%%%%%%%%%%%%%%%%%%%%%%%%%%%%%%%%%%%%%%%%%%%%%%%%%%%%%%%%%%%
In this section  we just assume that $\bbP$ is the law  of  a stationary  simple point process  on $\bbR^d$  with finite and positive intensity $m=\bbE[\xi([0,1]^d)]$. $\bbP_0$ and $\bbE_0$  denote the Palm distribution associated to $\bbP$ and the corresponding expectation, respectively. Our aim here is to analyze the  finite moment condition $\rho_\g:=\bbE[\xi([0,1]^d)^\g]<+\infty$. 
\begin{Lemma}
\label{lem}
 $\forall \g\geq  1$ and $\forall L\in \bbN$ it holds $\bbE[\xi ([0,L]^d)^\g]\le L^{d\g}\rho_\g$.
\end{Lemma}
\begin{proof} The claim is trivial for $L=0$. Let $L>0$. If $\g=1$, the claim follows from the stationarity of $\bbP$ and indeed it holds  $\bbE[\xi ([0,L]^d)]= L^d\rho_1$. Let us take $\g>1$. We use H\"older's inequality with  conjugate exponents $p:=\frac{\g}{\g-1}$ and $q:=\g$. Then, using also the
stationarity of $\bbP$, we get 
 \begin{equation*}
\begin{split}
\bbE[\> \xi([0,L]^d)^\g\>]&=\bbE\Big[\Big(\sum_{z \in [0,L-1]^d \cap \bbZ^d} \xi(z+[0,1]^d)\Big)^\g\>\Big]\\ \le & \bbE\Big[\Big(\sum_{z \in [0,L-1]^d \cap \bbZ^d} 1^{\frac{\g}{\g-1}} \Big)^{\frac{\g-1}{\g}\cdot\g} \Big( \sum_{z \in [0,L-1]^d \cap \bbZ^d} \xi(z+[0,1]^d)^\g\Big)^{\frac{1}{\g}\cdot\g}\Big] \\
=&L^{d(\g-1)} \sum_{z \in [0,L-1]^d \cap \bbZ^d} \bbE[\>\xi(z+[0,1]^d)^\g]
%\\=&L^{d(\g-1)} \sum_{z \in [0,L-1]^d \cap \bbZ^d} \bbE[\>\xi([0,1]^d)^\g ] \\
= L^{d \g} \bbE[\> \xi([0,1]^d)^\g ]\,.\qedhere
\end{split}
\end{equation*}
\end{proof}

\begin{Lemma}
\label{prop1}
For any $\g > 0$,  $L\in \bbN$ and $z\in \bbR^d$, it holds 
\[ \bbE_0[\>\xi(\L_L(z) )^\g]\le  \frac{1}{m}  (2L+2)^{d \g }  \rho_{1+\g}\,.\]
\end{Lemma}
%%%%%%%%%
\begin{proof}
Campbell's formula \eqref{campbell1}  with $f(x,\xi):= \mathds{1}_{ [0,1]^d} (x)  \xi (\L_L(z) )^\g$ implies that 
\begin{multline}
\label{lem1}
        \bbE_0[\>\xi(\L_L(z))^\g]        = \frac{1}{m} \> \int_\cN d\bbP(\xi) \int_{\bbR^d} d\xi(x) f(x,\t_x \xi)\\
        =  \frac{1}{m} \int_\cN d\bbP(\xi) \sum_{x \in \xi \cap [0,1]^d} \xi (x+\L_L(z) )^\g  \le  \frac{1}{m}\bbE[\> \xi([0,1]^d)\xi(\L_{L+1}(z))^\g]\,.
\end{multline}
We now apply H\"older's inequality  with conjugate exponents $p= (1+\g)/\g$ and $q=1+\g$, the stationarity of $\bbP$  and afterwards Lemma \ref{lem}. Since $q=\g p =1+\g$, we get
\be\label{nizzardo}
\begin{split}
\bbE[ \> \xi([0,1]^d)\xi(\L_{L+1}(z))^\g]& \leq \bbE[\> \xi([0,1]^d)^q]^{1/q} \bbE[ \>\xi(\L_{L+1})^{\g p }]^{1/p}\\
& \leq
\rho_q^{1/q} (2L+2)^{d\g} \rho_{ \g p} ^{1/p}
= \rho_{1+\g} (2L+2) ^{d \g}\,.
\end{split}
\en
By combining \eqref{lem1} and \eqref{nizzardo} we get the desired estimate.
\end{proof}
The following proposition clarifies the relation between the moments of the  number of points in a given box w.r.t. $\bbP$ and $\bbP_0$:
\begin{Lemma}
\label{prop2}
Given $\g > 0$ the following facts are equivalent:
\begin{itemize}
    \item[(i)]  $ \bbE_0\left[\xi\left(\L_L(z) \right )^\g\right]<\infty$ for any  $ L >0$ and $z\in \bbR^d$;
   % \item[(ii)]  $ \bbE_0\left[\xi\left(\L_L(z) \right)^\g\right]<\infty$  for some $L>0$ and $z\in \bbR^d$;
    \item[(ii)]  $\rho_{1+\g}<\infty$.
\end{itemize}
\end{Lemma}
\begin{proof} (ii) implies (i) by Lemma \ref{prop1}. Let us show that (i) implies (ii).
%Suppose that $ \bbE_0[\>\xi(\L_L(z))^\g]<\infty$ for some $L>0$ and $z\in \bbR^d$. We fix $l \in (0,L/2)$ and 
We
apply Campbell's formula \eqref{campbell1} with $f(x,\xi):=\mathds{1}_{ \L_{1/2}}(x)  \xi \left(\L_{1} \right)^\g$, getting
\[
  m \bbE_0\left[\xi\left(\L_{1}\right)^\g\right] = \int_\cN  d\bbP(\xi) \sum_{x \in\xi\cap \L_{1/2} } \xi\left(x+\L_{1}\right)^\g\ge \bbE[\xi(\L_{1/2})^{1+\g}] = \rho_{1+\g}\,.
\]
 As (i) implies that the above l.h.s. is bounded, we get (ii).
\end{proof}

\section{Fundamental region: technical lemmas}\label{sec_fund_region}
In this section we investigate the so--called  fundamental region  ${\rm D}(x|\xi)$ introduced in Section~\ref{fund_reg} (see Definition~\ref{def_fund_reg}).
%In this section we investivate the so--called  fundamental region  ${\rm D}(x|\xi)$ associated to a vertex $x\in \xi $ of a Delaunay triangulation and we give sufficient conditions to localize it within some ball. 

The fundamental region will be used  in the next section to estimate the degree  $\deg_{\rmdt(\xi)}(x)$ of the vertex $x$  in the Delaunay triangulation $\rmdt(\xi)$, as done in \cite{Zu} and then in \cite[Section 11]{Ro}, but with a different construction. The results in this section are purely geometric (i.e. no probability appears) and  are for a fixed configuration $\xi\in \cN_{\rm pol}$.
%%%%%%%%%%%%%%%%%%%%%%%%%%%%%%%%%%%
%\begin{Definition}[Fundamental region]  \label{def_fund_reg} Given  $\xi \in \cN_{\rm pol}$  and  $x\in \xi$, the \emph{fundamental region}  ${\rm D}(x|\xi)$ of $x$ is given by the union of the closed balls centered at $v$ and of radius $|v-x|$, where $v$ varies among the vertices $v$ of the cell  $\rmVor(x|\xi)$ (see Figure~\ref{fig_fund_reg}).
%\end{Definition}
%%%%%%%%%%%%%%%%%%%%%%%%%%%%%%%%%%%
%\begin{figure}[ht]
%\includegraphics[scale=0.20]{prova.png}
%\caption{The black points are the points of $\xi$ ($d=2$). The grey region represents the Voronoi cell of $x$. The fundamental region of $x$ is the union of the balls.}\label{fig_fund_reg}
%\end{figure}
%  %%%%%%%%%%%%%%%%%%%%%%
 \begin{Lemma}\label{lemma_FR_1}
 Let $\xi\in \cN_{\rm pol}$ and $x\in \xi$. Then the points in $\xi$ adjacent to $x$ in the Delaunay triangulation ${\rm DT}(\xi)$  belong to $D(x|\xi)$.
 \end{Lemma}
 The above fact is well known, we give a short proof for completeness.
 \begin{proof}
 Let $a\in \xi$ be  adjacent to $x$ in  ${\rm DT}(\xi)$, i.e. $\{a,x\}\in \cE_{\rm DT}(\xi)$. Then, by definition of the Delaunay triangulation, the Voronoi cells     ${\rm Vor}(x|\xi)$ and ${\rm Vor}(a|\xi)$ share a $(d-1)$--dimensional face $F$. Let us take a vertex $v$  of ${\rm Vor}(x|\xi)$ belonging to $F$  (it exists as $\xi \in \cN_{\rm pol}$). The points of $F$ are  at the same distance from $x$ and $a$, thus implying that  $|v-a|=|v-x|$. This proves that $a\in B_{|v-x|}(v)  \subset D(x|\xi)$.
  \end{proof}

%We set 
%\be\label{setI} I:= \{ z \in \bbZ^d\,:\, |z|_\infty =d\}\,.
%\en
%Given $z \in \bbR^d$ and $\ell,R>0$, we define (see Figure \ref{fig_scudo})
%\begin{equation}
%    \label{cubetti}
%    K_\ell(z):=z\ell+\big[-  \ell/2, \ell/2\bigr]^d=\L_{\ell/2}(z \ell)\,.
%\end{equation}
%
%\begin{figure}[ht]
%    \centering
%   \includegraphics[scale=0.3]{fig_scudo.png}
%    \caption{In grey: the box $K_\ell (0)$ and the shield of boxes $K_\ell(z)$ with $z\in I$ ($d=2$). \club \verde{cambiare figura, non serve $K_\ell(0))$}}
%    \label{fig_scudo}
%\end{figure}

%
%\begin{figure}
%%\includegraphics[scale=0.3]{FT_fig1.pdf}
%\caption{\rosso{\bf \club Sostituire il disegno a mano con una veras figura} In grey: the box $K_\ell (0)$ and the shield of boxes $K_\ell(z)$ with $z\in I$, $d=2$}\label{fig_scudo}
%\end{figure}

%Moreover we set $I:\left\{z \in \bbZ^d \> : \> |z|_\infty=d \right\}$. 

\begin{Lemma}
\label{lemcube}
If $\ell>0$ and 
 $B$ is a ball of radius at least $3 \ell d^2$ with   $0 \in \partial B$, then  there exists $z \in I$  such that $K_\ell(z) \subset \mathring B$.
\end{Lemma}
Above, and in what follows, $\mathring A$ denotes the interior of $A$ given $A\subset \bbR^d$.
\begin{proof} 
Without loss of generality   we take $\ell=1$. We set $K(z):=K_1(z)$ and consider 
 a closed ball $B=B_R(a)$ with $0 \in \partial B_R(a)$, i.e. $|a|=R$, and $R\geq 3 d^2$.

  We set $ \D:=\cup_{z\in\bbZ^d:|z|_\infty \leq d} K(z) = \L_{d+ 1/2}$  ($\D$ corresponds to the big box in Figure \ref{fig_fund_reg}-(right)).  
We note that  $a \notin \D$ since $|a|=R\geq 3 d^2$, while all points of $\D$ have norm at most $\sqrt{d}\, (d+1) $
as $\D  \subset \L_{d+1}$ (we have $3 d^2>\sqrt{d}\, (d+1)$ for all $d\geq 1$).

We call $y$ the intersection of $\partial \D$ with the segment connecting  $a\not \in \D$  with $0$  and we take   $z\in I$ such that $y \in K(z)$. We claim that $K(z) \subset \mathring B_R(a)$.
To this aim, let $u \in K(z)$.  As $u,y\in K(z)$, we have $|u-y|\leq \sqrt{d}$ and therefore $ |a-u|\le |a-y|+\sqrt{d}$. On the other hand,
by definition of $y$, we have
$    |a-y|=|a|-|y|=R-|y|$.
We now observe that 
$ |y|> d$ as $y \in \partial \D=\partial\L_{d+1/2}$ As $    |a-y|=|a|-|y|=R-|y|$  we get   $|a-y| < R-d$. Combining this bound with $ |a-u|\le |a-y|+\sqrt{d}$, we get $|a-u|<R$. Since this holds for all $u \in K(z)$, we conclude that 
 $K(z)\subset \mathring B_R(a)$.
\end{proof}

By using  Lemma \ref{lemcube} we get the following:
\begin{Lemma}
\label{star}
 Let $\xi \in \cN_{\rm pol}\cap \cN_0$ and $\ell>0$. Suppose that $\xi \cap K_\ell(z)\neq \emptyset$ for any $z \in I$. Then   $D( 0 |\xi)\subset B_{6 \ell d^2}(0)$.
\end{Lemma}
%%%%%%%
\begin{proof} As $\xi \in \cN_0$ we have $0\in \xi$.  By  definition of the  fundamental region we just need to show that $B:=B_{|v|} (v) \subset 
B_{6 \ell d^2}(0)$ for any vertex $v$ of the cell ${\rm Vor}(0|\xi)$. As $0\in \partial B$, 
 it  is then enough to show that the radius of $B$ is at most $3 \ell d^2$. Suppose by contradiction that  the radius is more than  $3 \ell d^2$. Then,  by Lemma \ref{lemcube}, there would exist $z \in I$  such that $K_\ell(z) \subset \mathring B$. Since by hypothesis $\xi \cap K_\ell(z)\neq \emptyset$, we would conclude that $\xi \cap  \mathring B\neq \emptyset$. Hence, there would be a point $y\in \xi \cap \mathring B$. %  B_{|x-v|} (v) $.
By definition of $B$ this implies that $|y-v|<|v|$. As $v$ belongs to a face of the polytope ${\rm Vor}(0|\xi)$ shared with another polytope ${\rm Vor}(a|\xi)$, it must be $|a-v|\leq |y-v|$ otherwise  $v$ could not belong to ${\rm Vor}(a|\xi)$  (recall that $y\in \xi$). On the other hand, we know that   $|y-v|<|v|$ and this allows  to conclude that $|a-v|< |v|$. At this point we get a contradiction as the points as $v$ belonging to ${\rm Vor}(0|\xi) \cap  {\rm Vor}(a|\xi)$ are at equal distance from $x$ and $a$. 
 \end{proof}
 At this point it is trivial that the implications in \eqref{armand_zero} are an immediate consequence of 
 Lemmas~\ref{lemma_FR_1} and~\ref{star}.
%%%%%%%%%%%%%%%%%%%%%%%%%%%%%%%%%%%%%%%%%%%%%%%%%%%%%%%%%%%%%%%%%%%%%%%%%%%%%%%%%%%%%%%%%%%%%%%%%%%%%%%%%%%%%%%%%%%%%%%%%%%%%%%%%%%%%%%%%%%%%%%%%%%%%%%%%%%%%%%%%%%%%%%%%%%%%%%%%%%%%%%%%%%%%%%%%%%%%%%%%%%%%%%%%%%%%%%%%%%%%%%%%%%%%%%%%%%%%%%%%%%%%
\section{Finiteness of the expectation $\bbE_0\big[  \sum_{x: x \sim 0}    |x|^\z\big]$}\label{sec_aspetto}

 % \section{Proof of Proposition \ref{prop_suff} and Theorem \ref{teo1}}

The results presented in this section do not involve the conductance field. One could forget $\cP$ and think simply that  $\bbP$ is the law  of  a  stationary  simple point process  on $\bbR^d$  with finite and positive intensity $m=\bbE[\xi([0,1]^d)]$ 
 and that $\bbP(\xi\not=\emptyset)=1$.

\medskip

  We  introduce a family of events, in part inspired by the construction in \cite[Section~2]{Zu} and \cite[Section~11]{Ro}. 
Recall definitions~\eqref{setI} and \eqref{cubetti}, i.e.~$I:=\{z\in \bbZ^d\,:\,|z|_\infty =d\}$  and   $K_\ell(z):=z\ell+\big[-  \ell/2, \ell/2\bigr]^d=\L_{\ell/2}(z \ell)$.  Fixed  $\b>1$ (below we will take $\beta$ large enough), we define the following balls and boxes:
\be
\G^n:=B_{6\b^nd^2}(0)\,, \qquad K^n(z):=K_{\b^n}( z)=\L _{\b^n/2}(\b ^n z)  \qquad \forall n \in \bbN\,.
\en 
We point out that $K^{n-1}(z) \subset \G^n$, for any $n \ge 1$ and for any $z \in I$. In fact if $x \in K^{n-1}(z)$, then $|x|_\infty \leq \b^{n-1}(d+\frac{1}{2})$, hence $|x|\leq \sqrt{d} \b^{n-1}(d+\frac{1}{2})<6 \b^n d^2$.

We also define the  sequences of Borel sets $A_0,A_1,\dots$ and $T_0, T_1, \dots$ in $\cN$ :
\begin{align*}
& A_n:=\cap_{z\in I }\bigl\{\xi\,:\,\xi(K^n(z))>0\bigr\} \qquad \forall  n \in \bbN\,, \\
& T_0:= A_0 \,,\\
&  T_n:=A_n\setminus A_{n-1} \qquad \forall n \in \bbN_+\,.
\end{align*}
By \eqref{armand_zero} we have
%As a byproduct of Lemma \ref{lemma_FR_1} and Lemma  \ref{star} we get the following implications:
\begin{multline}\label{armand}
\xi   \in \cN_{\rm pol}\cap \cN_0\cap A_n
 \Longrightarrow   \{x\,:\, x\sim 0\}\subset D(0|\xi) \subset \G^n\\
  \Longrightarrow \begin{cases}
  {\rm deg}_{\rmdt (\xi)} (0) \leq \xi(\G^n)\\
  \max\{ |x|:x\sim 0\}\leq 6 \b^n d^2
  \end{cases}
  \,.
\end{multline}
%%%%%%%%%%%%%%%%%%%%%%%%%
Recall that $\bbP_0$ denotes the Palm distribution associated to  $\bbP$. Below we denote by $\bbE_0[\,\cdot\,]$ the expectation w.r.t. $\bbP_0$.
\begin{Lemma}\label{lemma_totale}
It holds 
 %If  $\rho_p<+\infty$ for some $p>1$, then 
   $\bbP_0( \cup _{n\geq 0} A_n) = \bbP_0(\cup _{n \geq 0} T_n)=1$.
\end{Lemma}
%%%%%%%%%%%%%%%%%%%%%%%%%
\begin{proof}
Trivially, $\cup _{n\geq 0} A_n= \cup _{n \geq 0} T_n$. To prove that  $\bbP_0( \cup _{n\geq 0} A_n)=1$, it is enough   to show that $\lim_{n\to \infty} \bbP_0(A_n)=1$. By a union bound we get
\be\label{pera1} 
\bbP_0( A_n^c) \leq \sum_{z\in I} \bbP_0\bigl( \xi ( K^n(z))=0\bigr) \leq |I| \sup_{z\in I} \bbP_0\bigl( \xi ( K^n(z))=0\bigr) \,.
\en
Given  $x\in \L_{1/2}$, the box $K^n(z)+x= \L_{\b^n /2}(\b^n z)+x$ contains  the box  $\L_{(\b^n-1) /2}(\b^n z)$.  Hence, if    
   $\xi(K^n(z)+x)= \t_x\xi( K^n(z))=0$,  then $\xi( \L_{(\b^n-1) /2}(\b^n z) )=0$.
As a consequence, by
 Campbell's formula \eqref{campbell1}, we get\footnote{Recall that ${\rm E}[X,A]:={\rm E}[X\mathds{1}_A]$.}
 \be\label{pera2}
\begin{split}
m \bbP_0\bigl( \xi ( K^n(z))=0\bigr) 
 & = 
  \int_{\cN} 
 d \bbP(\xi) \int _{\L_{1/2}} d\xi(x) \mathds{1} _{\{ \t_x\xi( K^n(z))=0\}}\\
 & \leq 
 \bbE\left[
 \xi\left (\L_{1/2}\right) ,\xi\left( \L_{(\b^n-1) /2}(\b^n z)\right)=0
 \right]\,.
%\\  & \leq  \bbE\left[ \xi\left ( \L_{1/2} \right )^p
 %\right]^{1/p} \bbP\left(\xi\left( \L_{(\b^n-1)/2}\right)=0\right)^{1/q}
  \end{split}
\en
 Suppose for the moment that $\bbP$ were also ergodic. Then by the ergodic theorem (see \cite[Chapter~12]{DV} for regions centered at the origin and \cite[Proposition~3.1]{F_hom} for the general case as in the application below), fixed $z\in I$,  for  $\bbP$--a.a. $\xi$  it holds 
\[ \lim_{n \to \infty} \xi\left( \L_{\b^n /4}(\b^n z)\right) / (\b^n/2)^d = m  \,. \] 
This implies that 
\be\label{tony} \lim_{n\to +\infty}\mathds{1}\left(\xi\left( \L_{(\b^n-1) /2}(\b^n z)\right)=0\right)=0 \text{ for $\bbP$--a.a. $\xi$}\,.
\en 
If  $\bbP$  were not ergodic, one could instead represent  $\bbP$  as a convex combination of stationary ergodic probability measures on $\cN$ with finite positive  intensity (as $\bbP(\xi\not=\emptyset)=1$), and the conclusion  \eqref{tony} would be the same. Having \eqref{tony} and by applying the dominated convergence theorem (recall that $\xi\mapsto \xi\left (\L_{1/2}\right)$ belongs to $L^1(\bbP)$ since $\bbP$ has finite intensity),  we conclude that the expectation in the last term of \eqref{pera2}  goes to $0$ as $n\to +\infty$. Coming back to \eqref{pera1} we then get $\lim_{n\to \infty} \bbP_0(A_n)=1$. 
\end{proof}

\begin{Proposition}\label{cuore}  Given $\z\geq 0$, it holds  $\bbE_0\big[  \sum_{x \sim 0}    |x|^\z\big]<+\infty$ 
 in both the following two cases: 
\begin{itemize}\item[(C1)] 
For some $\bar\a>0$ and $C_0>0$ it holds
\be\label{cond_V'''}
\sup_{x\in \bbR^d } \bbP_0 \left( \xi( \L_{\ell}(x)  ) =0\right) \leq C_0\,  \ell ^{- \bar\a} \qquad \forall \ell>0 \,.\en
In addition, it holds    $\rho_\g<+\infty$ for some $\g>2$ such that 
 $\bar \a > (d+\z) \frac{\g-1}{\g-2}$.
 \item[(C2)]  $\rho_2<+\infty$  and $\bbP$ has finite range of dependence.
 \item[(C3)] $\rho_2<+\infty$, $\bbP$ has positive association and Condition $C(\a)$ holds for some $\a>d+\z$.
\end{itemize}
\end{Proposition}
%%%%%%%%%%%%%%
%\subsection{Proof of Theorem \ref{teo1}}
\begin{proof}
  Below  constants $c,C,\dots
    $ can change from line to line (they are finite, positive  and do not depend on $n$,  while they can depend on the other parameters).
Due to \eqref{armand} and Lemma \ref{lemma_totale}, we can estimate
\begin{equation}
    \label{parte1a}
    \begin{split}
         \bbE_0\Big[ &\sum_{x \sim 0} |x|^\z \Big] \le  \bbE_0 \Big[  {\rm deg}_{\rmdt (\xi)} (0)  \, \max_{x \sim 0} |x|^\z \Big] 
           =  \sum_{n=0}^\infty \bbE_0 \Big[  {\rm deg}_{\rmdt (\xi)} (0)  \, \max_{x \sim 0} |x|^\z , T_n\Big]\\
%          \biggl[ \biggl( Q \max_{x \sim 0} |x|^k \biggr)^2 ; {T_n}\biggr]\\
        & \le \sum_{n=0}^\infty \bbE_0 \Big[  \xi(\G^n) (6d^2)^\z\b^{n\z} , {T_n}\Big]= c  \sum_{n=0}^\infty \b^{n \z} \bbE_0 \Big[  \xi(\G^n) , {T_n}\Big]\,.
             \end{split}
           \end{equation}
    Due to Lemma~\ref{prop1} each term in the last series is finite as $\rho_2<+\infty$.
       
      When $n\geq 1$, then  $T_n\subset A_{n-1}^c= \cup_{z\in I} \bigl\{ \xi ( K^{n-1}(z))=0\}$. Hence, by a union bound,  for $n\geq 1$ we have
\be\label{parte2a}
\bbE_0 \big[  \xi(\G^n) , {T_n}\big]\leq \sum_{z\in I} \bbE_0 \big[  \xi(\G^n) ,   \xi ( K^{n-1}(z))=0\big]\,.
\en

\medskip

\noindent
{\bf Case   (C1)}.
We assume  that
case (C1)  occurs. We  define $q$ as the exponent conjugate to $p:=\g-1>1$.
  Take $n\geq 1$.
Then for $z\in I$
\be\label{parte3a}
\bbE_0 \left[  \xi(\G^n) ,   \xi ( K^{n-1}(z))=0\right]\leq \bbE_0 \left[  \xi(\G^n) ^p\right]^{1/p} \bbP_0\left(   \xi ( K^{n-1}(z))=0\right)^{1/q}\,.
\en
By Lemma \ref{prop1} we have 
\be\label{parte4a} 
\bbE_0 \left[  \xi(\G^n) ^p\right]^{1/p}\leq
C  \rho_{1+p}^{1/p} \b^{nd}=C  \rho_\g^{1/p} \b^{nd} \,,
\en
while by \eqref{cond_V'''} (recall that  we have $ K^{n-1}(z)=\L_{\b^{n-1}/2}(\b^{n-1} z)$ and $|z|_\infty =d$) 
\be\label{parte5a}
\bbP_0\left(   \xi ( K^{n-1}(z))=0\right)^{1/q}\leq C  (\b^{n-1})^{- \bar\a  /q}\,.
\en
By combining \eqref{parte1a},  \eqref{parte2a}, \eqref{parte3a}, \eqref{parte4a}, \eqref{parte5a} and using that $\rho_\g<+\infty$
we get 
\be\label{luna}
\bbE_0\Big[ \sum_{x \sim 0} |x|^\z \Big] \le  c  \sum_{n=0}^\infty \b^{n (\z+ d- \bar\a /q)} \,.
\en
Then the above series is finite since  $\frac{{\bar{\a} }}{q}= \bar \a\frac{\g-2}{\g-1}>\z+d$. 

\medskip

\noindent 
{\bf Case (C2)}. We now assume  that case (C2)  occurs. Let $n\geq 2 $.
Since  for $z\in I$
\[ 
\L _{\b^{n-2}/2 }   ( \b^{n-1} z) \subset   \L _{\b^{n-1} /2}   ( \b^{n-1} z) =K^{n-1}(z)\,,
\] we   can write
 \begin{multline}\label{kievkiev}
 \bbE_0 \big[  \xi(\G^n) ,   \xi ( K^{n-1}(z))=0\big]\leq \\
 \bbE_0 \left[  \xi\left(\G^n\setminus   \L _{\b^{n-1} /2}   ( \b^{n-1} z)  \right) ,   \xi \left( \L _{\b^{n-2} /2}   ( \b^{n-1} z) \right)=0\right]\,.
 \end{multline}
 By Campbell's formula \eqref{campbell0} the right expectation in  \eqref{kievkiev} equals
 \be\label{bici}
 \frac{1}{m}
 \bbE\Big[
  \int_{\L_{1/2}} d\xi(x)
    \xi\left(x+[\G^n\setminus   \L _{\b^{n-1} /2}   ( \b^{n-1} z) ] \right)  \mathds{1}\left( \xi \left(x+ \L _{\b^{n-2} /2}   ( \b^{n-1} z) \right)=0\right) 
  \Big]\,.
\en

The two sets $\G^n\setminus   \L _{\b^{n-1} /2}   ( \b^{n-1} z)$ and  $\L _{\b^{n-2} /2}   ( \b^{n-1} z)
$ have uniform distance lower bounded by $(\b^{n-1} -\b^{n-2})/2$.
 In particular, by fixing $\beta $ large enough, 
this lower bound can be made arbitrarily large. 
 Moreover, always for $\b$ large, for $n\geq 3$,  for all $x\in \L_{1/2}$ the box  $x+ \L _{\b^{n-2} /2}  $ contains the box $\L _{\b^{n-3} /2} $.  In particular, \eqref{bici} (and therefore also \eqref{kievkiev}) can be upper bounded by 
$
\frac{1}{m}
 \bbE\Big[ F(\xi) G(\xi)\Big]$, where
 \begin{align*}
 & F(\xi):=
  \int_{\L_{1/2}} d\xi(x)
    \xi\left(x+[\G^n\setminus   \L _{\b^{n-1} /2}   ( \b^{n-1} z) ] \right)  \\
    & G(\xi):= \mathds{1}\left( \xi \left( \L _{\b^{n-3} /2}   ( \b^{n-1} z) \right)=0\right) 
 \,.
\end{align*}
Note that, up to here, we have never used the assumption that $\bbP$ has finite range of dependence (this observation will be used for case (C3)).
It is easy to check that, for $\b$ large, $F$ and $G$ are determined by  $\xi_{|A}$ and $\xi_{|B}$, with $A$ and $B$ sets in $\bbR^d$ with Euclidean distance larger than $L$. We can then use that $\bbP$ has range of dependence  smaller than $L$ to conclude (thanks also to \eqref{campbell0} and the stationarity of $\bbP$) that 
%\end{document}
\[\frac{1}{m}\bbE[FG]=\frac{1}{m}\bbE[F]\bbE[G]\leq 
 \bbE_0 \left[  \xi\left(\G^n\right) \right]
 \bbP\left( 
   \xi \left( \L _{\b^{n-3} /2}    \right)=0\right)\,.
\]
% 
 %for $n\geq 2$ we can use the finite range dependence of $\bbP_0$.
%, which follow from our hypotheses and from Proposition \ref{prop_FR}.
 We then get from \eqref{kievkiev} that for $n\geq 3$
 \be\label{kievkiev_1}
 \bbE_0 \big[  \xi(\G^n) ,   \xi ( K^{n-1}(z))=0\big]\leq 
  \bbE_0 \left[  \xi\left(\G^n\right) \right]
 \bbP\left( 
   \xi \left( \L _{\b^{n-3} /2}    \right)=0\right)\,.
 %\bbE_0 \left[  \xi\left(\G^n    \right) \right]\bbP_0\left( 
  % \xi \left( \L _{\b^{n-2} /2}   ( \b^{n-1} z) \right)=0\right)\,.
 \en
 The first expectation in the r.h.s.  of \eqref{kievkiev_1} equals $C \rho_2 \b^{n d}$, while  by Proposition~\ref{prop_FR}  we can bound  the probability 
 in the r.h.s. by  $\exp\{- c' \b^{d(n-3)} 2^{-d} \}$.
 Combining the above estimates with \eqref{parte1a}  and  \eqref{parte2a},   we then have
%\begin{equation}
%       \begin{split}
%         \bbE_0\Big[ \sum_{x \sim 0} |x|^\z \Big] &\le \rosso{c+}c \sum_{\rosso{n=3}}^\infty \b^{n \z}\sum_{z\in I}  \bbE_0 \left[  \xi\left(\G^n    \right) \right]
% \bbP_0\left( 
%   \xi \left( \L _{\b^{n-2} /2}   ( \b^{n-1} z) \right)=0\right) \\
%   & \leq c' \sum_{n=0}^\infty \b^{n (\z+d-\bar\a )  }\,.
%         \end{split}
%         \end{equation}
 \begin{equation*}
   \begin{split}
         \bbE_0\Big[ \sum_{x \sim 0} |x|^\z \Big] &\le c+c \sum_{n=3}^\infty \b^{n \z} \bbE_0 \left[  \xi\left(\G^n    \right) \right]
 \bbP\left( 
   \xi \left( \L _{\b^{n-3} /2}   \right)=0\right) 
   \\&\leq c + C\sum _{n=3}^\infty \beta^{n(\z+d)}\exp\{- c' \b^{d(n-3)} 2^{-d} \} <+\infty\,.
         \end{split} 
          \end{equation*}

\medskip

\noindent 
{\bf Case (C3)}. We now assume  that case (C3)  occurs. Let $n\geq 3 $ and  $z\in I$. Due to the discussion for Case (C2), we have that \eqref{kievkiev} is upper bounded by $\frac{1}{m}\bbE[F(\xi) G(\xi)]$. Let   $A$ be the set of all points having uniform distance at most $1$ from $\G^n=
B_{6 \beta^n d^2}(0)   $ and let $B:=\L _{\b^{n-3} /2}(\b^{n-1} z)  $.    We can upper bound
\begin{equation*} 
\begin{split} F(\xi) G(\xi) & \leq  \xi (\L_{1/2} )\xi (A\setminus B) \mathds{1}( \xi(B)=0)\\
&  =
\Big( \xi (\L_{1/2} )^2 +  \xi (\L_{1/2} )\xi (A\setminus (B\cup \L_{1/2}) )\Big)  \mathds{1}( \xi(B)=0) 
\end{split} \end{equation*}
 Since the three sets $\L_{1/2}$, $A\setminus (B\cup \L_{1/2}) $ and $B $ are disjoint, by definition of positive association (use that $z\mapsto - \mathds{1}(z=0)$ is weakly increasing on $\bbR_+$) we have 
 \[
 \bbE[FG] \leq \bbE\big[ \xi (\L_{1/2} )\xi (A\setminus B )\big]\bbP( \xi(B)=0)\,.
 \]
Since $\rho_2<+\infty$ we can bound the expectation  in the r.h.s. by  $C   \b^{nd}$  due to the Cauchy-Schwarz inequality and afterwards Lemma~\ref{lem}. We can bound the void probability in   the r.h.s.  by Condition $C(\a)$ and the stationarity of $\bbP$. At the end we obtain that   $\eqref{kievkiev}\leq m^{-1} \bbE[FG]\leq C'  \b^{n(d-\a) }$. Coming back to \eqref{parte2a} we conclude that 
\be
 \bbE_0\Big[ \sum_{x \sim 0} |x|^\z \Big]  \leq c+ c\sum_{n=3}^\infty \beta^{n( \z+d-\a)}\,.
\en
Trivially the above series is finite whenever $\a>\z+d$.
      \end{proof}
 
 We conclude by showing that Condition C($\a$) implies Condition \eqref{cond_V'''} for a suitable $\bar\a$:
\begin{Lemma}\label{implicazione} 
Suppose that Condition C($\a$)  holds for some $\a>0$ and that $\rho_\g<+\infty$ for some $\g>1$. Then  we have 
\be\label{cond_V'}
\sup_{x\in \bbR^d} \bbP_0 \left( \xi( \L_{\ell}(x)  ) =0\right) \leq C_0\,  \ell ^{- \a/\g'}  \qquad \forall \ell>0
\en
for  a suitable constant  $C_0>0$,  where $\g'=\frac{\g}{\g-1}$ is the exponent conjugate to $\g$.
 \end{Lemma} 
%%%%%%
%\verde{dire che constant $C,C_0,..$ possono can change from line to line}
\begin{proof} The bound in \eqref{cond_V'} is trivially true for $\ell<1$ when $C_0\geq 1$, hence we can restrict to  $\ell \geq1$.  Given $x\in \bbR^d$,
 by Campbell's formula \eqref{campbell1} with $f(z,\xi):= \mathds{1} ( z\in \L_{1/2}) \mathds{1}(\xi(\L_\ell(x))=0)$,   we have
\be \label{deca0}
\begin{split}
\bbP_0\left( \, \xi\big(\L_\ell(x)\bigr) =0 \, \right)& = \frac{1}{m } \int d\bbP (\xi) \int _{\L_{1/2}} d \xi (z) \mathds{1}\left( \, \xi \big(\L_\ell (x+z)\big)=0\,\right)\\
& \leq  \frac{1}{m } \int d\bbP (\xi) \int _{\L_{1/2}} d \xi (z)\mathds{1}\left( \, \xi \big(\L_{\ell/2} (x)\big)=0\,\right)\\
&= \frac{1}{m} \bbE\left[ \xi (\L_{1/2}) ,  \xi \big(\L_{\ell/2} (x)\big)=0\right]\,.
\end{split}
\en
In the above inequality we have used that $\L_{\ell/2} (x)  \subset  \L_\ell (x+z)  $ for any $z\in \L_{1/2}$ as $\ell \geq 1$.

Finally   we obtain  
\be\label{deca1bis}
 \bbE\left[ \xi (\L_{1/2}) ,  \xi \big(\L_{\ell/2} (x)\big)=0\right]
\leq \bbE\left[ \xi (\L_{1/2}) ^\g\right]^{1/\g} \bbP\left( \xi \big(\L_{\ell/2}\big)=0  \right)^{1/\g'}\leq C \ell^{- \a  /\g'}
\en
for a suitable constant $C>0$. Indeed, the first inequality follows from  H\"older's inequality and  the stationarity of $\bbP$, while the second one follows from Lemma~\ref{lem} and     Condition C$(\a)$  
By combining \eqref{deca0} and \eqref{deca1bis} we can conclude. 
 \end{proof}
 
 As an  immediate consequence of the above lemma,  (take $\bar \a= \a/\g'=\a(\g-1)/\g $ in  \eqref{cond_V'''}) we get:
  \begin{Corollary}\label{cor_cuore} Case (C1) in Proposition~\ref{cuore} occurs if 
 $\rho_\g<+\infty$ and Condition~$C(\a)$  is verified for some $\g$ and $\a$ such that $\g>2$ and $\a>(d+\z)\frac{\g}{\g-2}$.
  \end{Corollary}
 %%%%%%%%%%%%
 
 %%%%%%%%%%%%%%%%%%%%%%%%%%%%%%%%%%%%%%%%%%%%
\section{Finiteness of the expectation $\bbE_0\big[ \big( {\rm deg}_{{\rm DT}(\xi)}(0)\big)^p\big]$}\label{sec_aspettoV}

As in the previous section, the results presented in this section do not involve the conductance field. One could forget $\cP$ and think simply that  $\bbP$ is the law  of  a  stationary  simple point process  on $\bbR^d$  with finite and positive intensity $m=\bbE[\xi([0,1]^d)]$ 
 and that $\bbP(\xi\not=\emptyset)=1$.
We use \eqref{armand} and proceed as in the previous section.

\begin{Proposition}\label{cuoreV}  Given $p\in [1,\infty)$, it holds  $\bbE_0\big[ \big( {\rm deg}_{{\rm DT}(\xi)}(0)\big)^p\big]$ 
 in  the following  cases:
\begin{itemize}\item[(C1)] 
For some $\bar\a>0$ and $C_0>0$  the bound \eqref{cond_V'''} is satisfied. 
In addition,    $\rho_\g<+\infty$ for some $\g>p+1$ and
 $\bar\a >   \frac{dp (\g-1)}{\g-1-p}
$.
 \item[(C2)]   $\rho_{1+p}<+\infty$  and $\bbP$ has finite range of dependence.
  \item[(C3)]  $\rho_{1+p}<+\infty$, $\bbP$ has positive association and Condition $C(\a)$ holds for some $\a>dp$. \end{itemize}
\end{Proposition}
%%%%%%%%%%%%%%
%\subsection{Proof of Theorem \ref{teo1}}
\begin{proof}
 Due to \eqref{armand} and Lemma \ref{lemma_totale}, we can estimate
\be\label{ottino}
 \bbE_0 \Big[  \big( {\rm deg}_{\rmdt (\xi)} (0)\big)^p  \Big] 
           =  \sum_{n=0}^\infty \bbE_0 \Big[ \big( {\rm deg}_{\rmdt (\xi)} (0) \big)^p , T_n\Big]\\
%          \biggl[ \biggl( Q \max_{x \sim 0} |x|^k \biggr)^2 ; {T_n}\biggr]\\
         \le \sum_{n=0}^\infty \bbE_0 \Big[  \xi(\G^n) ^p , {T_n}\Big]\,.    
         \en
        Due to  Lemma \ref{prop1} and since $\rho_{1+p}<+\infty$,  all terms in the last series are finite.
   As for \eqref{parte2a}    for  $n\geq 1$ 
         we have
\be\label{parte2aV}
\bbE_0 \big[  \xi(\G^n) ^p, {T_n}\big]\leq \sum_{z\in I} \bbE_0 \big[  \xi(\G^n)^p ,   \xi ( K^{n-1}(z))=0\big]\,.
\en

\medskip

\noindent
{\bf Case   (C1)}. We assume 
case (C1)  occurs. Recall that $\rho_\g<+\infty$ by assumption.  We write $\g= p \theta+1$. Since $\g>p+1$ we get that $\theta >1$. 
 We  define $\theta_*$ as the exponent conjugate to $\theta$.
  Take $n\geq 1$. 
Then, by H\"older's inequality,  for $z\in I$
\be\label{parte3aV}
\bbE_0 \left[  \xi(\G^n) ^p,   \xi ( K^{n-1}(z))=0\right]\leq \bbE_0 \left[  \xi(\G^n) ^{p\theta}\right]^{1/\theta} \bbP_0\left(   \xi ( K^{n-1}(z))=0\right)^{1/\theta_*}\,.
\en
By Lemma \ref{prop1} we have  \
\be\label{parte4aV} 
\bbE_0 \left[  \xi(\G^n) ^{p\theta}\right]^{\rosso{1/\theta}}\leq
C  \rho_{1+p\theta}^{\rosso{1/\theta}} \b^{nd\rosso{p}}=C  \rosso{ \rho_\g^{1/\theta} \b^{nd p} }\,,
\en
while by \eqref{cond_V'''} (recall that  we have $ K^{n-1}(z)=\L_{\b^{n-1}/2}(\b^{n-1} z)$ and $|z|_\infty =d$) 
\be\label{parte5aV}
\bbP_0\left(   \xi ( K^{n-1}(z))=0\right)^{1/\theta_*}\leq C  (\b^{n-1})^{- \bar\a  /\theta_*}\,.
\en
By combining \eqref{ottino}, \eqref{parte2aV}, \eqref{parte3aV}, \eqref{parte4aV}, \eqref{parte5aV} and using that $\rho_\g<+\infty$
we get 
\be\label{lunaV}
\bbE_0\Big[ \big( 
{\rm deg}_{\rm DT(\xi)}(0)\big)^p \Big] \le  c \beta^{\bar\a/\theta_*}  \sum_{n=0}^\infty \b^{n ( d\rosso{p} - \bar\a /\theta_*)} \,.
\en
Observe that \rosso{$\theta=\frac{\g-1}{p}$} and \rosso{$\theta_*=\frac{\g-1}{\g-1-p}$}. Then the above series in \eqref{lunaV} is finite  since
$\bar\a > d\rosso{p\theta_*= \frac{dp (\g-1)}{\g-1-p}}
$.

\medskip

\noindent 
{\bf Case (C2)}. We now assume  case (C2)  occurs.  In this case one proceeds exactly as in the proof of Proposition~\ref{cuore} for case  (C2) there.  Indeed, given $n\geq 2$, \eqref{kievkiev} becomes 
%\rosso{Let $n\geq 2 $}.
%Since  for $z\in I$ \[  \L _{\b^{n-2}/2 }   ( \b^{n-1} z) \subset   \L _{\b^{n-1} /2}   ( \b^{n-1} z) =K^{n-1}(z)\,, \] we   can write
 \begin{multline}\label{kievkievV}
 \bbE_0 \big[  \xi(\G^n)^p ,   \xi ( K^{n-1}(z))=0\big]\leq \\
 \bbE_0 \left[  \xi\left(\G^n\setminus   \L _{\b^{n-1} /2}   ( \b^{n-1} z)  \right)^p ,   \xi \left( \L _{\b^{n-2} /2}   ( \b^{n-1} z) \right)=0\right]\,.
 \end{multline}
 By Campbell's formula \eqref{campbell0} the above r.h.s.  equals
 \be\label{banff100}
 \frac{1}{m}
 \bbE\Big[
  \int_{\L_{1/2}} d\xi(x)
    \xi\left(x+[\G^n\setminus   \L _{\b^{n-1} /2}   ( \b^{n-1} z) ] \right) ^p \mathds{1}\left( \xi \left(x+ \L _{\b^{n-2} /2}   ( \b^{n-1} z) \right)=0\right) 
  \Big]\,.
  \en
  We stress that up to now we have not used that $\bbP$ has finite range of dependence.
  From the above formula and reasoning as in the derivation of 
  \eqref{kievkiev_1} 
(we take $\b$ and $n$ large and use that $\bbP$ has finite range of dependence), we get that 
 \be\label{kievkiev_1V}
 \bbE_0 \big[  \xi(\G^n)^p ,   \xi ( K^{n-1}(z))=0\big]\leq 
  \bbE_0 \left[  \xi\left(\G^n\right) ^p\right]
 \bbP\left( 
   \xi \left( \L _{\b^{n-3} /2}    \right)=0\right)\,.
 %\bbE_0 \left[  \xi\left(\G^n    \right) \right]\bbP_0\left( 
  % \xi \left( \L _{\b^{n-2} /2}   ( \b^{n-1} z) \right)=0\right)\,.
 \en
 By Lemma   \ref{prop1}
we can then estimate  the first expectation in the r.h.s. of \eqref{kievkiev_1V}
 by $C \rho_{1+p} \b^{n dp}$, while by Proposition~\ref{prop_FR} we can bound  the probability 
 in the r.h.s. by a stretched exponential. By combining the above estimates with \eqref{ottino}  and  \eqref{parte2aV} and using that $\rho_{1+p}<+\infty$,   we then have
%\begin{equation}
%       \begin{split}
%         \bbE_0\Big[ \sum_{x \sim 0} |x|^\z \Big] &\le \rosso{c+}c \sum_{\rosso{n=3}}^\infty \b^{n \z}\sum_{z\in I}  \bbE_0 \left[  \xi\left(\G^n    \right) \right]
% \bbP_0\left( 
%   \xi \left( \L _{\b^{n-2} /2}   ( \b^{n-1} z) \right)=0\right) \\
%   & \leq c' \sum_{n=0}^\infty \b^{n (\z+d-\bar\a )  }\,.
%         \end{split}
%         \end{equation}
 \be
   \begin{split}
         \bbE_0\Big[ \big( {\rm deg}_{{\rm DT}(\xi)}(0)\big)^p \Big] &\le c+c \sum_{n=3}^\infty  \bbE_0 \left[  \xi\left(\G^n    \right)^p \right]
 \bbP\left( 
   \xi \left( \L _{\b^{n-3} /2}   \right)=0\right) 
   \\
   &\leq c + C\sum _{n=3}^\infty \beta^{ndp}\exp\{- c' \b^{d(n-3)} 2^{-d} \} <+\infty\,.
         \end{split}
 \en

\medskip

\noindent 
{\bf Case (C3)}. We now assume  case (C3)  occurs. Take $n\geq 3$. By \eqref{kievkievV} and \eqref{banff100} and reasoning as for Case (C3) in the proof of Proposition~\ref{cuore} (we keep the notation used there), thanks to the positive association of $\bbP$ we get that
\be\label{calgary100}
 \bbE_0 \big[  \xi(\G^n)^p ,   \xi ( K^{n-1}(z))=0\big]\leq C \bbE\big[ \xi (\L_{1/2}) \xi (A\setminus B ) ^p\big] \bbP(\xi(B)=0)\,.
\en
\rosso{We can  bound the expectation in the r.h.s. by H\"older's inequality applied with conjugate exponents  $1+p$ and  $\frac{1+p}{p}$. Using also Lemma~\ref{lem} and that $\rho_{1+p}<+\infty$ we get 
\[
\bbE\big[ \xi (\L_{1/2}) \xi (A\setminus B ) ^p\big] 
\leq \bbE\big[ \xi (\L_{1/2})^{1+p}\big]^\frac{1}{1+p}
\bbE\big[  \xi (A) ^{1+p}\big]^\frac{p}{1+p}\leq C \beta^{n d p}\,.
\]  
To bound the probability in the r.h.s. of \eqref{calgary100} we use Condition $C(\a)$. We conclude that \eqref{calgary100} is bounded from above by 
$c \beta^{n(dp-\a)}$. Coming back to \eqref{ottino} and \eqref{parte2aV},} we get 
  that  
 \be
            \bbE_0\Big[ \big( {\rm deg}_{{\rm DT}(\xi)}(0)\big)^p \Big]    \leq c + C\sum _{n=3}^\infty \beta^{n(dp-\a) }  <+\infty
 \en
 whenever $\rho_{1+p}<+\infty$ and  $\a>dp$.
     \end{proof}

 \rosso{As an  immediate consequence of the Lemma~\ref{implicazione}  (take $\bar \a= \a/\g'=\a(\g-1)/\g $ in  \eqref{cond_V'''}) we get:
  \begin{Corollary}\label{cor_cuoreV} Case (C1) in Proposition~\ref{cuoreV} occurs if 
 $\rho_\g<+\infty$ and Condition~$C(\a)$  is verified for some $\g$ and $\a$ such that $\g>1+p$ and    $\a>\frac{dp \g}{\g-1-p}$.
  \end{Corollary}}

 %%%%%%%%%%%%

 %%%%%%%
 \section{Proof of Theorem \ref{teo1} and  Proposition \ref{prop_veloce}}\label{sec_proof_teo1}
\subsection{Proof of Theorem \ref{teo1}}\label{siria1}
   We recall that ${\rm E}_0[\cdot]$ denotes the expectation w.r.t. the Palm distribution $\cP_0$.  Given $k=0,2$, by \eqref{condizione_zero} we can bound
   \begin{multline} \label{parte0}
 {\rm E}_0 \big[ \l_k (\o)\big] = {\rm E}_0\Big[ \sum_{x \sim 0}c_{0,x} (\o)  |x|^k \Big] 
 % = {\rm E}_0\Big[ {\rm E}_0\big[ \sum_{x \sim 0}c_{0,x} (\o)  |x|^k \,|\, \hat \o\big] \Big]\\&
  =
  {\rm E}_0\Big[  \sum_{x \sim 0}
  {\rm E}_0\big[c_{0,x} (\o)  \,|\, \hat \o\big]   |x|^k\Big] 
 \\
   \leq    {\rm E}_0\Big[  \sum_{x \sim 0}   \phi(|x|)  |x|^k\Big] =  {\bbE}_0\Big[  \sum_{x \sim 0}   \phi(|x|)  |x|^k\Big] \,.%\leq C_*   \bbE_0\Big[  \sum_{x \sim 0}    |x|^\z\Big] \,.
\end{multline}
%\be \label{parte0}
%\begin{split}
% & {\rm E}_0 \big[ \l_k (\o)\big] = {\rm E}_0\Big[ \sum_{x \sim 0}c_{0,x} (\o)  |x|^k \Big] 
% % = {\rm E}_0\Big[ {\rm E}_0\big[ \sum_{x \sim 0}c_{0,x} (\o)  |x|^k \,|\, \hat \o\big] \Big]\\&
%  =
%  {\rm E}_0\Big[  \sum_{x \sim 0}
%  {\rm E}_0\big[c_{0,x} (\o)  \,|\, \hat \o\big]   |x|^k\Big] 
% \\
% &  \leq    {\rm E}_0\Big[  \sum_{x \sim 0}   \phi(|x|)  |x|^k\Big] =  {\bbE}_0\Big[  \sum_{x \sim 0}   \phi(|x|)  |x|^k\Big] \,.%\leq C_*   \bbE_0\Big[  \sum_{x \sim 0}    |x|^\z\Big] \,.
%  \end{split}
%\en
Since $\phi$ is locally bounded, there exists $C>0$ such that $\phi(|x|) |x|^k \leq C$ for all $x$ with $|x|\leq 1$.
In particular, using also Lemma  \ref{prop2} and that $\rho_2<+\infty$, 
we get 
 \be\label{parte1}\bbE_0\Big[  \sum_{x \sim 0: |x|\leq 1}   \phi(|x|)  |x|^k\Big]\leq C\, \bbE_0\big[ \xi( B_1(0))\big] <+\infty\,.\en
Due to the above bound and since $\phi(r) r^2 \leq  C_0 r^\z$ for $r\geq 1$, to prove that the last expectation in \eqref{parte0} is finite we just need to show  that $\bbE_0\big[  \sum_{x \sim 0}    |x|^\z\big] <+\infty $. This follows from Theorem~\ref{primo_int}.

\subsection{Proof of Proposition  \ref{prop_veloce}}\label{piaga} By the same arguments presented at the beginning of the previous subsection and using that $\rho_2<+\infty$, we get again \eqref{parte0} and \eqref{parte1}  and in \eqref{parte1} we can replace $|x|$, $B_1(0)$ with $|x|_\infty$, $\L_1$ respectively. As  $\phi (r) r^2 \leq C_0 r^{-d-\e}$  for $r\geq 1$, it then remains to prove that 
\be\label{riduzione_bis}
\bbE_0\big[  \sum_{x \sim 0: \rosso{|x|_\infty\geq 1}}    |x|^{-d-\e}\big] <+\infty \,.
\en
By partitioning $\bbR^d$ in cubes of size $1$,  we can upper bound the l.h.s. of \eqref{riduzione_bis} by 
\[ \bbE_0\big[  \sum_{x \in \xi : | x|_\infty\geq 1}    |x|^{-d-\e}\big]\leq 
C\sum_{z\in \bbZ^d: z\not =0 } |z|^{-d-\e} \bbE_0 [ \xi ( \L_1 (z)) ]\,.
\]  By Lemma \ref{prop1} and since $\rho_2<+\infty$, the expectations  $\bbE_0 [ \xi (\L_{1}(z) )]$ are bounded uniformly in $z\in \bbZ^d$.
This allows to conclude.

%%%%%%%%%%%%%%%%%%%%%%%%%%%%%%%%%%%%%%%%%%%%%%%%%%%%%%%%%%%%%%%%%%%%%%%%%%%%%%%%%%%%%%%%%%%%%%%%%%%%%%%%%%%%%%%
%%%%%%%%%%%%%%%%%%%%%%%%%%%%%%%%%%%%%%%%%%%%%%%%%%%%%%%%%%%%%%%%%%%%%%%%%%%%%%%%%%%%%%%%%%%%%%%%%%%%%%%%%%%%%%%%%%%%%%%%%%%%%%%%%%%%%%%%%%%%%%%%%%%%%%%%%%%%%%%%%%%%%%%%
\section{Bernoulli bond percolation on the Delaunay triangulation:  Proof of Theorem  \ref{teo_perc}}\label{sec_bond_perc}

In this section (and in the next one) we keep  Assumption (A) stated at the beginning  of Section~\ref{sec_results2}.
%  On the probability space $(\O,\cF,\cP)$ we have a  simple point process  $\O\ni \o\mapsto \hat \o\in \cN$  on $\bbR^d$ whose law $\bbP$ is stationary (hence $\bbP(\cN_{\rm pol})=1$) and has finite intensity $m:=\bbE[\xi ([0,1]^d)]$.
 Moreover, without loss of generality, we assume that $ \bbP(\{ \emptyset\})=0$ (i.e. $\hat \o$ contains some point for $\cP$--a.a.~$\o$). By the stationarity of $\bbP$,  this implies that $\bbP( \{\xi\in \cN:\sharp \xi=\infty\})=1$. 

\medskip

The proof of Theorem  \ref{teo_perc} is given at the end of this section.  We first  need  to go 
 through an intermediate process, called $\bbZ^d$\emph{-process}, as done in \cite{BB,BBQ} for percolation on Voronoi  tilings.  To this aim we fix some notation. 
We fix a  length  $R>0$ and  introduce  the boxes
\[
C:=[0,R]^d\,, \qquad   C_x:=xR+[0,R]^d \;\; \;\;\forall  x \in \bbZ^d\,.
\]
Note that the dependence of $C$ and $C_x$ on $R$ is understood in our notation. At the end we will fix the value of $R$.
Moreover, given $r>0$ and $A\subset \bbR^d$,  we set 
\be \label{telefono}B_r(A):=\big\{ x \in \bbR^d \> : \> \inf_{y \in A} \> |x-y| < r \big\}= \big\{x\in \bbR^d\,:\, \text{dist}(x,A)<r \big\}\,.
\en
Note that, when $A=\{x\}$,  $B_r(\{x\})$ is the open ball centered at $x$ with radius $r$, while $B_r(x)$ is the closed ball centered at $x$ with radius $r$. 
%%%%%%%%%%%
\begin{Definition}[Property (P)]\label{def_prop_P}
Given $(\xi,W)$  we say that a  path $\g:I \to\bbR^d$  satisfies property (P) if, for any $t \in I $, $\g(t)$ either lies in the interior of some Voronoi cell ${\rm Vor}(z|\xi)$ with $z \in \xi$ or in a $(d-1)-$dimensional face joining two Voronoi cells ${\rm Vor}(z|\xi)$ and ${\rm Vor}(z'|\xi)$ with $\{z,z' \}$ edge in $\cG(\xi, W)$ (i.e. $z,z'\in \xi$, $\{z,z'\}\in \rmdt (\xi)$ and $W_{z,z'}=1$).
\end{Definition}
We stress that in the above definition, both  $z\in \xi$ and the pair $z,z'\in\xi$ can change when changing $t$.

\begin{Definition}[$(\xi,W)$--open box $C_x$] \label{zdpro}
Given $(\xi, W) \in \cX$ and 
 $x \in \bbZ^d$  we say that the box $C_x$ is $(\xi,W)$-\emph{open} if at least one of the following two conditions is satisfied:
\begin{itemize}
    \item[(i)] there exists $y \in B_{R/4}(C_x)$ such that the open ball $B_{R/4}(\{y\})$ contains no point of $\xi$;
    \item[(ii)] there exists a continuous curve $\g:[0,1] \to  B_{R/4}(C_x)$ such that $\g(0) \in \partial C_x$, $\g(1) \in \partial B_{R/4}(C_x)$ and $\g$ satisfies property (P) introduced in  Definition \ref{def_prop_P}.
    \end{itemize}
We say that $C_x$ is $(\xi,W)$-\emph{closed} if it is not  $(\xi,W)$-open.
Moreover we define the configuration $\eta=\eta(\xi,W) \in \{0,1\}^{\bbZ^d}$ as 
\begin{align*}
	\eta_x=&\left\{
		\begin{aligned}
		0 \quad & \text{if $C_x$ is $(\xi,W)$-closed}\,, \\
		1 \quad  &\text{if $C_x$ is $(\xi,W)$-open}\,.	
		\end{aligned} 		\right.
\end{align*}
\end{Definition}
We can now define the $\bbZ^d$--process (our terminology is inspired by the one in \cite{BB,BBQ}):
\begin{Definition}[$\bbZ^d$--process and graph $\bbG(\xi,W)$] \label{def_5apr} We call   $\bbZ^d$--\emph{process} 
the site percolation on $\bbZ^d$ given by the map $\eta$ defined on the probability space $(\cX, \bbQ)$. To each configuration $\eta=\eta(\xi,W)$ we associate the graph $\bbG (\xi,W)$ with vertex set $\{x\in \bbZ^d\,:\, \eta_x=1\}$ and edges given by the   pairs $\{x,y\}$ with $x,y\in \bbZ^d$, $\eta_x=\eta_y=1$ and $|x-y|=1$.
\end{Definition}

Our interest in the $\bbZ^d$-process is due to the following simple but crucial observation:
%%%%%%%%%%%%%
\begin{Lemma}
\label{inclusion}
Given $(\xi,W)\in \cX$, if the graph $\cG(\xi,W)$ has an 
 unbounded connected component, then the same holds for the graph $\bbG(\xi, W)$.
 \end{Lemma}
 %%%%%%%%
 \begin{proof} Below we write $\cG,\bbG, C_z,\eta$ understanding their dependence from $(\xi,W)$. If  $\cG$ has an 
 unbounded connected component, then we can build an unbounded 
  % (i.e. whose support is not contained in a bounded region) 
  continuous curve $\g :[0,+\infty)\to \bbR^d$ such that  $\g$ proceeds along  segments in $\bbR^d$ connecting adjacent vertices of $\cG$. In particular,     $\g$ satisfies property (P).
   As $\g$ is unbounded, $\g$  intersects infinitely many boxes of the form $ C_z $, $z \in \bbZ^d$.   For  any such  box $C_z$    there exists  a subpath of $\g$  connecting $\partial C_z$ with $\partial B_{R/4} (C_z)$ (as $\g$ is unbounded). At a cost of a reparametrization,  this subpath can be thought of as having domain $[0,1]$. By construction it   satisfies property (P) and therefore   $\eta_z=1$. 
%This implies that every cube $C_z$ that intersects $\g$ is an open cube in the $\bbZ^d-$process.
Finally we observe   that the set   $\cZ:=\{z\in \bbZ^d \,:\, \text{$\g([0,+\infty))$ intersects $C_z$}\}$ forms a connected subset of the graph $\bbG$. To prove this claim suppose that $\g$ visits $C_z$ and then, when exiting $C_z$, it visits $C_{z'}$. Then $|z-z'|_\infty=1$.  It could be  $|z-z'|_1\not=1$,  but 
in any case we  have that $C_z\cap C_{z'}$ is contained in a sequence of intermediate boxes $C_{a^{(1)}}, C_{a^{(2)}},\dots,C_{a^{(k)}}$ such that $a^{(1)},..., a^{(k)}\in \bbZ^d$ and 
\be\label{arrosto}
|z-a^{(1)}|_1=|a^{(1)}-a^{(2)}|_1=\cdots=|a^{(k-1)}-a^{(k)}|_1=|a^{(k)}-z'|_1=1\,.
\en  Indeed, if $z,z'$ differ in exactly $k$ entries $i_1,i_2,\dots, i_k$, then it is enough to define  $a^{(j)}$, for $1\leq j \leq k$, as 
\[ 
a^{(j)}_i:=\begin{cases}
z'_i & \text{ if } i \leq i_j\,,\\
z_i & \text{ if } i>i_j\,, 
\end{cases}\qquad 1\leq i \leq d\,.
\]
Trivially, by construction, \eqref{arrosto} holds. On the other hand, we have 
\begin{align*}
& C_z\cap C_{z'}= \prod _{i=1}^d \Big( [Rz_i,R(z_i+1)]\cap [Rz'_i,R(z'_i+1)] \Big)\,, \\
& C_{a^{(j)} }= \Big(\prod_{i=1}^{i_j} [Rz'_i,R(z'_i+1)]\Big)\times   \Big(\prod_{i=i_j+1}^{d} [Rz_i,R(z_i+1)]\Big)\,,
\end{align*}
thus implying that $\big(C_z\cap C_{z'}\big) \subset C_{a^{(j)} }$.
%  for all coordinate $i$ up to the $j$--th one in which $z,z'$ are different, and equal to $z'_i$ for the remaining coordinates $i$.
As $\g([0,+\infty))\subset (C_z\cap C_z')$, we conclude that all points $z,a^{(1)},a^{(2)},\dots, a^{(k)},z'$ (which form a path in the standard lattice $\bbZ^d$) belong to  $\cZ$. This  proves  that $\cZ$ is connected. 
Since by construction $\cZ$ is unbounded (as $\g$ is unbounded), we conclude that the graph $\bbG$  has an  unbounded connected component.
\end{proof}

\begin{Lemma}\label{lemma_dipendenza}  Take  $(\xi,W)\in \cX$ and   $x \in \bbZ^d$. The knowledge of the  set $\xi \cap B_{R/2}(C_x)$ allows to verify if Condition (i) 
in Definition~\ref{zdpro} is  satisfied or not.  If Condition (i) in Definition~\ref{zdpro} is not satisfied, then the following holds:

The knowledge of the  set $\xi \cap B_{R/2}(C_x)$ allows to determine the nonempty sets of the form  $V \cap B_{R/4}(C_x)$, where $V$ is a Voronoi cell with nucleus in $\xi$. Moreover, if $V \cap B_{R/4}(C_x)\not=\emptyset$, then the nucleus of $V$  belongs to $B_{R/2}(C_x)$.
 
 The knowledge of the  set $\xi \cap B_{R/2}(C_x)$ allows to know also   the sets of the form $F\cap B_{R/4}(C_x)$, where $F$ is a $(d-1)$--dimensional face shared by adjacent Voronoi cells with nuclei in $\xi$.  In particular, calling  $\rmVor(y|\xi) $ and $\rmVor(z|\xi)$ these two cells, if we also know the 
  value of $W_{y,z}$, then we can determine if   the box $C_x$  is $(\xi,W)$--open.
\end{Lemma}

%
%\begin{Lemma}\label{lemma_dipendenza} Given $(\xi,W)\in \cX$ and   $x \in \bbZ^d$,
%the knowledge of the  set $\xi \cap B_{R/2}(C_x)$ allows to determine the nonempty sets of the form $\rmVor(y|\xi) \cap B_{R/4}(C_x)$ and the nonempty sets of the form $F\cap B_{R/4}(C_x)$, where $F$ is a $(d-1)$--dimensional face shared by adjacent Voronoi cells $\rmVor(y|\xi) $ and $\rmVor(z|\xi)$, for which necessarily   $y,z\in \xi\cap B_{R/2}(C_x)$. In particular, knowing also the values of $W_{y,z}$  with $y,z$  associated to $F$ as above, we can determine if   the box $C_x$  is $(\xi,W)$--open.
%\end{Lemma}

%\begin{Lemma}\label{lemma_dipendenza} Given $(\xi,W)\in \cX$ and   $x \in \bbZ^d$, the openness of the box $C_x=C_x(\xi,W)$, i.e.~the value of $\eta_x(\xi,W)$, is determined by  the set $\xi \cap B_{R/2}(C_x)$  (which also allows to determine if $\{y,z\}\in \cE_{\rmdt}(\xi)$ for all $y,z\in B_{R/4}(C_x)$) and the numbers   $W_{y,z}$ with $y,z \in B_{R/2}(C_x)$ and $\left\{y,z\right\}\in \cE_{\rmdt}(\xi)$. 
%\end{Lemma}
%%%%%%%%%%%%%%%%%%%%%%%%%%%%%%%%%%%%%%%%%%%%%%%%%%%%%
\begin{proof}  In what follows,   (i) and (ii) refer to the two items in Definition \ref{zdpro}. 
%We call (ii') the condition that (ii)  holds but (i) fails. Hence the box $C_x$ is $(\xi,W)$--open if either   (i) or (ii') holds. 
Trivially, to verify the validity of  (i) it is enough to know the set  $\xi\cap B_{R/2}(C_x)$. 
Let us now    suppose that, by observing  $\xi\cap B_{R/2}(C_x)$, we have  inferred that  (i) fails. Given  $ y \in B_{R/4}(C_x)$  let  $z \in \xi$ be such that  $y\in {\rm Vor}(z|\xi)$.
As (i) fails,  there are points in $\xi \cap B_{R/4} (\{y\})$. On the other hand, by definition of Voronoi cell, 
 $z$ is the point of $\xi$ closest to $y$. As a consequence, $|y-z|$  has to be smaller than ${R/4}$, and therefore $z \in B_{R/2}(C_x)$.
This observation implies that, for any $ y \in B_{R/4}(C_x)$,
\begin{equation}
\label{phob}
\begin{split}
     \left\{ z \in \xi  : \> y \in \text{Vor}(z|\xi) \right\} = \argmin_{z \in \xi \cap B_{R/2}(C_x)} |y-z|\,.
     \end{split}
\end{equation}
 In particular   the Voronoi tessellation of $\xi$ in $B_{R/4}(C_x)$ is determined by  the set $\xi\cap B_{R/2}(C_x)$: the family of nonempty subsets ${\rm Vor}( z|\xi ) \cap B_{R/4}(C_x)$ indeed equals the family of nonempty subsets ${\rm Vor}( z|\xi\cap  B_{R/2}(C_x) ) \cap B_{R/4}(C_x)$.
%$$**$$
% Take now $y,z\in B_{R/4}(C_x) $. We have that $\{y,z\}\in \cE_{\rmdt}(\xi)$ if and only if $\rmVor (y|\xi)$ and $\rmVor ( z|\xi)$ share a $(d-1)$--dimensional face $F$. On the other hand, since $B_{R/2}(C_x) $ is convex, the segment between $y$ and $z$ is contained in $B_{R/4}(C_x)$. The shared face $F$, if existing, then  passes through the middlepoint of this segment (which in the interior of $B_{R/4}(C_x)$). This implies that to decide if $\rmVor (y|\xi)$ and $\rmVor ( z|\xi)$ share a $(d-1)$--dimensional face, we just need to observe the sets ${\rm Vor}( y|\xi ) \cap B_{R/4}(C_x)$ and ${\rm Vor}( z|\xi ) \cap B_{R/4}(C_x)$. As already pointed out, to determine these sets we just need to now the set $\xi\cap B_{R/2}(C_x)$.
% $$***$$

Suppose again that (i) fails. Let $F$ be a $(d-1)$--dimensional face shared by adjacent Voronoi cells $\rmVor(y|\xi) $ and $\rmVor(z|\xi)$ and suppose that $F\cap B_{R/4}(C_x)\not =\emptyset$. By the previous discussion we know that $y,z \in  \xi\cap B_{R/2}(C_x)$. Moreover, as  $B_{R/4}(C_x)$ is open, $F\cap B_{R/4}(C_x)$ is itself a $(d-1)$--dimensional set, shared by  $\rmVor(y|\xi) \cap B_{R/4}(C_x)$ and $\rmVor(z|\xi)\cap B_{R/4}(C_x)$. In particular, to detect the above sets $F$ and the associated $y$ and $z$, we just need to observe the nonempty sets of the form $\rmVor(a|\xi) \cap B_{R/4}(C_x)$  with $a\in \xi \cap  B_{R/2}(C_x)$  (and these are determined by $\xi \cap  B_{R/2}(C_x)$ by the previous discussion). Finally, if we also know the value $W_{y,z}$ as $y,z$ vary as above, then given a 
 a continuous curve $\g:[0,1] \to  B_{R/4}(C_x)$ such that $\g(0) \in \partial C_x$ and  $\g(1) \in \partial B_{R/4}(C_x)$, we can  determine if $\g$ satisfies property (P).
%  we need to know the nonempty sets ${\rm Vor}( z|\xi ) \cap B_{R/4}(C_x)$ (and for this it is enough to know $\xi\cap B_{R/2}(C_x)$) but also the nonempty sets of the form $F\cap B_{R/4}(C_x)$, where  $F$ is a $(d-1)$--dimensional face shared by some cells ${\rm Vor}( y|\xi ) $ and ${\rm Vor}(z|\xi)$ and for such sets $F$ we need to know $W_{y,z}$.
%  If $F\cap B_{R/4}(C_x)\not=\emptyset$, then both  ${\rm Vor}( y|\xi ) $ and ${\rm Vor}(z|\xi)$ intersect $B_{R/4}(C_x)$ and therefore (by the previous discussion) $y,z\in \xi \cap B_{R/2}(C_x)$.
%  Moreover, as $B_{R/4}(C_x)$, if $F\cap B_{R/4}(C_x)\not=\emptyset$ then $F\cap B_{R/4}(C_x)\not=\emptyset$  is a $(d-1)$--dimensional sets. Hence, we just need to observe the nonempty sets of the form ${\rm Vor}( z|\xi ) \cap B_{R/4}(C_x)$ and if two of them share a $(d-1)$--dimensional face, then the corresponding Voronoi cells in $\bbR^d$ are adjancent with nuclea in $\xi\cap B_{R/2}(C_x)$ and we just need no know.
%   
%  
%   and $\{y,z\}\in \cE_{\rmdt}(\xi)$. \rosso{\club c'e' un problema}
%    If we also know the value of 
%   $W_{y,z}$ with $y,z $ as above,  %and $\left\{y,z\right\}$ edge in $\cG(\xi,W)$,
%   then  we are able to state if $\g$ satisfies  Property (P) and therefore the occurrence of (ii'). 
\end{proof}
We observe that the   $\bbZ^d$-process has finite range of dependence if the same holds for $\bbP$:
%To simplify, one derive just what we will really need for the proof of our main results.%%%%%%%%%%%%%%%%%%%%%%%%%%%%%%%%%%%%%%%%%%%%%%%%%%%%%
\begin{Lemma}\label{indipendenza}
Suppose that $\bbP$ has finite range of dependence smaller than $L$.
Let  $A,D$ be finite subsets of $\bbZ^d$ with distance at least $3+L/R$. Then the random vectors $(\eta_x)_{x\in A}$ and $(\eta_x)_{x\in D}$ are independent under $\bbQ$.
%$x,x_1,x_2,\dots, x_n$ be points in $\bbZ^d$ such that the distance of $x$ from $\{x_1,x_2, \dots, x_n\}$ is at least XXX. Then, under $\bbQ$, the random variable $\eta_x$ and the random vector $(\eta_{x_1},\eta_{x_2},...,\eta_{x_n})$ are independent.
\end{Lemma}
\begin{proof}
%Let $A_R:=B_{3R/2} (A)$ and $D_R:= B_{3R/2} (D)$.
%Since $C_x\subset xR+ [-R,R]^d$, due to Lemma \ref{lemma_dipendenza} the random vector 
% $(\eta_x)_{x\in A}$  depends from $(\xi,W)$  only in terms of  the restriction of $\xi $ to $A_R$ and of the value  $W_{x,y}$, as $x,y$ varies in $A_R$. The same  holds for the random vector  $(\eta_x)_{x\in D}$, by replacing $A_R$ with $D_R$. To get the independence of the two random vectors, we just need that ${\rm dist}(A_R,D_R)\geq L$. As  ${\rm dist}(A_R,D_R) \geq {\rm dist}(A,B)- 3R$, the claim follows.
Let $A_R:=\cup_{x\in A}  B_{R/2}(C_x)$ and $D_R:= \cup_{x\in D}  B_{R/2}(C_x)$. Due to Lemma \ref{lemma_dipendenza} the random vector 
 $(\eta_x)_{x\in A}$  depends on $(\xi,W)$  only in terms of  the restriction of $\xi $ to $A_R$ and of the value  $W_{x,y}$, as $x,y$ varies in $A_R$. The same  holds for the random vector  $(\eta_x)_{x\in D}$, by replacing $A_R$ with $D_R$. To get the independence of the two random vectors, we just need that ${\rm dist}(A_R,D_R)\geq L$. As  ${\rm dist}(A_R,D_R) \geq R \,{\rm dist}(A,D) - 3R$, the claim follows.
 \end{proof}

The above lemma will enter in the proof of the following result:
\begin{Proposition}
\label{perzdthm}
Suppose that $\bbP$ has  finite range of dependence. Then there exist positive constants $R_\ast $ and $c_*$  such that the following holds: for all $R$ and $p\in (0,1)$ such that  $R\geq  R_\ast$  and $p R^{d+1} \leq  c_*$,  all the connected components of the  graph $\bbG(\xi,W)$ are finite $\bbQ$--a.s. 
\end{Proposition}
We postpone the proof of Proposition \ref{perzdthm} to Section \ref{java}.
We can finally prove Theorem \ref{teo_perc}:

\begin{proof}[Proof of Theorem \ref{teo_perc}]
We use the same notation of Proposition \ref{perzdthm}. 
We take  $R:=R_*$ in the construction of the $\bbZ^d$--process 
  and we set  $p_*:=c_* R^{-d-1} \wedge 1$ (one could remove indeed ``$\wedge 1$" since it must be $c_* R_*^{-d-1}<1$, indeed if it were $c_* R_*^{-d-1}\geq 1$   the conclusion of Proposition~\ref{perzdthm} would hold for $p=1$ which is impossible). Then Theorem \ref{teo_perc} is an immediate consequence of  Lemma \ref{inclusion} and  Proposition \ref{perzdthm}.
  \end{proof}

%%%%%%%%%%%%%%%%%%%%%%%%%%%%%%%%%%%%%%%%%%%%%%%%%%%%%%%%%%%%%%%%%%%%%%%%%%%%%%%%%%%%%%%%%%%%%%%%%%%%%%%%%%%%%%%%%%%%%%%%%%%%%%%%%%%%%%%%%%%%%%%%%%%%%%%%%%%%%%%%%%%%%%%%%%%%%%%%%%%%%%%%%%%%%%%%%%%%%%%%%%%%%%%%%%

\section{Subcritical phase in the $\bbZ^d$-process: proof of Proposition \ref{perzdthm}}\label{java}

In this section we prove Proposition \ref{perzdthm}.
The proof  will be given only at the end (see Subsection \ref{stilton}), while before we  collect several
preliminary results.

We define the following subsets of $\cN_{\rm pol}$  (below $\overline{B_{3R/4}(C_0)}$ denotes the closure of $B_{3R/4}(C_0)$):
\begin{align*}
    \cA_1& :=\left\{\, \xi\big( B_{R/8}(\{y\})\big)>0 \;\; \forall y \in \overline{B_{3R/4}(C_0)}\, \right\}\,, \\
%    	\cA_2& :=\left\{
%		\begin{aligned}
%		&\text{every Voronoi cell intersecting } B_{R/4}(C_0) \text{ has}\\
%		& \> \text{ at most } R^{d+1} \text{ neighboring Voronoi cells} 
%		\end{aligned}
%		\right\}
%\\
 \cA_2& := \left\{ \, \sharp\{z\,:\,z\sim y\} \leq R^{d+1} \text{ for all } y \in \xi \text{ with } \rmVor(y|\xi) \cap B_{R/4}(C_0) \not =\emptyset \, \right\} \,, 
\\
 \cA_3 &:=\left\{
 \, \sharp \left\{ \text{Voronoi cells intersecting  } \partial C_0\right\}\le R^{d+1}
  \, \right\}\,.
\end{align*} 
Above, and also in the rest of this section,  events will be subsets of $\cN_{\rm pol}$, i.e. $\xi$ will vary in $\cN_{\rm pol}$ although not written explicitly.

 Below, given a Voronoi cell ${\rm Vor}(x|\xi)$ with nucleus $x$, we define its radius as 
\[ 
{\rm rad} \big({\rm Vor}(y|\xi) \big):= \max\big\{|z-y|\,:\, z \in {\rm Vor}(y|\xi) \big\}\,.
\]
\begin{Lemma}
\label{ipa1}
Suppose that $\xi\in \cA_1$. Then the following holds:
\begin{itemize} 
   \item[(i)] ${\rm rad} \big({\rm Vor}(y|\xi) \big) \le {R/8}$ for all  $y \in \xi \cap B_{5R/8}(C_0)$;
   % \item[(ii)]   ${\rm diam}({\rm Vor}(y|\xi ))\le {R/4}$ for all  $ y \in \xi \cap B_{5 R/8}(C_0)$;
    %\item[(iii)]  ${\rm Vor}(y|\xi)\subset B_{3R/4}(C_0)$ for all  $ y \in \xi \cap B_{5 R/8}(C_0)$;
    \item[(ii)] 
    $y\in B_{5R/8}(C_0)$ if $y\in \xi$ and  ${\rm Vor}(y|\xi)\cap 
     B_{R/2}(C_0)\not =\emptyset$.
\end{itemize} 
\end{Lemma}
\begin{proof}
Let us prove Item (i).
Suppose by contradiction that there exists $y \in \xi \cap B_{5R/8}(C_0)$ 
and $a\in {\rm Vor}(y|\xi)$ with $|y-a| >R/8$.
We distinguish between the  two cases: $a \in B_{3R/4}(C_0)$ and $a\not \in B_{3R/4}(C_0)$.

 If $a \in B_{3R/4}(C_0)$ then, since  $\xi \in \cA_1$, there exists $z \in \xi$  such that $|z-a|< {R/8}<|y-a|$, thus implying that   $a \notin{\rm Vor}(y|\xi)$, which is absurd.
 
   If $a \notin B_{3R/4}(C_0)$, we denote by  $\overline{y a}$  the segment joining $y$ and $a$, we  define $b$ as the intersection point of  $\overline{y a}$ and $ \partial B_{3R/4}(C_0)$ (recall that $y \in  B_{5R/8}(C_0)\subset B_{3R/4}(C_0)$) and we define $c$ as the intersection point of $\overline{ya}$ and $ \partial B_{\frac{5}{8}R}(C_0)$ (see Figure ~\ref{fig:incicciottato}).
Since $b \in \overline{B_{3R/4}(C_0)}$ and $\xi\in \cA_1$,  there exists $z \in \xi $ such that $|z-b|<{R/8}$. 
  Then
$ |y-b|=|y-c|+|c-b|\ge|y-c|+(3/4-5/8)R\geq  R/8$.
As $|y-b|\geq  {R/8}$ while $|z-b|<R/8$, then it must be $z \neq y$. 
We get that
\[ 
    |z-a| \le |z-b|+|b-a|=|z-b|+|y-a|-|b-y| 
    < {R/8}+|y-a|-{R/8},
\]
 and therefore  $|y-a|>|z-a|$. This implies $a \notin{\rm Vor}(y|\xi)$, which is absurd.  This concludes the proof of Item (i).

\begin{figure}[ht]
\includegraphics[scale=0.38]{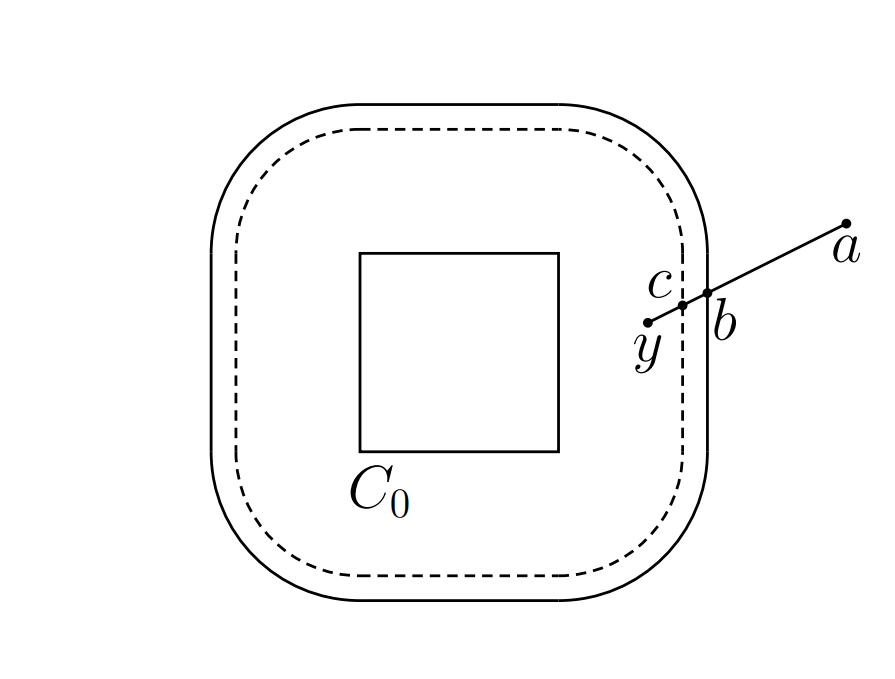}\caption{The square corresponds to  $C_0$, the dotted closed line corresponds to  $\partial B_{5R/8}(C_0)$ and the external closed line corresponds to   $\partial C_{3R/4}(C_0)$ ($d=2$). }\label{fig:incicciottato}
\end{figure}

%Items (ii) and (iii) are an immediate consequence of Item (i).
 
It remains to prove Item  (ii). Let $y \in \xi$ be such that ${\rm Vor}(y|\xi) \cap B_{R/2}(C_0)\neq \emptyset$. 
We then fix $ a \in {\rm Vor}(y|\xi) \cap B_{R/2}(C_0)$. 
Since $B_{5R/8}(C_0)=B_{{R/8}+{R/2}}(C_0)=B_{{R/8}}\big(B_{R/2}(C_0)\big)$, it is sufficient to prove that $|y-a|<{R/8}$. Let us assume the opposite, i.e. $|y-a|\ge {R/8}$. Since $a \in B_{R/2}(C_0)$ and $\xi\in \cA_1$, then there exists $z \in \xi$ such that $|z-a| < {R/8}$. Then we have  $|z-a|<{R/8}\le |y-a|$, thus  implying that  $a \notin {\rm Vor}(y|\xi)$, which is absurd.
\end{proof}
\begin{Lemma}
\label{apa0}  We have $\bbP(\cA_i)\to 1$ as $R\to +\infty$ for any $i=1,2,3$.
% In particular, \rosso{for some positive constants $c,C$ independent from $R$,}  it holds\begin{equation}\label{CarloIII}
%\begin{split} \bbP(\cA_1^c\cup \cA_2^c\cup \cA_3^c) & \leq \rosso{c}\, \bbP\big(\, \xi \big(\L_{R/(16\sqrt{d})}\big)=0\,\big)\\
%&+\bbP\big( \xi (B_{7R/8}(C_0)>R^{d+1}\big)+\bbP\left( \xi \big(\overline{B_{R/8}(C_0)}\big)>R^{d+1} \right)\\
%& \leq \rosso{c}\, \bbP\big(\, \xi \big(\L_{R/(16\sqrt{d})}\big)=0\,\big)+ C /R
%\end{split}
%\end{equation}
%and $\lim_{R\to +\infty} \bbP\big(\, \xi \big(\L_{R/(16\sqrt{d})}\big)=0\,\big)$.
\end{Lemma}
\begin{proof} $\bullet$ We start with $\bbP(\cA_1)$.
We cover the  box $\L_{2R} \supset \overline{B_{3R/4}(C_0)}$ by open subboxes $B_j$, $j\in J$, of side length $R/(8\sqrt{d})$  (hence having diameter $R/8$). To this aim we need at most $c(d)$ subboxes (the constant $c(d)$ depends only on $d$ and in particular is independent from $R$). Since every $y\in \overline{B_{3R/4}(C_0)}$ must belong to some $B_j$ and $B_j$ is an open box with diameter $R/8$, we conclude that  $ \{  \xi (B_j)\geq 1 \; \forall j \in J\}\subset \cA_1$. 
%Hence we just need to prove that $\lim_{R\to +\infty} \bbP( \tilde{\cA}_1)=1$.To this aim, 
By a union  bound and due to translation invariance, we estimate
\begin{equation}\label{Zstima1}
\begin{split}
\bbP(\cA_1^c) & \leq
\bbP \big( \xi (B_j)=0 \; \text{for some } j \in J\big) \leq  c(d)\, \bbP\big(\, \xi \big(\L_{R/(16\sqrt{d})}\big)=0\,\big)\,.\end{split}
\end{equation}
%{thus implying \eqref{Zstima1}.}
On the other hand, by monotonicity,  $ \bbP\big(\, \xi \big(\L_{R/(16\sqrt{d})}\big)=0\,\big) \to \bbP( \xi (\bbR^d)=0)=0$ as $R\to +\infty$. As a consequence  $\lim_{R\to +\infty} \bbP( \cA^c_1)=0$.

\smallskip

$\bullet$ We now move to   $\bbP(\cA_2)$. First we show that $(\tilde{\cA}_2\cap \cA_1) \subset \cA_2$, where  
\[
\tilde{\cA}_2:=\big \{\xi (B_{7R/8}(C_0))\leq R^{d+1}\big\}\,.
\]
% Afterwards we prove that $\lim_{R\to +\infty} \bbP(\tilde{\cA}_2)=1$. Combining the above two steps with $\lim_{R\to +\infty}\bbP( \cA_1)=1$, one immediately gets our claim for $i=2$. Hence let us first show that $(\tilde{\cA}_2\cap \cA_1) \subset \cA_2$. 
To this aim suppose  that 
$\xi \in \tilde{\cA}_2\cap \cA_1$ and 
 take  $y\in \xi$ with 
${\rm Vor}(y|\xi )\cap B_{R/4}(C_0)\neq\emptyset$. We need to prove that ${\rm Vor}(y|\xi )$ has at most $R^{d+1}$ neighboring Voronoi cells. 
By applying first  Lemma \ref{ipa1}--(ii) we get that $y\in 
 B_{5R/8}(C_0) $ and therefore, by
   Lemma \ref{ipa1}--(i),  we obtain  that     ${\rm rad}({\rm Vor}(y|\xi))\leq {R/8}$. 
 By using this last  information, we  now prove that   the distance between $y$ and any of its neighbors $z$ in ${\rm DT}(\xi)$  is at most $R/4$. To this aim take  a point $a$ in the  $(d-1)-$dimensional face shared by ${\rm Vor}(y|\xi)$ and ${\rm Vor}(z|\xi)$.
 We then have 
\begin{equation*}
    |z-y|\le |z-a|+|y-a|=2|y-a|\le 2\, {\rm rad}\big({\rm Vor}(y|\xi)\big) \le  R/4\,.
    \end{equation*}
    The above bound proves that $z\in B_{R/4}(y)\cap \xi$. As $y\in 
 B_{5R/8}(C_0) $, we get that $z\in \xi \cap B_{7R/8}(C_0)$. As $\xi \in \tilde{\cA}_2$, we then conclude that there can be at most $R^{d+1}$ of such points $z$. This concludes the proof that $(\tilde{\cA}_2\cap \cA_1) \subset \cA_2$.
 
% Finally, to prove  that $\lim_{R\to +\infty} \bbP(\tilde{\cA}_2)=1$ it is enough to observe that 
To conclude that $\lim_{R\to +\infty} \bbP(\cA_2)=1$ we use that $\lim_{R\to +\infty} \bbP(\cA_1)=1$
and the estimate (recall that the intensity of the SPP is finite)
 \be\label{Zstima2}
  \bbP(\tilde{\cA}_2^c)= \bbP\big( \xi (B_{7R/8}(C_0))> R^{d+1}\big) \leq R^{-(d+1)}\bbE \big[ \xi (B_{7R/8}(C_0))\big]=O(R^{-1})\,.
  \en

 \smallskip
 
 $\bullet$ Let us  now show that  $\bbP(\cA_3)\to 1$. Due to the previous results, it is enough to show that $\bbP(\cA_3^c \cap \cA_1)\to 0$ as $R\to +\infty$. Suppose that $\xi \in \cA_1$ and take $y\in \xi$ with  ${\rm Vor}(y|\xi)$ intersecting $\partial C_0$.
 By applying first Item (ii) and then Item (i) in  Lemma \ref{ipa1}   we obtain  that     ${\rm rad}({\rm Vor}(y|\xi))\leq {R/8}$. As ${\rm Vor}(y|\xi)$ intersects $\partial C_0$, we conclude that  $y$ belongs to the closure $\overline{B_{R/8}(C_0)}$. In particular, if $\xi \in \cA_3^c \cap \cA_1$, then it must be $\xi \big(\overline{ B_{R/8}(C_0)}\big)>R^{d+1}$. Therefore we have 
 \be\label{Zstima3}
 \bbP\left(\cA_3^c \cap \cA_1\right) \leq \bbP\left( \xi \big(\overline{B_{R/8}(C_0)}\big)>R^{d+1} \right) \leq 
 R^{-(d+1)}\bbE\left[ \xi \big(\overline{B_{R/8}(C_0)}\big)\right]=O(R^{-1})\,.
 \en
 In particular,  $\bbP(\cA_3^c \cap \cA_1)\to 0$.
 \begin{comment}
  \smallskip
 $\bullet$ We conclude with the estimate \eqref{CarloIII}. Since $\big(\cA_1^c\cup \cA_2^c\cup \cA_3^c\big)\rosso{\subset}
 \big(\cA_1^c\cup  \tilde \cA_2^c \cup (\cA_3^c\cap \cA_1)\big)$, the bound \eqref{CarloIII} follows from \eqref{Zstima1}, \eqref{Zstima2} and \eqref{Zstima3}.
 \end{comment}
 \end{proof}
%%%
We point out  that the parameters $p,R$ have been  understood in the notation up to now. We define
\be \phi(p,R):=\bbQ( C_0 \text{ is open})=\bbQ( C_x \text{ is open}) \qquad x\in \bbZ^d
\en
(the final equality comes from the stationarity of $\bbP$ and therefore of $\bbQ$).
\begin{Lemma}
\label{opencube}
Given $p\in [0,1]$ and $R>0$ such that  $pR^{d+1}\leq 1/2$, we have
\be \phi(p,R)\leq   \bbP(\cA_1^c\cup \cA_2^c\cup \cA_3^c)+ 2p R^{d+1}\,.
\en
%For any  $ \d\in (0,1)$ there exist constants  ${R_\d}>0$ and $c_\d\in (0,1)$ such that $\phi(p,R)\leq \d$ whenever $R\geq R_\d$ and $pR^{d+1}\leq c_\d$.
\end{Lemma}
%%%%%%%%%%%%%%%%%%%%%%%%%%%%%%%%%%%%%%%%%%%%%%%%%%%%%%%%%%%%%%%%%%%%%%%%%%%%%%%%%%%%%%%%%%%%%%%%%%%%%%%%%%%%%%%%%%%%%%%%%%%%%%%%%%%%%%%%%%%%%%%%%%%%%%%%%%%%%%%%%%%%%%%%%%%%%%%%%%%%%%%%%%%%%%%%%%%%%%%%%%%%%%%%%%%%%%%%%%%%%%%%%%%%%%%%%%%%%%%%
\begin{proof}We define 
\begin{align}
\cA:=\left\{ (\xi, W)\in \cX\,:\, C_0 \text{ is $(\xi,W)$-open}\right\}\,,\\
\bar \cA_i:=\left\{(\xi,W)\in \cX\,:\, \xi \in \cA_i \right\} \;\; i=1,2,3\,.
\end{align}
%From Lemma \ref{apa0} it follows that there exists $ R_\d>0$ such that   $\bbP(\cA_i^c)\leq  \d/6$ for  $ R\geq  R_\d$  and  $i=1,2,3$. Hence, we have that $\bbP(\cA)\leq \d$    whenever $R\geq  R_\d$  and $\bbP(\cA \cap\bar\cA_1 \cap \bar\cA_2 \cap \bar\cA_3)\leq \d/2$.
We have 
\be\label{zorro123}
\phi(p,R)\leq   \bbQ(\cA_1^c\cup \cA_2^c\cup \cA_3^c)+ \bbQ(\cA \cap\bar\cA_1 \cap \bar\cA_2 \cap \bar\cA_3)
\en
We then need to prove that $\bbQ(\cA \cap\bar\cA_1 \cap \bar\cA_2 \cap \bar\cA_3)\leq 2p R^{d+1}$ if $p R^{d+1}\leq 1/2$. To this aim 
we introduce the following random path sets, defined on the space $\cX$ (recall Definition~\ref{grafo_cG}):

\begin{Definition}[Path sets $\G_n$ and $\tilde \G_n$] \label{self_avoiding}
  For $n\geq 1$  we define $\G_n=\G_n(\xi,W)$ as the family of 
self-avoiding paths  $(x_1,...,x_n) $ in $\cG=\cG(\xi,W)$ such that 
\be\label{bicicletta}
\begin{cases}
 {\rm Vor}(x_1|\xi)\cap\partial C_0\neq\emptyset\,,\\
  {\rm Vor}(x_i|\xi )\cap B_{R/4}(C_0) \neq \emptyset \;\; \forall i=2,\dots,n-1\,,\\
  {\rm Vor}(x_n|\xi )\cap \partial B_{R/4}(C_0) \neq \emptyset\,,
\end{cases}
\en
and we set   $\G:=\cup_{n=1}^\infty \G_n$.
We also define  
 $\tilde{\G}_n=\tilde{\G}_n(\xi)$ as the family of 
self-avoiding paths  $(x_1,...,x_n) $ in ${\rm DT}(\xi)$ satisfying \eqref{bicicletta} and we set $\tilde \G:=\cup_{n=1}^\infty \tilde \G_n$.
\end{Definition}
Let   $(\xi,W) \in \cA\cap \bar \cA_1$.  Since $\bar\cA_1$ holds,  Condition (i)  for $C_0$ in Definition \ref{zdpro}   automatically fails and therefore, to guarantee $\cA$, 
$(\xi,W) $ satisfies    Condition (ii) for $C_0$ in Definition \ref{zdpro}. Let us now show  that 
 this  implies that    $\G(\xi,W) \neq \emptyset$.
To prove our claim let $\g:[0,1] \to  B_{R/4}(C_0)$  be 
a continuous curve  as in  Condition (ii) for $C_0$.
At cost to cut and reparameterize $\g$, we can assume that  $\g(0) \in \partial C_0$, $\g(1) \in \partial B_{R/4}(C_0)$ and  $\g(t)\in B_{R/4}(C_0)\setminus C_0$ for every $t \in (0,1)$.  Then we define $(x_1,...,x_n)$ as the sequence of points of $\xi$ such that, in chronological order, the path $\g$ visits ${\rm Vor}(x_1|\xi)$, ${\rm Vor}(x_2|\xi)$,...,${\rm Vor}(x_n|\xi)$ and such that $W_{x_1,x_2}=1$,  $W_{x_2,x_3}=1$,..., $W_{x_{n-1},x_n}=1$. At cost to prune the loops, we can make $\g$ a self-avoiding path. Then necessarily we have $\g\in \G(\xi,W)$, thus proving our claim.

Due to the above observations we can bound
\begin{multline}
\label{genziana}
      \bbQ(\cA \cap  \bar\cA_1 \cap  \bar \cA_2 \cap \bar \cA_3)
          \le\bbQ\big(\left\{\G \neq \emptyset\right\}\cap \bar \cA_1 \cap \bar \cA_2 \cap \bar \cA_3\big) \le \bbE_{\bbQ}\left[|\G|\, \mathds{1}_{\bar \cA_1\cap \bar\cA_2 \cap \bar\cA_3}\right] \\= 
          \bbE_{\bbQ}\left[\, \bbE_\bbQ\left[|\G| \mathds{1}_{\bar \cA_1\cap \bar\cA_2 \cap \bar\cA_3}\,|\xi\right]\,\right]=
         \bbE \left[
                  \bbE_\bbQ
                       \left[
                           |\G| |\xi\right] \mathds{1}_{ \cA_1\cap\cA_2 \cap \cA_3} 
                           \right]\,,
\end{multline}
where $\bbE_{\bbQ}$ denotes the expectation w.r.t. $\bbQ$ and $|\G |$ denotes the cardinality of $\G$.
%We point out that under the validity of $A_1$, for every $y \in \o$ such that $Vor(y)\cap \partial C_0 \neq \emptyset$ follows that $y \in B_{R/8}(C_0)$. In fact suppose that dist$(y, C_0)> {R/8}$,{R/ let $a\in Vor(y)\cap \partial C_0$. It follows that $|a-y|>{R/8}$ i.e. $y \notin B_{{R/8}}(a)$. As $A_1$ holds and $a \in \partial C_0$ it follows that $|\o \cap B_{{R/8}}(a)|>0$. Then we get that $y \notin \o \cap B_{{R/8}}(a)$ and that there is a nearest point to $a$ in the set $\o \cap B_{{R/8}}(a)$ meaning that $a \notin \text{Vor}(y)$ which is a contradiction to our assumption that $a \in \text{Vor}(y)$.
If $\xi\in \cA_2$, then for every $ y\in \xi$ such that ${\rm Vor}(y|\xi) \cap B_{R/4}(C_0)\not=\emptyset$, $y$ has at most $R^{d+1}$ neighbors in DT$(\xi)$.
Moreover if   $\xi\in \cA_3$ then the number of Voronoi cells associated to $\xi$ and intersecting $\partial C_0$ is at most $R^{d+1}$.
Hence, for  $\xi \in \cA_2\cap \cA_3$, we have 
$    |\tilde{\G}_n|\le R^{n(d+1)}$. 
By using the above bound,  for  $\xi \in \cA_2\cap \cA_3$  we obtain 
\begin{equation}
   \begin{split}
        \label{rosa}
        \bbE_\bbQ\left[ |\G_n|\>\,|\xi\right]&=\sum_{(x_1,...,x_n)\in \widetilde{\G}_n}\bbQ\left((x_1,...,x_n)\in {\G}_n| \xi\right) \\
        &= \sum_{(x_1,...,x_n)\in \tilde{\G}_n}\bbQ\left(W_{x_i,x_{i+1}}=1, \> i=1,\dots,n-1\>| \xi \right) \\
         &=p^n  |\tilde{\G}_n|\le p^n R^{n(d+1)}\,.
    \end{split}
\end{equation}
We choose $p$ such that $pR^{d+1}< 1$. Then, as $\G=\cup_{n\geq 1} \G_n$,  we have
\[
\text{r.h.s. of \eqref{genziana}}\leq   \sum_{n=1}^\infty p^{n} R^{n(d+1)}  \bbP(\cA_1\cap \cA_2 \cap \cA_3) \le \sum_{n=1}^\infty  (pR^{d+1})^n  = \frac{pR^{d+1}}{1-pR^{d+1}} \,.
\]
For $p R^{d+1}\leq 1/2$, the rightmost term is bounded by $2pR^{d+1}$,
thus allowing to conclude.
%Recall that we want that $\bbP(\cA \cap\bar\cA_1 \cap \bar\cA_2 \cap \bar\cA_3)\leq \d/2$. Hence  we just need to assure that $ \frac{pR^{d+1}}{1-pR^{d+1}} \leq \frac{\d}{2}$, i.e. $p R^{d+1} \leq h^{-1} (\d/2)=:c_\d$, where $h:(0,1)\to (0,+\infty)$ is  the bijective increasing function $h(x)=x/(1-x)$.
\end{proof}
%\begin{Remark}
%\label{vf}
%We point out that by a simple coupling argument, the function $\phi(p,R)$ is non decreasing in $p\in [0,1]$.
%\end{Remark}

%\begin{Lemma} [\bf{0-1 law for the $\bbZ^d$-process} ]
%If $\o$ is sampled as an ergodic stationary point process, then $\bbP(\eta \in A) \in \left\{0,1\right\}$ $\forall A \subset \left\{0,1\right\}^{\bbZ^d}$ which is Borel and translation invariant.
%\end{Lemma}
%\begin{proof}
%Let $\o$ be a stationary point process. It suffices to show that every $ A \subset \left\{0,1\right\}^{\bbZ^d}$ which is Borel and translation invariant  corresponds to some translation invariant event for $\o$. In fact, by using Proposition 12.3.III in \cite{DV}, one has that a stationary point process is ergodic if and only if the shift-invariant $\sigma$-algebra is trivial. This implies that $\forall E$ in the  Borel $\sigma-$algebra of $\bbR^d$ that is shift invariant, $\bbP(\o \in E)\in \left\{0,1\right\}$. So we only need to prove that the events $ A \subset \left\{0,1\right\}^{\bbZ^d}$ which are Borel and translation invariant in the $\bbZ^d$-process corresponds to shift-invariant event for $\o$.
%\end{proof}

%%%%%%%%%%%%%%%%%%%%%%%%%%%%%%%%%%%%%%%%%%%%%%%%%%%%%%%%%%%%%%%%%%%%%%%%%%%%%%%%%%%%%%%%%%%%%%%%%%%%%%%%%%%%%%%%%%%%%%%%%%%%%%%%%%%%%%%%%%%%%%%%%%%%%%%%%%%%%%%%%%%%%%%%%%%%%%%%%%%%%%%%%%%%%%%%%%%%%%%%%%%%%%%%%%%%%%%%%%%%%%%%%%%%%%%%%%%%%%%%%%%%%%%%%%%%%%%%%%%%%%%%%%%%%%%%%%%%%%%%%%%%%%%%%%%%%%%%%%%%%%%%%%%%%%%%%%%%%%%%%%%%%%%%%%%%%%%%%%%%%%%%%%%%%%%%%%%%%%%%%%%%%%%%%%%%%%%%%%%%%%%%%%%%%%
\subsection{Proof of Proposition \ref{perzdthm}}\label{stilton}
%%%%%%%%%%%%%%%%%%%%%%%%%%%%%%%%%%%%%%%%%%%%%%%%%%%%%%%%%%%%%%%%%%%%%%%%%%%%%%%%%%%%%%%%%%%%%%%%%%%%%%%%%%%%%%%%%%%%%%%%%%%%%%%%%%%%%%%%%%%%%%%%%%%%%%%%%%%%%%%%%%%%%%%%%%%%%%%%%%%%%%%%%%%%%%%%%%%%%%%%%%%%%%%%%%%%%%%%%%%%%%%%%%%%%%%%%%%%%%%%%%%%%%%%%%%%%%%%%%%%%%%%%%%%%%%%%%%%%%%%%%%%%%%%%%%%%%%%%%%%%%%%%%%%%%%%%%%%%%%%%%%%%%%%%%
We show that the connected component of the origin in the graph $\bbG(\xi,W)$ is a finite set $\bbQ$--a.s. By translation  invariance, the same holds for all connected components, and this would allow to conclude.

 Given $(\xi,W)\in \cX$ we declare a point  $z\in \bbZ^d$ to be $(\xi,W)$--open  if the corresponding box $C_z$ is 
 $(\xi,W)$--open, i.e. if $\eta_z(\xi,W)=1$. We recall that  a  path $\g$ in  $\bbZ^d$    is a sequence $(x_1,x_2,\dots, x_n)$ of points in $\bbZ^d$  such that $|x_i-x_{i-1}|=1$ for all $i=1,2,\dots, n-1$. We say that $\g$ is $(\xi,W)$--open if $x_i$ is $(\xi,W)$--open   for any site $x_i$.

Given $N\geq 1$ we define $\G^*_N$ as the family of self-avoiding paths $(x_1,x_2,\dots, x_n)$ in $\bbZ^d$ such that  $x_1=0$, $ |x_n|_\infty=N$  and  $|x_i|_\infty<N $ for all $ i=2,\dots,n-1$. We also introduce the event
\[
\cB_N:= \left\{ \;(\xi,W)\in \cX\,:\, \exists \g \in \G^*_N\ \text{ which is $(\xi,W)$-open}\;\right\}\,.
\]
Suppose that $\bbP$ has range of dependence smaller than $L$ and set $\ell = 3+L/R$.
Given a path $\g=(x_1,x_2, \dots, x_n)$ we extract a subpath $\g':=(x_1',x_2', \dots, x_m')$ according to the following procedure. We set $x_1':=x_1$. Then we define $x_2'$ as the first point (if existing) visited by $\g$ in the set $\bbZ^d\setminus B_\ell(x_1')$, i.e.  $x_2':=x_r$  where $r$  is the  minimal  index such that $x_r \in
 \{x_1,x_2,\dots, x_n\}\setminus B_\ell (x_1')$ (if $ \{x_1,x_2,\dots, x_n\}\setminus B_\ell (x_1') =\emptyset$, then set $m:=1$ and stop). In general, if we have defined $x_1',x_2',\dots, x_k'$ we proceed as follows: if 
 \be\label{condimento}
  \{x_1,x_2,\dots, x_n\}\setminus \left( \cup _{i=1}^k B_\ell (x_i')\right) =\emptyset\,,\en then we set $m:=k$ and this completes the definition of $\g'$, otherwise we set $x_{k+1}':=x_r$ where  $r$  is the  minimal  index such that $x_r \in
 \{x_1,x_2,\dots, x_n\}\setminus \left( \cup _{i=1}^k B_\ell (x_i')\right)$ and  we continue with the construction.
 Since the set $\cup _{i=1}^k B_\ell (x_i')$ has at most $k c(d) \ell^d$ points, condition \eqref{condimento} is violated if $n> k c(d) \ell^d$. In particular, if $n$ is the length of $\g$, the length $m$ of $\g'$ is at least $[ n/ (c(d) \ell^d) ]$, where $[\cdot]$ denotes the integer part.
Note that by construction, all points in $\g'$ have reciprocal distance larger than $\ell$.  Hence, using Lemma \ref{indipendenza}, we have 
\be
\begin{split}
& \bbQ(\g=(x_1,\dots, x_{n})\text{ is $(\xi,W)$--open} ) \leq 
\bbQ( \eta_{x_1'}=1,\dots, \eta_{x_m'}=1)\\
& = \prod _{r=1}^{m} \bbQ( \eta_{x_{r}'}=1\,|\, \eta_{x_1'}=1,\, \eta_{x_2'}=1,\dots, \eta_{x_{r-1}'}=1)\\
&= 
\prod _{r=1}^{m} \bbQ( \eta_{x_{r}'}=1)=\phi(p,R)^m\leq \phi(p,R)^{  n/ (c(d) \ell^d) -1}\,.
\end{split}
\en 
Now we observe that if $\g\in \G^*_N$, then its length is at least $N$. Moreover, there are at most $(2d)^n$ paths $\g$ in $\bbZ^d$ of length $n$ starting at the origin.
 The above observations allow us to bound
 \be
 \begin{split}
 \bbP(\cB_N) & \leq \sum _{\g \in \G_N^*} \bbQ( \g \text{ is $(\xi,W)$--open} )\\
 &\leq \sum_{n=N}^\infty (2d)^n \phi(p,R)^{  n/ (c(d) \ell^d) -1}=\phi(p,R)^{-1} \frac{\k ^N}{1-\k} \,,
 \end{split}
 \en
 where $\k:= 2d\, \phi(p,R)^{  1/ (c(d) \ell^d) }$. Note that if $\k <1$, then $\sum_{N=1}^\infty \bbP(\cB_N)<+\infty$. Hence, by Borel-Cantelli Lemma, if $\k<1$ then   $\bbQ$--a.s. the event $\cB_N$ fails for $N$ large enough. On the other hand, if the connected component $\bbC$ of the origin in the graph $\bbG(\xi,W)$ was infinite, all events $\cB_N$ would take place. We therefore conclude that $\bbQ$--a.s.  $\bbC$ is finite.
 
The bound $\k<1$ corresponds to $\phi(p,R)< \d:= (1/2d)^{c(d) \ell^d}$.
By  Lemma \ref{apa0} we can fix $R_*$ such that  $ \bbP(\cA_1^c\cup \cA_2^c\cup \cA_3^c)\leq \d/3$ for all $R\geq R_*$.  If in addition $p R^{d+1} \leq  (1/2) \wedge (\d/6)$, then  by Lemma~\ref{opencube} we conclude that $\phi(p,R)\leq 2\d/3<\d$.
  
%%%%%%%%%%%%%%%%%%%%%%%%%%%%%%%%%%%%%%%%%%%%%%%%%%%%%%%%%%%%%%%%%%%%%%%%%%%%%%%%%%%%%%%%%%%%%%%%%%%%%%%%%%%%%%%%%%%%%%%%%%%%%%%%%%%

\appendix

\section{Basic facts about point processes}\label{app_basico}

\begin{Lemma}\label{pimpa2}  Under \eqref{rel_cov_1}  the following holds:
\begin{itemize}
\item[(i)] If $\cP$ is stationary, then $\bbP$ is stationary.
\item[(ii)] If $\cP$ is ergodic, then $\bbP$ is ergodic.
\end{itemize}
\end{Lemma}
\begin{proof} (i) The stationarity   of $\bbP$ means   that $\bbP(A)= \cP(\hat \o  \in A) $ equals $\bbP( \t_x A)=\cP ( \hat \o  \in \t_x A)$ for any $A\in \cB(\cN)$ and $x\in \bbR^d$. Due to \eqref{rel_cov_1}, given $\o\in \O_*$, we have   $ \hat\o  \in \t_x A $ if and only if $ \widehat{ \theta_{-x} \o} \in A$. Hence,  using also that $\cP(\O_*)=1$,  the stationarity of $\bbP$  
is equivalent to the family of identities   $\cP(\hat \o \in A)= \cP( \widehat{ \theta_{-x} \o} \in A)$ for all $A\in \cB(\cN)$ and $x\in \bbR^d$. These identities are fulfilled due to 
 the stationarity of $\cP$. 
 
(ii)  The ergodicity of $\bbP$ means   that $\bbP(A)=\cP (\hat\o \in A)\in \{0,1\}$ for any $A\in \cB(\cN)$ with $\t_x A=A$ for all $x\in \bbR^d$.   By \eqref{rel_cov_1} and the translation invariance of $\O_*$, for such a set $A$ 
 the set $B:=\{\o\in \O_*\,:\, \hat \o \in A\}$ is measurable and translation invariant. Then, the ergodicity of $\bbP$ follows from the ergodicity of $\cP$.
 \end{proof}

%%%%%%%%%%%%%%%%%%%%%%%%%%%%%%%%%%%%%%%%%%%%%%%%%%
\subsection{Proof of Lemma \ref{criterio_poliedro}}\label{app_criterio}
Since convex polytopes coincide with bounded polyhedra, we need to show that for all $x\in \xi$ the cell $\rmVor(x|\xi) $ is a bounded  polyhedron, i.e.~it is  the bounded intersection of finitely many   closed half-spaces.

To this aim fix $x\in \xi$.  We choose $y \in \bbZ^d$ such that $(y+ Q_\s)\subset (x +Q_\s)$. As $y+Q_\s$ intersects $\xi$ by our assumption,  the same holds for $x+Q_\s$ and therefore we can fix  $z^\s\in Q_\s$ such that $(x+z^\s) \in \xi \cap (x+Q_\s)$.  

We first show that $\rmVor(x|\xi)$ is bounded. Take  $v\in \bbR^d$ with $x+v \in \rmVor(x|\xi)$  and  let $\s\in \{-1,+1\}^d$ be defined as $\s_i:=+1$ if $v_i \geq 0$, otherwise $\s_i:=-1$. As $z^\s\in Q_\s$, $z_i^\s$ and $\s_i$ have the same sign, thus implying that $z^\s_i v_i \geq 0$.  
Since  $x+v \in \rmVor(x|\xi)$, then $x$ minimizes in $\xi$ the distance from $x+v$ and therefore 
  $|(x+v)- (x+z^\s)|^2 - |(x+v)-x|^2= \sum_{i=1}^d ( v_i - z^\s_i)^2- \sum _{i=1}^d v_i^2=|z^\s|^2 -2\sum_{i=1}^d |v_i z^\s_i|$  must be  non-negative. We can therefore conclude that 
   $|v_i |\leq|z^\s|^2/|z_i^\s|$ for all $i=1,\dots, d$ and for all $v\in \bbR^d$ with $x+v\in {\rm Vor}(x|\xi)$. This proves that $ \rmVor(x|\xi)$ is bounded.

 Since ${\rm Vor}(x|\xi)$ is bounded, we can fix $R>0$ such that $ \rmVor(x|\xi)$ is included in the  ball  $B_R(x)$. Given $z\in \xi\setminus\{x\} $, consider the closed half-space $H(x,z):=\{y \in \bbR^d\,:\, | y-x| \leq |y-z|\}$.
By definition, $\rmVor(x|\xi)= \cap _{z \in \xi:z\not=x} H(x,z)= A\cap A_*$ where 
\[
A:= \cap _{z \in \xi: 0<|z-x|\leq 4R} H(x,z)\,,\qquad A_*:=  \cap _{z \in \xi: |z-x|> 4R} H(x,z)\,.
\]
Since $H(x,z)$ contains the closed ball centered at $x$ of radius $|x-z|/2$, we get that  $B_{2R}(x) \subset A_* $.
On the other hand, since 
  $A\cap A_*={\rm Vor}(x|\xi)\subset B_R(x)$, 
  we have $ A\cap A_* \subset B_R(x)$. 
  We claim  that $A\subset A_*$. Take by contradiction $y \in A\setminus A_*$.  Since  $B_{2R}(x) \subset A_*$, it must be $y\not \in B_{2R}(x)$. By convexity of $A$, the segment from $x$ to $y$ belongs to $A$.  This segment must contain a point $z$  at the boundary of $B_{2R}(x)\subset A_*$.  Then $z\in A\cap A_*$ and $|z-x| =2R$, thus contradicting that $A\cap A_* ={\rm Vor}(x|\xi)\subset B_R(x)$. This concludes the proof of our claim.

As   $A\subset A_*$ it must be  ${\rm Vor}(x|\xi)=A\cap A_*= A$ and this  proves that  ${\rm Vor}(x|\xi) $ is the intersection of finitely many   closed half-spaces.

\subsection{Proof of Lemma~\ref{lemma_npot}}\label{app_forza}
By Lemma~\ref{criterio_poliedro}, we just need to show that $\bbP(\xi \cap (x+Q^\s)=\emptyset)=0$ for all $x\in \bbZ^d$ and $\s\in \{-1,+1\}^d$. Suppose by contradiction that $\a:=\bbP(\xi \cap (x+Q^\s)=\emptyset)>0$ for some $x$ and $\s$ as above.  Then by the stationarity of $\bbP$ we get $\bbP(\xi \cap (-n \s+ Q^\s) =\emptyset)=\a$ for all $n\in \bbN$. Since the sequence of quadrants $-n \s+ Q^\s$ is increasing and invading all $\bbR^d$, we get that 
\[\bbP(\xi =\emptyset)= \bbP (\cap _{n\in \bbN}\{\xi \cap (-n \s+ Q^\s) =\emptyset\}) =
\lim_{n\to +\infty} \bbP(\xi \cap (-n \s+ Q^\s) =\emptyset)=\a\,,
\]thus contradicting the assumption that $\bbP(\xi=\emptyset)=0$.
%%%%%%%%%%%%%%%%%%%%%%%%%%%%%%%%%%%%%%%%%%%%%%
\subsection{Proof of Proposition~\ref{prop_FR}}\label{nerone1}  Suppose first that  $\bbP$ has  range of dependence smaller than $L$.
  On $\L_\ell$ we can fix cubes of side length $1$  at reciprocal Euclidean distance at least $L$: it is enough to consider the cubes $\L_{1/2}(z)$ as $z$ varies in $(L+1)\bbZ^d \cap \L_{\ell-1/2}$ (we will write simply $z\in \cZ$, where $\cZ$ depends on $\ell,L$). The cardinality $n$ of these cubes is lower bounded by $C \ell^d$ with $C=C(L)$ for $\ell $ large enough (as we assume below).

  It must be $\d:=\bbP( \xi (\L_{1/2})=0)<1$ otherwise, by stationarity, we would get that  $\bbP( \xi (\L_{1/2}(z))=0)=1$ for any $z$ and therefore $\bbP(\xi =\emptyset)=0$, against our assumption of positive intensity $m$.

We write $\cZ:=\{z_1,z_2,\dots, z_n\}$ and set $\cA_i:=\big\{\xi(\L_{1/2}(z_i))=0 \big\}$ for $i=1,2,\dots, n$.
By stationarity $\bbP(\cA_i)=\d$ for all $i$. 
By the finite range of dependence and since the reciprocal distance of the $\L_{1/2}(z_i) $ is at least $L$ we have 
\begin{equation*}
%\begin{split}
\bbP \big( \xi(\L_\ell)=0 \big) \leq 
 \bbP( \cA_1\cap \cA_2\cap \cdots \cap \cA_n)=\bbP( \cA_1) \bbP (\cA_2)  \cdots  \bbP(\cA_n)= \d^n \leq \d^{C \ell^d}\,.
%\end{split}
\end{equation*}

Suppose now that $\bbP$ has negative association. We use the same construction and notation  as above with $L=0$. Since the function $f(z)=\mathds{1}(z=0)$ on $\bbR_+$ is weakly decreasing,  by Remark~\ref{banff_arte} we conclude that 
\[
\bbP \big( \xi(\L_\ell)=0 \big) \leq 
 \bbP( \cA_1\cap \cA_2\cap \cdots \cap \cA_n)\leq \bbP( \cA_1) \bbP (\cA_2)  \cdots  \bbP(\cA_n)= \d^n \leq \d^{C \ell^d}\,.
\]

%%%%%%%%%%%%%%%%%%%%%%%%%
\subsection{Proof of Proposition~\ref{prop_tail}}\label{nerone2} The arguments are very close to the ones in the previous proof.
Let $\d:=\bbP( \xi (\L_{1/2})=0)<1$. We  fix $L$ large enough to have   $\| \bbP( \xi (\L_{1/2})=0| \cT_L) -\d\|_\infty \leq (1-\d)/2$. This is possible due to \eqref{mannaro}.
  Then
  \be
  0 \leq    \bbP( \xi (\L_{1/2})=0| \cT_L) \leq \d+(1-\d)/2=:\g<1\,.
  \en
  This implies that for any event $B$ defined by the behavior of $\xi$ in a region having distance at least $L$ from $\L_{1/2}$ it holds
  \[
    \bbP( \xi (\L_{1/2})=0| B) = \frac{ \int_B   \bbP( \xi (\L_{1/2})=0| \cT_L) d\bbP (\xi)
    }{\bbP(B)} \leq \g\,.
\]
By stationarity the same holds if we translate everything. 
Having $\ell$ and $L$ we define $\cZ$, $n$ and $\cA_i$ as in the previous proof.
Hence, by writing 
\[
\bbP( \cA_1\cap \cA_2\cap \cdots \cap \cA_n)=\prod_{k=1}^n \bbP(\cA_k| \cA_1\cap \cA_2 \cap \cdots \cap \cA_{k-1})\,,
\]
we can upper bound each factor in the above r.h.s. by $\g$. This  allows to get $\bbP \big( \xi(\L_\ell)=0 \big)  
\leq \bbP( \cA_1\cap \cA_2\cap \cdots \cap \cA_n)\leq  \g^n \leq \g^{C \ell^d}$ for $\ell $ large.


\begin{thebibliography}{99}


\bibitem{AFST} S.~Andres, M.~Slowik, A.~Faggionato, Y.~Tokushige. \emph{Invariance principle, local CLT and H\"older regularity  for random conductance models on Delaunay triangulations}. In preparation. 

\bibitem{BB} P.~Balister, B.~Bollob\'{a}s; \emph{Bond percolation with attenuation in high dimensional Voronoi tilings}. Random Structures Algorithms 36, 5--10 (2009).

\bibitem{BBQ} P.~Balister, B.~Bollob\'{a}s, A.~Quas; \emph{Percolation in Voronoi tilings.}
Random Structures Algorithms 26, 310--318 (2005).


%\bibitem{Bi} P. Billingsley; \emph{Convergence of probability measures}.  Second edition. Wiley Series in Probability and Statistics: Probability and Statistics. John Wiley \& Sons, Inc., New York, 1999.



\bibitem{BR} B.~Bollob\'{a}s, O.~Riordan; \emph{Percolation}. Cambridge, Cambridge University Press (2006).



\bibitem{Br}  H.~Brezis; \emph{Functional Analysis, Sobolev Spaces and Partial Differential Equations}.  New York, Springer Verlag, 2010. 



\bibitem{CF} A.~Chiarini, A.~Faggionato;  \emph{Equilibrium density fluctuations of symmetric simple exclusion processes on random graphs with random conductances}. In preparation.

\bibitem{DV} D.J.~Daley, D.~Vere-Jones;  \emph{An Introduction to the Theory of Point Processes}. New York, Springer Verlag, 1988.



  \bibitem{demasi} A.\ De Masi, P.A.\ Ferrari, S.\ Goldstein, W.D.\ Wick; \emph{An invariance principle for reversible Markov processes. Applications to random motions in random environments.}  J. Stat. Phys. {\bf  55}, 787--855 (1989).

\bibitem{De} D.~Dereudre; \emph{Introduction to the theory of Gibbs point processes}.  In: Stochastic Geometry: Modern Research Frontiers, (D. Coupier, ed.). Lecture Notes in Mathematics 2237, Springer Verlag, 2019.



\bibitem{DS} P.G.~Doyle, J.L.~Snell; \emph{Random walks and electric networks}. Carus Mathematical Monographs, Mathematical Association of America  (1984). 


%\bibitem{DV1} D.J.~Daley, D.~Vere-Jones;  \emph{An Introduction to the Theory of Point Processes. Volume I: Elementary Theory and Methods}.  Second Edition,  New York, Springer  Verlag, 2003.

%\bibitem{DV2} D.J.~Daley, D.~Vere-Jones;  \emph{An Introduction to the Theory of Point Processes. Volume II: General Theory and Structure}.  Second Edition,  New York, Springer  Verlag, 2008.





\bibitem{F_sep} 
A.~Faggionato; \emph{Hydrodynamic limit of simple exclusion processes in symmetric random environments via duality and homogenization}. Probab. Theory Relat. Fields. {\bf 184}, 1093--1137 (2022).

 \bibitem{F_hom} A.~Faggionato; \emph{Stochastic homogenization of random walks on  point processes}. Ann.~Inst.~H.~Poincar\'{e} Probab.~Statist. {\bf 59}, 662--705 (2023). 


\bibitem{F_muratore}
A.~Faggionato; \emph{Graphical constructions of simple exclusion processes with applications to random environments}. ALEA, Lat. Am. J. Probab. Math. Stat. {\bf 21}, 1949--1983 (2024) 


   \bibitem{F_resistor}  A.~Faggionato; \emph{Scaling limit of the directional conductivity of random resistor networks on simple point processes}.  Ann.~Inst.~H.~Poincar\'{e} Probab.~Statist. {\bf 61},  1487--1521 (2025).




\bibitem{F_in_prep}
A.~Faggionato; \emph{Graphs with random conductances on point processes: RW homogenization, SSEP hydrodynamics and resistor networks}. Forthcoming.

\bibitem{FH} A.~Faggionato, I.~Hartarsky; \emph{Crossings and diffusion in Poisson driven marked random connection models}. Preprint, arXiv:2507.03965 (2025).


%%%%%%%%%%%

  
\bibitem{G} J.~Gallier.  \emph{Geometric Methods and Applications. For Computer Science and Engineering}. Text in Applied Mathematics {\bf 38}, Springer Verlag,  New York, 2001.


   

%\bibitem{H}  M.~Heida. Convergences of the squareroot approximation scheme to the Fokker-Planck operator. \textit{Math. Models Methods Appl. Sci.} \textbf{28}  (2018) 2599--2635.



%\bibitem{HT} A.~Hraivoronska, O.~Tse; \emph{Diffusive limit of random walks on tessellations via generalized gradient flows}.  SIAM Journal on Mathematical Analysis {\bf 55}, 2948--2995 (2023).



\bibitem{HW} D.~Hug, W.~Weil; \emph{Lectures on convex geometry}. Graduate Texts in Mathematics {\bf 286}, Springer International Publishing, 2020.

\bibitem{J} S.~Jansen; \emph{Gibbsian point processes}. Lecture notes. Available online.
\url{https://www.mathematik.uni-muenchen.de/~jansen/gibbspp.pdf} 


\bibitem{MR}  R. Meester, R. Roy;  \emph{Continuum percolation}. Cambridge Tracts in Mathematics {\bf 119}. First edition, Cambridge University Press, Cambridge, 1996.


\bibitem{M} J.~M\so ller; 
\emph{Lectures on Random Voronoi Tessellations}.  Lecture Notes in Statistics {\bf 87},  Springer Verlag, New  York, 1994.

%\bibitem{O} J.~O'Rourke; \emph{Computational Geometry in C}. Cambridge University Press, Cambridge, 1998.

%\bibitem{R} J.-J.~Risler; \emph{Mathematical Methods for CAD}.  Cambridge   University Press, Cambridge, 1993.

\bibitem{Ro0} 
A.~Rousselle; \emph{Annealed invariance principle for random walks on random graphs generated by point processes in $\bbR^d$}. Preprint arXiv:1506.05638 (2015).


\bibitem{Ro} A.~Rousselle; \emph{Quenched invariance principle for random walks on Delaunay triangulations}. Electron. J. Probab. {\bf 20}, 1--32 (2015).


\bibitem{S} A.~Soshnikov; \emph{Determinantal random point fields}. Russian Math. Surveys {\bf 55}, 923--975 (2000).


\bibitem{Ze} H.~Zessin; \emph{Point processes in general position}.   Izv. Nats. Akad. Nauk Armenii Mat. {\bf 43}, 81--88 (2008); reprinted in J. Contemp. Math. Anal. 43, 59--65 (2008).

\bibitem{Zu} S.A.~Zuyev; \emph{Estimates for distributions of the Voronoi polygon's geometric characteristics}. Random Structures and Algorithms {\bf 3}, no. 2, 149-162 (1992).
 

\end{thebibliography}
\end{document}